\documentclass[a4paper,11pt,leqno]{article}
\usepackage{amssymb}
\usepackage{bbm}
\title{\bf Cyclic Homology of Strong Smash Product Algebras}
\author{Jiao ZHANG and Naihong HU }
\date{}
\usepackage{mathrsfs}
\usepackage{amsmath, amstext, amsbsy, amssymb,latexsym,bm,amsthm,amsfonts,amscd,dsfont}
\usepackage{graphicx}
\usepackage{epsfig}
\usepackage[latin1]{inputenc}
\usepackage[all,knot,poly]{xy}
\usepackage{fancyhdr}
\usepackage[T1]{fontenc}

\newcommand{\Tor}{\mathrm{Tor}}
\newcommand{\Tot}{\mathrm{Tot}}
\newcommand{\Hh}{\mathrm{HH}}
\newcommand{\Hc}{\mathrm{HC}}

\newcommand{\Hom}{\mathrm{Hom}}
\newcommand{\Ho}{\mathrm{H}}
\def\P{\mathcal {P}}
\newcommand{\adots}{\mathinner{\mkern2mu%
\raisebox{0.1em}{.}\mkern2mu\raisebox{0.4em}{.}%
\mkern2mu\raisebox{0.7em}{.}\mkern1mu}}

\def\ot{\otimes}
\def\ep{\varepsilon}
\def\ra{\rightarrow}
\def\de{\delta}
\def\sg{\sigma}
\def\pa{\partial}
\def\De{\Delta}

\newtheorem{theorem}{Theorem}[section]
\newtheorem{lemma}[theorem]{Lemma}

\newtheorem{prop}[theorem]{Proposition}
\newtheorem{coro}[theorem]{Corollary}

\theoremstyle{definition}
\newtheorem{defi}[theorem]{Definition}
\newtheorem{example}[theorem]{Example}
\newtheorem{remark}[theorem]{Remark}

\allowdisplaybreaks
\begin{document}

\maketitle
\begin{abstract}
For any strong smash product algebra $A\#_{_R}B$ of two algebras $A$
and $B$ with a bijective morphism $R$ mapping from $B\ot A$ to $A\ot
B$, we construct a cylindrical module $A\natural B$ whose diagonal
cyclic module $\Delta_{\bullet}(A\natural B)$ is graphically proven
to be isomorphic to
 $C_{\bullet}(A\#_{_R}B)$ the cyclic module of the algebra.
 A spectral sequence is established to converge to the cyclic homology
of $A\#_{_R}B$. Examples are provided to show how our results work.
Particularly, the cyclic homology of the Pareigis' Hopf algebra is
obtained in the way.
\end{abstract}

{\small\noindent{\textbf{Keywords}}: Cyclic homology, strong smash product algebra. }

\section*{Introduction}
Calculating cyclic homology of the crossed product algebra is an
attractive problem studied extensively in cyclic homology theory.
When $G$ is a
 discrete group or a compact Lie
 group and $A$ is an algebra or a  $C^{\infty}$ manifold acted by $G$, the cyclic homology of the
 crossed product algebra $A\rtimes G$ is considered by  B.L. Feigin and
B.L. Tsygan \cite{FT}, J.L. Brylinski \cite{Br}, V. Nistor \cite{N},
and E. Getzler and J.D.S. Jones \cite{GJ}. When $A$ is an $H$-module
algebra, where $H$ is a Hopf algebra with an invertible antipode, R.
Akbarpour and M. Khalkhali \cite{AK} investigated the cyclic
homology of the crossed product algebra $A\rtimes H$. Their results
generalize the work of Getzler and Jones in \cite{GJ}.

In recent decades there have appeared many kinds of products of
different types of algebras in the research of Hopf algebras, for
instance, the crossed product, or called (classical) smash product,
of a Hopf algebra and its module algebra, Takeuchi's smash product
\cite{T1} of a left comodule algebra and a left module algebra where
the action and the coaction are taken over one Hopf algebra, the
tensor product of two algebras in the natural sense or in a braided
tensor category, the (generalized) Drinfeld double, double
crossproduct of Hopf algebras,  etc. These concepts are closely
related with the factorization of an algebra into two subalgebras.
The algebra factorization is described by S. Majid \cite{M}, the
generalized factorization problem is stated by S. Caenepeel et al.
\cite{CBMZ}. A `generalized braiding', which is quasitriangular and
normal, is associated closely with the algebra factorization. When
it is a bijection, we call the product algebra a \textit{strong
smash product algebra}.

In this paper, we generalize both works of Getzler and Jones
\cite{GJ}, and of Akbarpour and Khalkhali \cite{AK} to strong smash
product algebras.
 Indeed, the crossed product algebras discussed in \cite{GJ} and
 \cite{AK} are special examples of the strong smash product algebras.
 We organize this paper as follows. In Section \ref{111}, we give the
explicit definition of the strong smash product algebra $A\#_{_R}B$.
In Section \ref{222}, we construct a cylindrical module $A\natural
B$. Using diagrammatical presentations we prove that
$\Delta_{\bullet}(A\natural B)$ the cyclic module related to the
diagonal of $A\natural B$ is isomorphic to the cyclic module of
$A\#_{_R}B$. In Section \ref{33}, we recall some notations and apply
the generalized Eilenberg-Zilber theorem for cylindrical modules due
to Getzler and Jones \cite{GJ}. In Section \ref{44}, we construct a
spectral sequence converging to the cyclic homology of $A\#_{_R}B$.
In Section \ref{55}, we apply our theorems to Majid's double
crossproduct of Hopf algebras after showing that they pertain to the
class of strong smash product algebras. As any Drinfeld's quantum
double has a double crossproduct structure (see \cite{M}), the
notion of strong smash product algebras does cover a wild range of
the recent interesting examples, for instance, the two-parameter or
multiparameter quantum groups, and the pointed Hopf algebras arising
from Nichols algebras of diagonal type (see \cite{ARS, AS, BW, BGH,
H, HP, HRZ, HW, PHR} and references therein). Besides these, another
concrete example for the computation of the cyclic homology of the
Pareigis' Hopf algebra $\mathcal P$ is given to illuminate our
results.

We assume that $k$ is a field containing $\mathbb{Q}$ in the whole
paper unless otherwise stated. Every algebra in this paper is
assumed to be a unital associative $k$-algebra.

\section{\label{111}Strong smash product algebra}

 Majid defined in his book \cite{M} an algebra
factorization. A unital and associative algebra $X$
\textit{factorizes} through its subalgebras $A$ and $B$, if the
product map defines a linear isomorphism $A\ot B\cong X$. The
necessary and sufficient conditions for the existence of an algebra
factorization is the existence of a linear map $R$ from $B\ot A$ to
$A\ot B$, which is quasitriangular and normal. In \cite{CBMZ}, the
algebra which can be factorized is called a smash product and
denoted by $A \#_{_R} B$. In addition, if $R$ is also an isomorphism
of vector spaces, we call $A\#_{_R} B$ a {\it strong} smash product
algebra. The explicit definitions are as follows:

\begin{defi} Let $A$ and $B$ be two algebras, and $R: B\ot A\ra A\ot
B$ be a linear map. $R$ is called \textit{quasitriangular} if it
obeys
\begin{gather*}
R\circ (m\ot id)=(id\ot m)R_{12}R_{23},\\
R\circ (id\ot m)=(m\ot id)R_{23}R_{12},
\end{gather*}
where $m$ is the product map, $R_{12}=R\ot id$ and $R_{23}=id\ot R$.

$R$ is called \textit{normal} if it obeys
\begin{gather*}
R(1\ot a)=a\ot 1, \quad\forall~ a\in A, \\
R(b\ot 1)=1\ot b, \quad\forall~ b\in B.
\end{gather*}
 The \textit{smash product
algebra} of $A$ and $B$ with a quasitriangular and normal $R$,
denoted by $A\#_{_R}B$, is defined to be $A\ot B$ as a vector space
equipped with product
$$(a\ot b)(a'\ot b')=aR(b\ot a')b', \quad \forall~a, a'\in A,~b, b'\in B.$$
 \end{defi}
 The smash product algebra
 $A\#_{_R}B$ defined above
 is a unital associative algebra with the unit $1_A\ot 1_B$.
 The product of $A \#_{_R} B$ appeared first in
\cite{vv}, where a sufficient condition is given for the product to
be associative.

\begin{defi} The smash product algebra $A\#_{_R}B$ is said to be
\textit{strong},
if $R$ is invertible.
\end{defi}

\begin{prop} \label{1.2} $A\#_{_R}B$ is a strong smash product algebra if and
only if $B\#_{_{R^{-1}}}A$ is a strong smash product algebra.

Indeed,
 $R$ is
quasitriangular (resp. normal) if and only if $R^{-1}$ is
quasitriangular (resp. normal).
\end{prop}
\begin{proof} Since $R$ is invertible with $R^{-1}:A\ot B\ra B\ot A$,
we have
\begin{gather*}
R(m_B\ot id)=(id\ot m_B)R_{12}R_{23}, \Leftrightarrow (m_B\ot
id)(R_{12}R_{23})^{-1}=R^{-1}(id\ot m_B),\\
\Leftrightarrow R^{-1}(id\ot m_B)=(m_B\ot
id)R_{23}^{-1}R_{12}^{-1},\\
R(id \ot m_A)=( m_A\ot id)R_{23}R_{12}, \Leftrightarrow (id \ot
m_A)(R_{23}R_{12})^{-1}=R^{-1}( m_A\ot id),\\
\Leftrightarrow R^{-1}( m_A\ot id)=(id \ot
m_A)R_{12}^{-1}R_{23}^{-1}.
\end{gather*}
 The normalization conditions are clear.
\end{proof}

It proves very convenient to do computations using diagrammatical
presentations. Present the multiplication of an algebra by
$\xygraph{!{0;/r1.0pc/:}[u(0.5)]!{\xcaph-@(1.5)}[l(0.5)][d(0.75)]!{\xcapv[0.5]@(0)}}$
and  $R$ from $B\ot A$ to $A\ot B$ by
$\xygraph{!{0;/r1.0pc/:}[u(0.5)]!{\xoverv}[d(0.5)][l(0.3)]{
A}[r(1.5)]B[uu][r(0.3)]A[l(1.5)]B}$. Thus $R^{-1}$ can be presented
by $\xygraph{!{0;/r1.0pc/:}[u(0.5)]!{\xunderv}[d(0.5)][l(0.3)]{
B}[r(1.5)]A[uu][r(0.3)]B[l(1.5)]A}$. The quasitriangular conditions
can be diagrammatically expressed as follows:
$$\xymatrix{\includegraphics[scale=0.5]{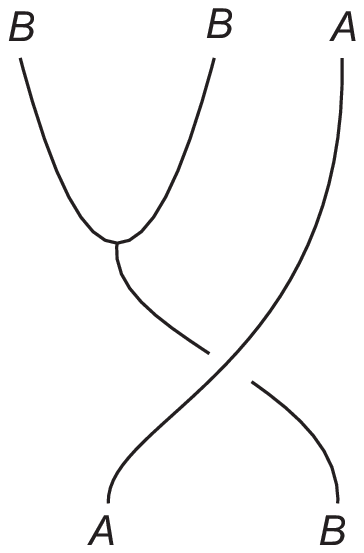}\ar@{}[r]|-{\cong}&\includegraphics[scale=0.5]{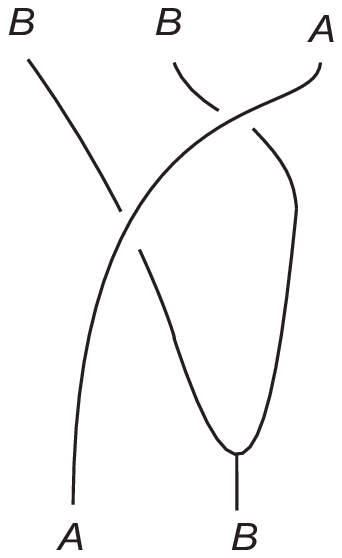}}$$
$$\xymatrix{\includegraphics[scale=0.5]{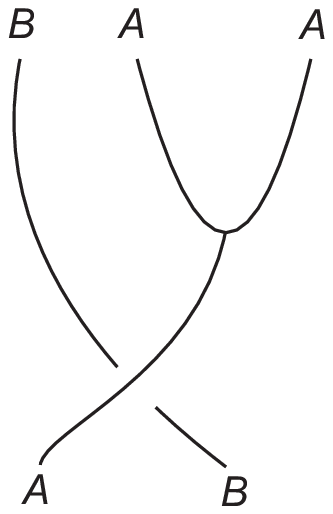}\ar@{}[r]|-{\cong}&\includegraphics[scale=0.5]{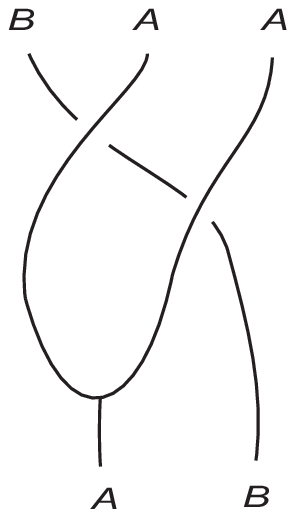}}$$

The concept of smash product algebra $A\#_{_R}B$ recovers the
crossed product algebra (or called classical smash product algebra)
$A\rtimes H$ and Takeuchi's smash product algebra $A\# B$ (defined
in \cite{T1}) where $H$ is a Hopf algebra, $A$ is an $H$-module
algebra and $B$ is an $H$-comodule algebra. The two subalgebras play
different roles in $A\rtimes H$ and  $A\# B$. One algebra produces
action on the other. However, in the strong smash product algebra
$A\#_{_R}B$, the status of $A$ and $B$ is equal. They act on each
other. The strong smash product algebra $A\#_{_R}B$ is a more
natural concept, as in physics the general principle is that every
action has a `reaction'.

Many smash product algebras are strong smash product algebras.

\begin{example} The tensor
product of two algebras in a braided tensor category is a strong
smash product algebra. Here $R$ is deduced directly from the
braiding in that category, so $R$ is invertible.
\end{example}

\begin{example}\label{eg ch1.5}
 Let $H$ be a Hopf algebra with an invertible antipode $S$. $A$ is a left
 $H$-module algebra and $B$ is a left $H$-comodule algebra.
  Takeuchi's  smash product $A\#B$ is an algebra
  with the multiplication $(a\#b)(a'\#b')=a(b_{[-1]}.a')\#b_{[0]}b'$ and the unit
  $1_A\ot 1_B$, where $b\mapsto b_{[-1]}\ot b_{[0]}$ is the left
  $H$-comodule structure map for $a, a'\in A$ and $b, b'\in B$. When $B=H$, $A\#B=A\rtimes H$ is
   the  crossed product algebra. Define $R:B\ot A\ra A\ot B$ by
$$
R(b\ot a)=b_{[-1]}.a\ot b_{[0]}.
$$
One can check that $R$ is quasitriangular and normal
through the definition of the module algebra and the comodule
algebra. $R$ has the inverse defined by $$R^{-1}(a\ot b)=b_{[0]}\ot
S^{-1}(b_{[-1]}).a,$$ for all $a\in A$ and $b\in B$.

Hence, the crossed product algebras discussed in \cite{GJ} and
\cite{AK} are special examples of our strong smash product algebras.
\end{example}

\section{\label{222}Paracyclic modules and cylindrical modules}
\noindent \textbf{2.1} \quad From Getzler and Jones' point of view,
all the operators of a cyclic module can be generated by only two
operators, i.e., the last face map and the extra degeneracy map.
Hence we can give an equivalent definition for cyclic modules. In
this subsection, $k$ can be a commutative ring.
\begin{defi}
 A \textit{cyclic module} is a sequence of $k$-modules $\{C_n\}_{n\geq
0}$ which is endowed for each $n$ with two $k$-linear maps
$d_n:C_n\ra C_{n-1}$ and $s_{-1}:C_{n}\ra C_{n+1}$, such that $d_n
s_{-1}$ is invertible, and by setting
\begin{align}
&t_n:=d_n s_{-1}:C_n\ra C_n,\label{1}\\
\begin{split}\label{2}
&d_i:=t_{n-1}^{-(n-i)}d_{n}t^{n-i}_{n}:C_n\ra C_{n-1},\quad \text{for}~0\leq i\leq n,\\
&s_i:=t^{i+1}_{n+1}s_{-1}t^{-(i+1)}_{n}:C_n\ra C_{n+1},\quad
\text{for}~0\leq i\leq n,
\end{split}
\end{align}
for any $i, j\in\mathbb{N}$, the following relations hold
\begin{align}
\begin{split}\label{3}
&d_id_j=d_{j-1}d_i \quad\text{for}~ i<j,\\
&s_is_j=s_{j+1}s_i \quad\text{for}~ i\leq j,\\
&d_is_j=\begin{cases} s_{j-1}d_i \quad\text{for}~ i<j,\\
id  \quad\text{for}~ i=j,i=j+1,\\
s_jd_{i-1} \quad\text{for}~ i>j+1.
\end{cases}
\end{split}\\
&t_{n}^{n+1}=id.\label{4}
\end{align}
$d_i$'s are called  \textit{face} maps and $s_i$'s are called
\textit{degeneracy} maps for $i\geq 0$, $t$ is called the
\textit{cyclic operator}. $s_{-1}$ is called the \textit{extra
degeneracy map} and $d_n:C_n\ra C_{n-1}$ is called the \textit{last
face} map for $C_n$.
\end{defi}

Therefore, a cyclic module can be regarded as an underlying
simplicial module $\{C_n\}_{n\geq 0}$, whose face maps, degeneracy
maps and cyclic operators are generated by the last face map $d_n$
and the extra degeneracy map $s_{-1}$ for each $C_n$ in the way
expressed in  \eqref{1} and \eqref{2} satisfying \eqref{3} and
\eqref{4}.

If the condition \eqref{4} is replaced by
\begin{equation}\label{5}
d_0t_n=d_n,\quad s_0t_n=t_{n+1}^2s_n,
\end{equation}
then that sequence of $k$-modules is called a \textit{paracyclic
module}. In fact, the equalities in \eqref{5} are  consequences of
the cyclicity of the invertible operator $t$, that is, from
\eqref{4}, one can get \eqref{5}.

For all $n$, set the following operators
\begin{equation}\label{6}\begin{split}
&\mathrm{b}=\sum_{i=0}^{n}(-1)^id_i:C_n\ra C_{n-1},\\
&\mathrm{T}=t^{n+1}:C_n\ra C_n,\\
&\mathrm{N}=\sum_{i=0}^{n}(-1)^{in}t^{i}:C_n\ra C_n,\\
&\mathrm{B}=(1+(-1)^{n}t)s_{-1}\mathrm{N}: C_{n}\ra C_{n+1}.
\end{split}
\end{equation}

\begin{lemma}[\cite{GJ}]\label{2.2} We have the equalities
$$\mathrm{bT=Tb}, ~\mathrm{bB+Bb=1-T}.$$
\end{lemma}

Getzler and Jones first introduced in \cite{GJ} the concepts of the
bi-paracyclic module and the cylindrical module. We recall their
definitions here.

\begin{defi}[\cite{GJ},\cite{AK}]
A \textit{bi-paracyclic module} is a sequence of $k$-modules
$$(\{C_{m,n}\}_{m,n\geq 0},d_i^{m,n},s_i^{m,n},t_{m,n},\bar{d}_j^{m,n},\bar{s}_j^{m,n},\bar{t}_{m,n})$$
such that $(\{C_{m,n}\}_{m,n\geq 0},d_i^{m,n},s_i^{m,n},t_{m,n})$
and $(\{C_{m,n}\}_{m,n\geq
0},\bar{d}_j^{m,n},\bar{s}_j^{m,n},\bar{t}_{m,n})$ are two
paracyclic modules and the operators $d_i^{m,n},s_i^{m,n},t_{m,n}$
commute with the operators
$\bar{d}_j^{m,n},\bar{s}_j^{m,n},\bar{t}_{m,n}$. Moreover, if in
addition, $t_{m,n}^{m+1}\bar{t}_{m,n}^{n+1}=id_{m,n}$ for all
$m,n\geq 0$, then this bi-paracyclic module is called a
\textit{cylindrical module}.
\end{defi}

Another interesting concept named parachain complex was also given
by Getzler and Jones  \cite{GJ}. The mixed complex defined by Kassel
 \cite{K} is a special case of parachain complexes. Here we need
only the mixed complex. The \textit{mixed complex} is, by
definition, a graded $k$-module $(M_n)_{n\in\mathbb{N}}$ endowed
with two graded commutative differentials, one decreasing the degree
and the other increasing the degree. That is,
$(M_{\bullet},\mathrm{b},\mathrm{B})$ with $\mathrm{b}:M_n\ra
M_{n-1}$ and $\mathrm{B}:M_n\ra M_{n+1}$ satisfies
$\mathrm{b}^2=\mathrm{B}^2=\mathrm{bB}+\mathrm{Bb}=0$.  A morphism
of mixed complexes $(M_{\bullet},\mathrm{b},\mathrm{B})$ to
$(M_{\bullet}',\mathrm{b},\mathrm{B})$ is a sequence of morphisms
$f_k:M_n\ra M'_{n+2k}$ for $k\geq 0$ such that $f=\sum_{k\geq 0}
u^kf_k$ commutes with $\mathrm{b}+u\mathrm{B}$. For a cyclic module
$(C_{\bullet},d_i,s_i,t)$ associated with the
 operators $\mathrm{b}$ and $\mathrm{B}$ defined in
\eqref{6}, $(C_{\bullet},\mathrm{b},\mathrm{B})$ is a mixed complex.

It is usually simpler to consider the complex with one differential
than to consider the mixed complex with two differentials. Actually,
a mixed complex can be converted into a complex. Let $V_{\bullet}$
be a non-negative graded $k$-module. Denote by $V_{\bullet}[[u]]$
the graded $k$-modules of formal power series in a variable $u$ with
coefficients in $V_{\bullet}$. Set the degree of $u$ be $-2$. If
$V_{\bullet}$ is endowed with a degree $-1$ endomorphism
$\mathrm{b}$ and a degree $1$ endomorphism $\mathrm{B}$, then
 $(V_{\bullet},\mathrm{b},\mathrm{B})$ is a mixed complex if and only if
 $(V_{\bullet}[[u]],\mathrm{b}+u\mathrm{B})$ is a complex with the differential
 $\mathrm{b}+u\mathrm{B}$. Here set
$V_{n}[[u]]=\sum_{i\geq 0}V_{n+2i}u^i$.\\

\noindent \textbf{2.2} \quad Now we return to our  strong smash
product algebra.

 The cyclic module $C_{\bullet}(A\#_{_R}B)$ of an algebra $A\#_{_R}B$ is defined
 as usual (see \cite{L} etc). That is, $C_n(A\#_{_R}B)=(A\#_{_R}B)^{\ot
 (n+1)}$ for all $n\in\mathbb{N}$ with \begin{align*}
d_i(x_0,\ldots,x_n)&=(x_0,\ldots,x_ix_{i+1},\ldots,x_n),\quad 0\leq i<n,\\
d_n(x_0,\ldots,x_n)&=(x_nx_0,\ldots,x_{n-1}),\\
t(x_0,\ldots,x_n)&=(x_n,x_0,\ldots,x_{n-1}),\\
s_j(x_0,\ldots,x_n)&=(x_0,\ldots,x_j,1,x_{j+1},\ldots,x_n), \quad
0\leq j\leq n,
 \end{align*}
where $x_0,\ldots,x_n\in A\#_{_R}B$.

For $A$ and $B$ the subalgebras of $A\#_{_{R}}B$, we introduce a
cylindrical module denoted by $A\natural B$ which generalizes the
cylindrical module constructed in the paper \cite{GJ} by Getzler and
Jones  where $B$ is a group algebra and $A$ is a $B$-module algebra,
also generalizes the cylindrical module constructed in the paper
\cite{AK} by  Akbarpour  and Khalkhali  where $B$ is a Hopf algebra
with an invertible antipode and $A$ is a $B$-module algebra.

For $p, q\in \mathbb{N}$, set $A\natural B(p,q)=B^{\ot (p+1)}\ot
A^{\ot (q+1)}$ endowed with the following operators which are mainly
defined on $B$'s side:
\begin{align}\label{7}\begin{split}
t_{p,q}(b_0,\ldots,b_p\,|\,a_0,\ldots,a_q)
&=\mathrm{f}^{p+q+1,1}(b_0,\ldots,b_{p-1},\Theta_q(b_p,a_0,\ldots, a_q)),\\
d_i^{p,q}(b_0,\ldots,b_p\,|\,a_0,\ldots,a_q)
&=(b_0,\ldots,b_ib_{i+1},\ldots,b_{p}\,|\,a_0,\ldots,a_q),\quad 0\leq i<p,\\
s_i^{p,q}(b_0,\ldots,b_p\,|\,a_0,\ldots,a_q)
&=(b_0,\ldots,b_i,1,b_{i+1},\ldots,b_{p}\,|\,a_0,\ldots,a_q),\quad
0\leq i\leq p;\end{split}
 \end{align}
 and the following operators which are mainly defined on $A$'s
side:
\begin{align}\label{8}\begin{split}
\bar{t}_{p,q}(b_0,\ldots,b_p\,|\,a_0,\ldots,a_q)
&=(\Gamma_p(a_q,b_0,\ldots,b_p),a_0,\ldots,a_{q-1}),\\
\bar{d}_j^{p,q}(b_0,\ldots,b_p\,|\,a_0,\ldots,a_q)
&=(b_0,\ldots,b_{p}\,|\,a_0,\ldots,a_ja_{j+1},\ldots,a_q),\quad 0\leq j<q,\\
\bar{s}_j^{p,q}(b_0,\ldots,b_p\,|\,a_0,\ldots,a_q)
&=(b_0,\ldots,b_{p}\,|\,a_0,\ldots,a_j,1,a_{j+1},\ldots,a_q),\quad
0\leq j\leq q,\end{split}
 \end{align}
where $\mathrm{f}^{\cdot,\cdot}$ is the flip map defined by
$$\mathrm{f}^{m,n}(c_1,\ldots,c_m,c_1',\ldots,c_n'):=(c_1',\ldots,c_n',c_1,\ldots,c_m),$$
 $\Theta_q$ is a composition of $R$'s defined by
$$\Theta_q:=R_{q+1,q+2}\cdots R_{23}R_{12}:B\ot A^{\ot (q+1)}\ra A^{\ot (q+1)}\ot
B,$$ and $\Gamma_p$ is a composition of $R^{-1}$'s defined by
$$\Gamma_p:=R^{-1}_{p+1,p+2}\cdots R^{-1}_{23}R^{-1}_{12}:A\ot B^{\ot(p+1)}\ra B^{\ot(p+1)}\ot A,$$
for $a_i\in A$ and $b_i\in B$. Define the last face maps by
$d_p^{p,q}=d_0^{p,q}t_{p,q}$ and
$\bar{d}_q^{p,q}=\bar{d}_0^{p,q}\bar{t}_{p,q}$.

We can simply write $t_{p,q}=\mathrm{f}^{p+q+1,1}\circ (id^{\ot
p}\ot \Theta_q)$ and $\bar{t}_{p,q}=(\Gamma_p\ot id^{\ot
q})\circ\mathrm{f}^{p+q+1,1}$. Graphically, present the flip map
between $A\ot B$ and $B\ot A$ by
$\xygraph{!{0;/r1.0pc/:}[u(0.5)]!{\xunderv@(0)}[d(0.5)]A[r]B[uu]A[l]B}$,
its inverse is
$\xygraph{!{0;/r1.0pc/:}[u(0.5)]!{\xunderv@(0)}[d(0.5)]B[r]A[uu]B[l]A}$.
The identity is denoted by~
$\xygraph{!{0;/r1.5pc/:}[u(0.8)]!{\xcapv@(0)}}$. Then $t_{p,q}$ and
$\bar{t}_{p,q}$ can be presented by
$$t_{p,q}=\raisebox{-2.5pc}{
\includegraphics[scale=0.5]{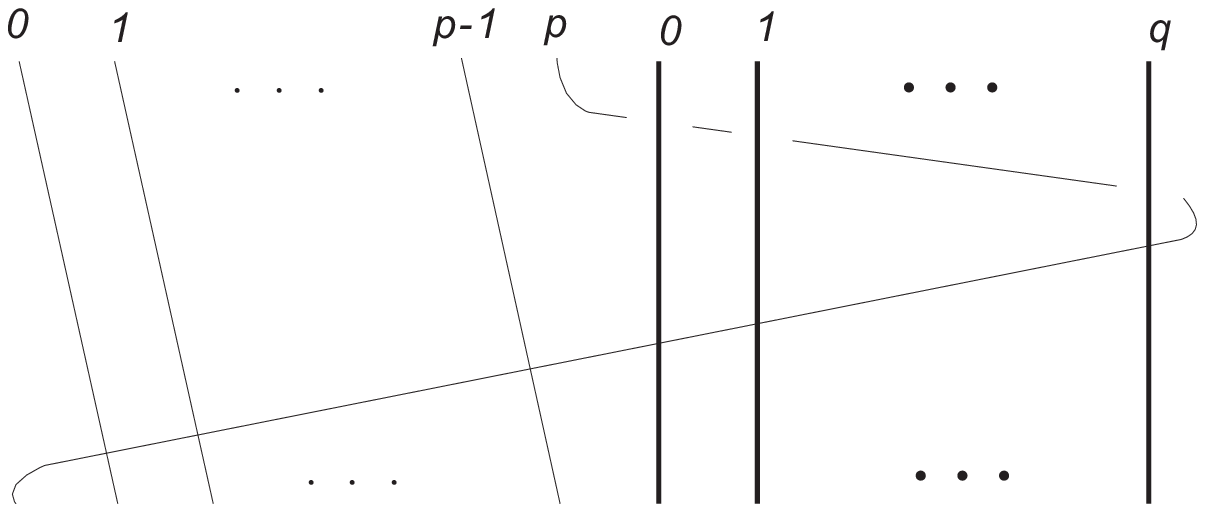}},
\quad \bar{t}_{p,q}= \raisebox{-2.5pc}{
\includegraphics[scale=0.5]{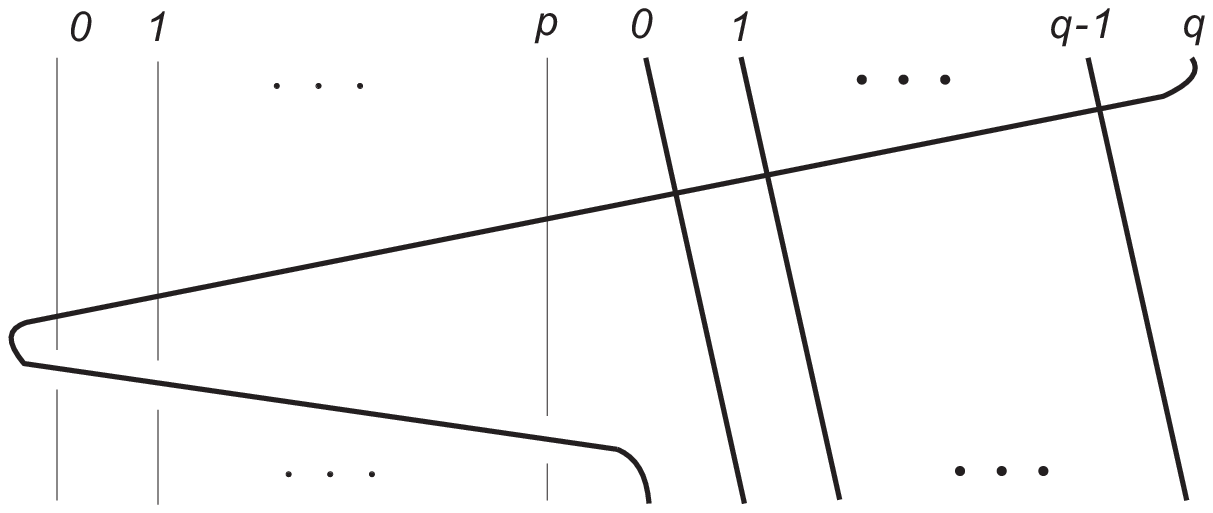}}.$$
The elements in $A$ are drawn with thick lines and the elements in
$B$ are drawn with thin lines in order to show
differences.\vspace{1em}

Since $RR^{-1}=R^{-1}R=id$, we have \begin{equation}\tag{I}
\raisebox{-1.5pc}{\includegraphics[scale=0.5]{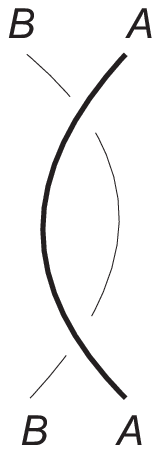}}\cong
\raisebox{-1.5pc}{\includegraphics[scale=0.5]{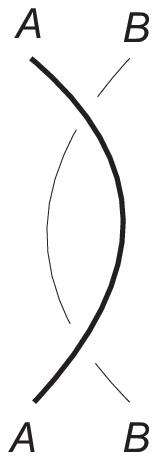}}\cong
\raisebox{-1.5pc}{\includegraphics[scale=0.5]{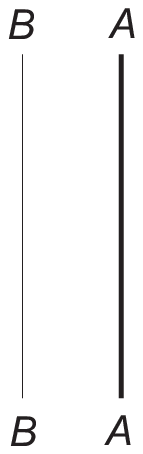}}\cong
\raisebox{-1.5pc}{\includegraphics[scale=0.5]{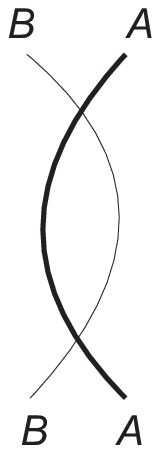}}.\end{equation}

 Although  $R$ does not satisfy the braid relations, the flip maps
 always satisfy them and are involutions. When the three crosses in
 one side of the braid relations consist of two flip maps and one $R$
 or $R^{-1}$, we still have the ``braid'' relations.

\begin{lemma}
$\mathrm{f}_{12}\mathrm{f}_{23}R_{12}
=R_{23}\mathrm{f}_{12}\mathrm{f}_{23}$,
$\mathrm{f}_{12}R_{23}\mathrm{f}_{12}
=\mathrm{f}_{23}R_{12}\mathrm{f}_{23}$,
$R_{12}\mathrm{f}_{23}\mathrm{f}_{12}
=\mathrm{f}_{23}\mathrm{f}_{12}R_{23}$, where $\mathrm{f}$ denotes
$\mathrm{f}^{1,1}$, i.e., the flip map of two elements. The
graphical notations are
\begin{equation}\tag{II}\raisebox{-2.5pc}{\includegraphics[scale=0.5]{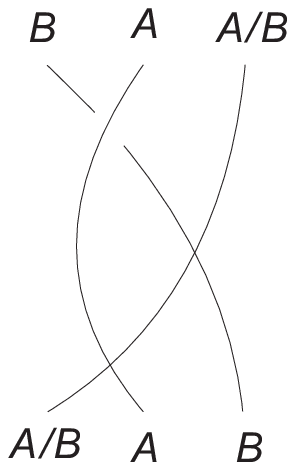}}
{\cong}
\raisebox{-2.5pc}{\includegraphics[scale=0.5]{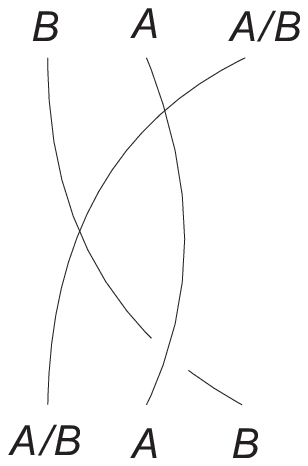}},\qquad
\raisebox{-2.5pc}{\includegraphics[scale=0.5]{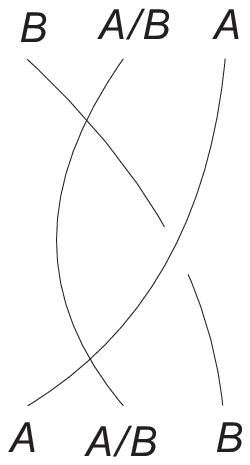}} {\cong}
\raisebox{-2.5pc}{\includegraphics[scale=0.5]{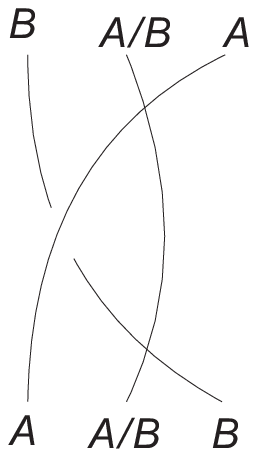}}, \qquad
\raisebox{-2.5pc}{\includegraphics[scale=0.5]{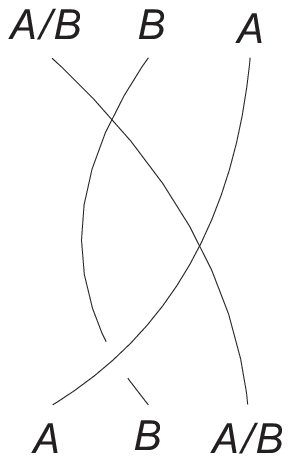}} {\cong}
\raisebox{-2.5pc}{\includegraphics[scale=0.5]{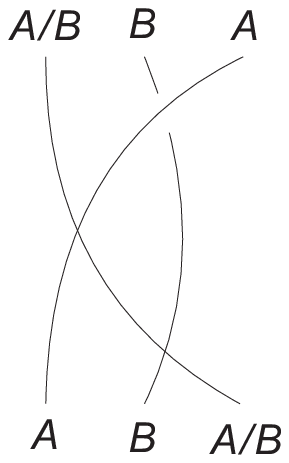}}.
\end{equation} For $R^{-1}$, we have the same relations.
\end{lemma}

\begin{prop}
$(A\natural B,d_i,s_i,t,\bar{d}_j,\bar{s}_j,\bar{t})$ is a
cylindrical module.
\end{prop}
\begin{proof}
We  check  the commutativity of the barred operators and unbarred
operators first. We would like to use the graphical proof.

$$\xymatrix{\textrm{(i)}:~~{t}_{p,q}\bar{t}_{p,q}=\raisebox{-5pc}
{\includegraphics[scale=0.5]{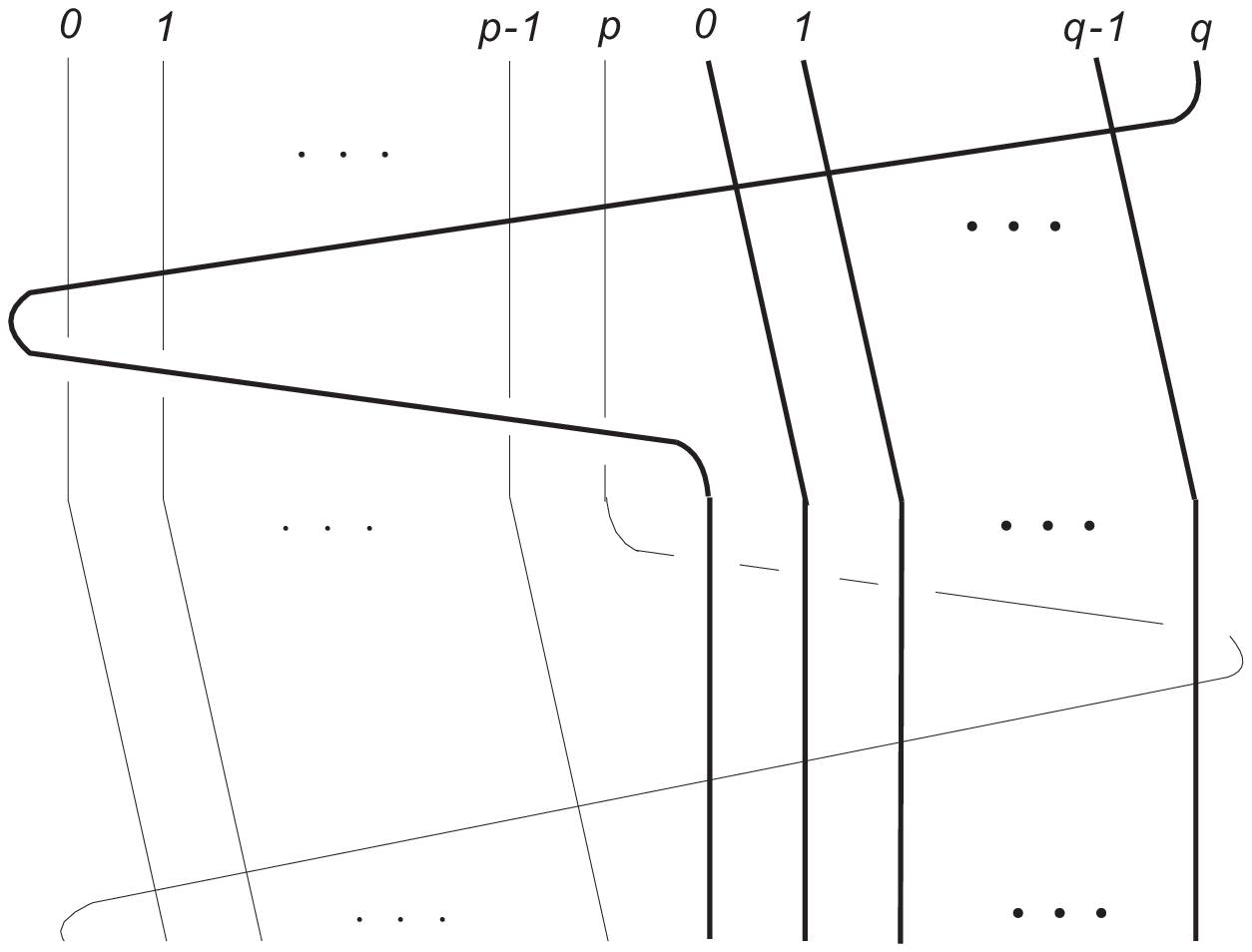}}\ar[r]^-{ \txt{act
(I)}}&}$$
$$\xymatrix{\includegraphics[scale=0.5]{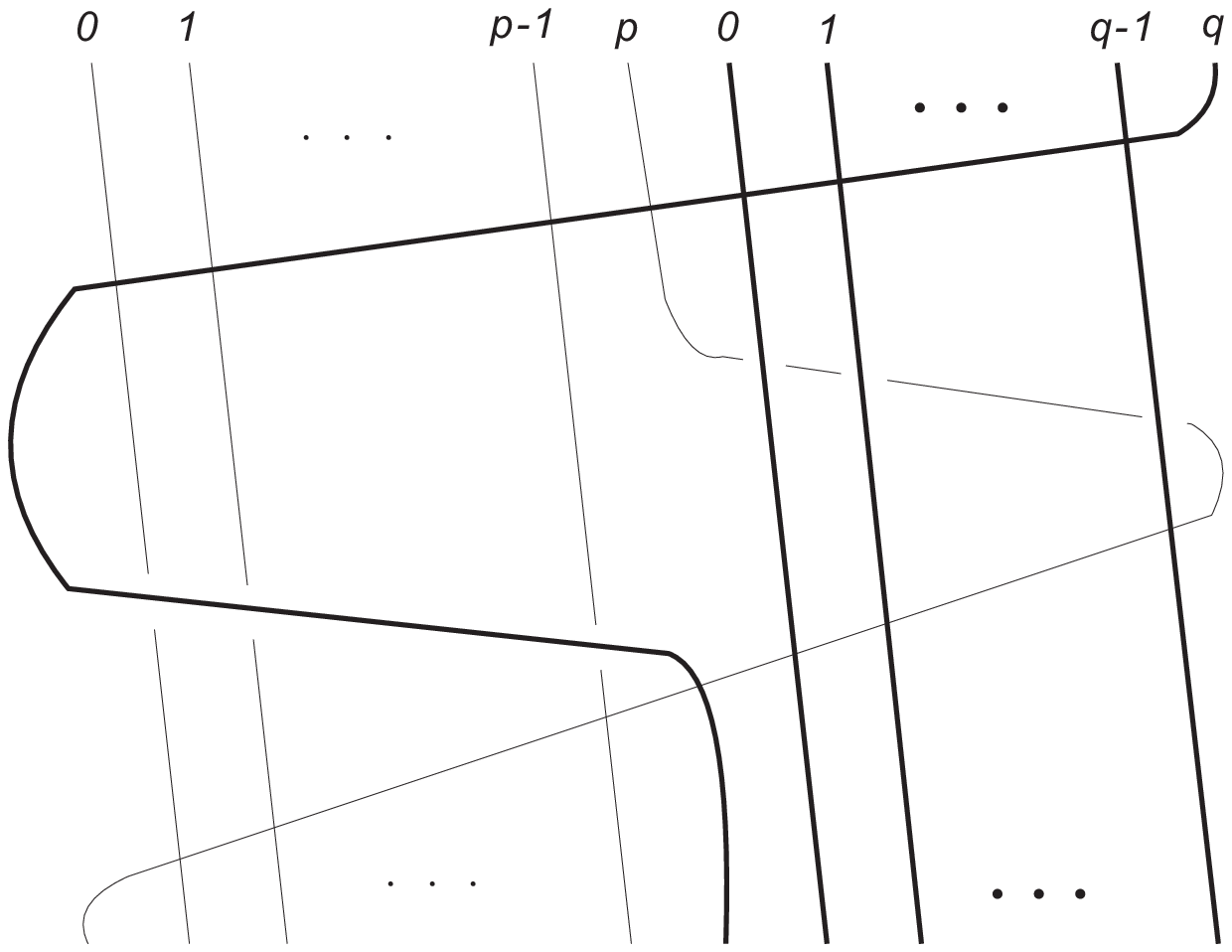}&&
\ar[rr]^-{\txt{act (II) on\\ its crosses in\\the upper right
corner}}_-{\txt{ and the\\ lower left corner}}&&}$$
$$\xymatrix{\includegraphics[scale=0.5]{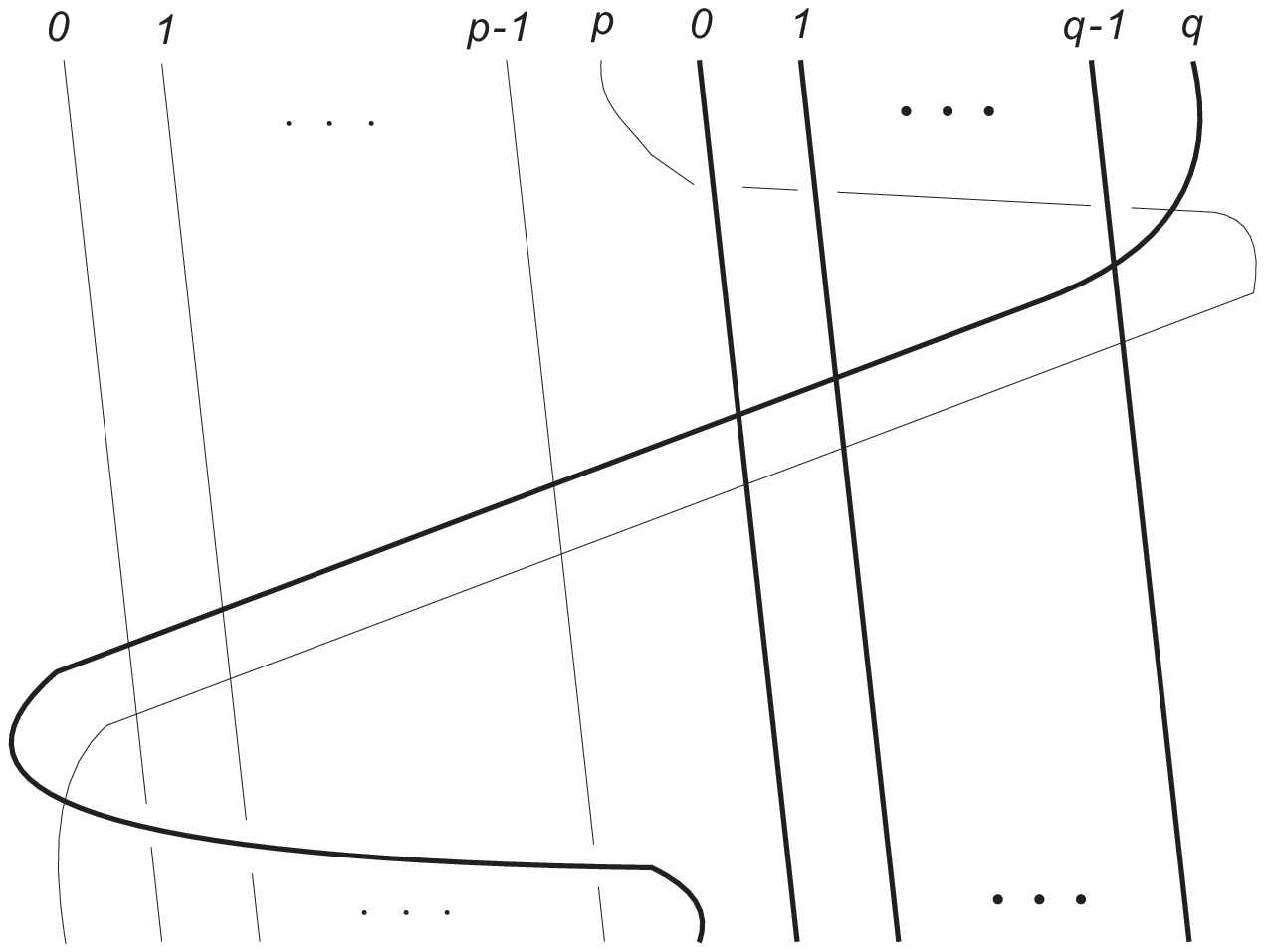}&\ar[rr]^-{\txt{the
flip maps\\ are involutions\\ and obey\\ braid relations}}&&}$$

$$\xymatrix{\includegraphics[scale=0.5]{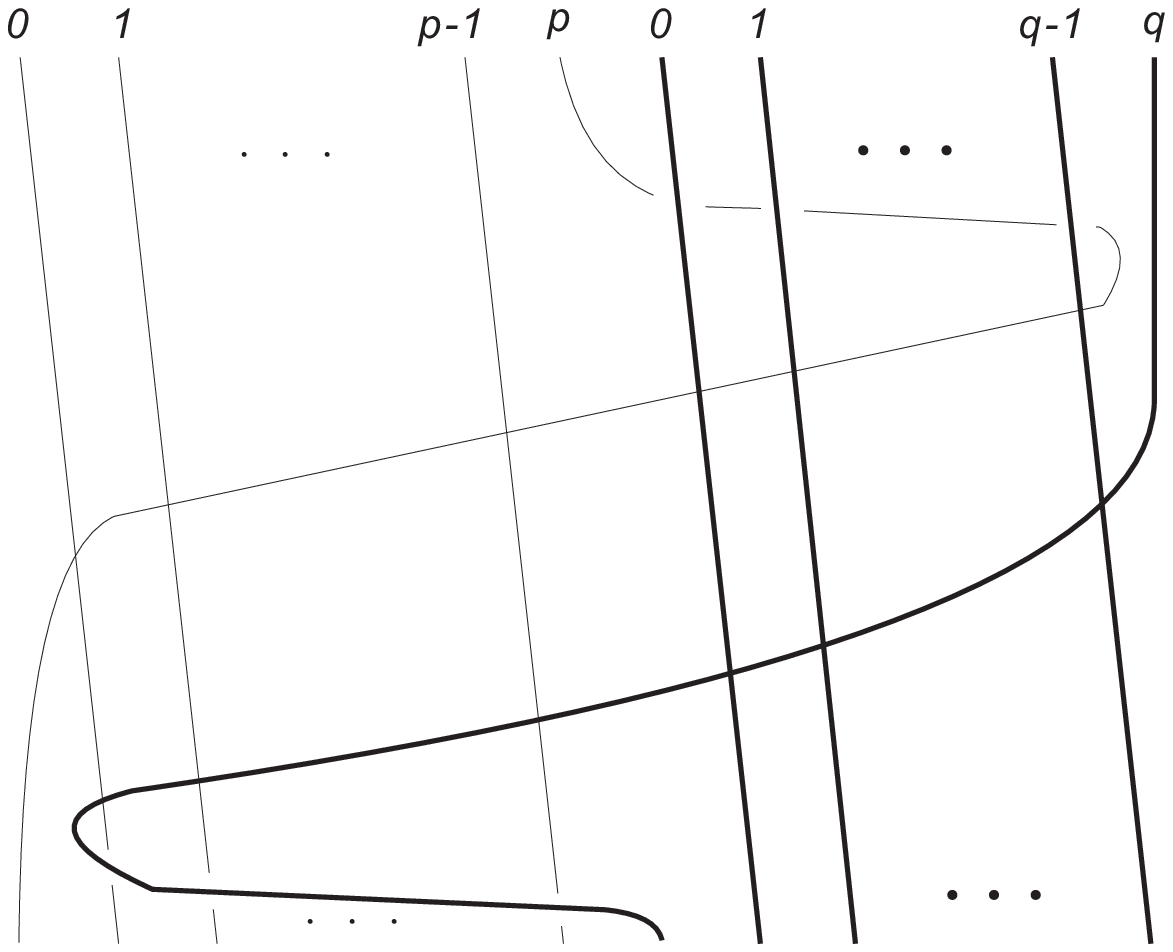}\ar[r]^-{\txt{(I)}}&
{\includegraphics[scale=0.5]{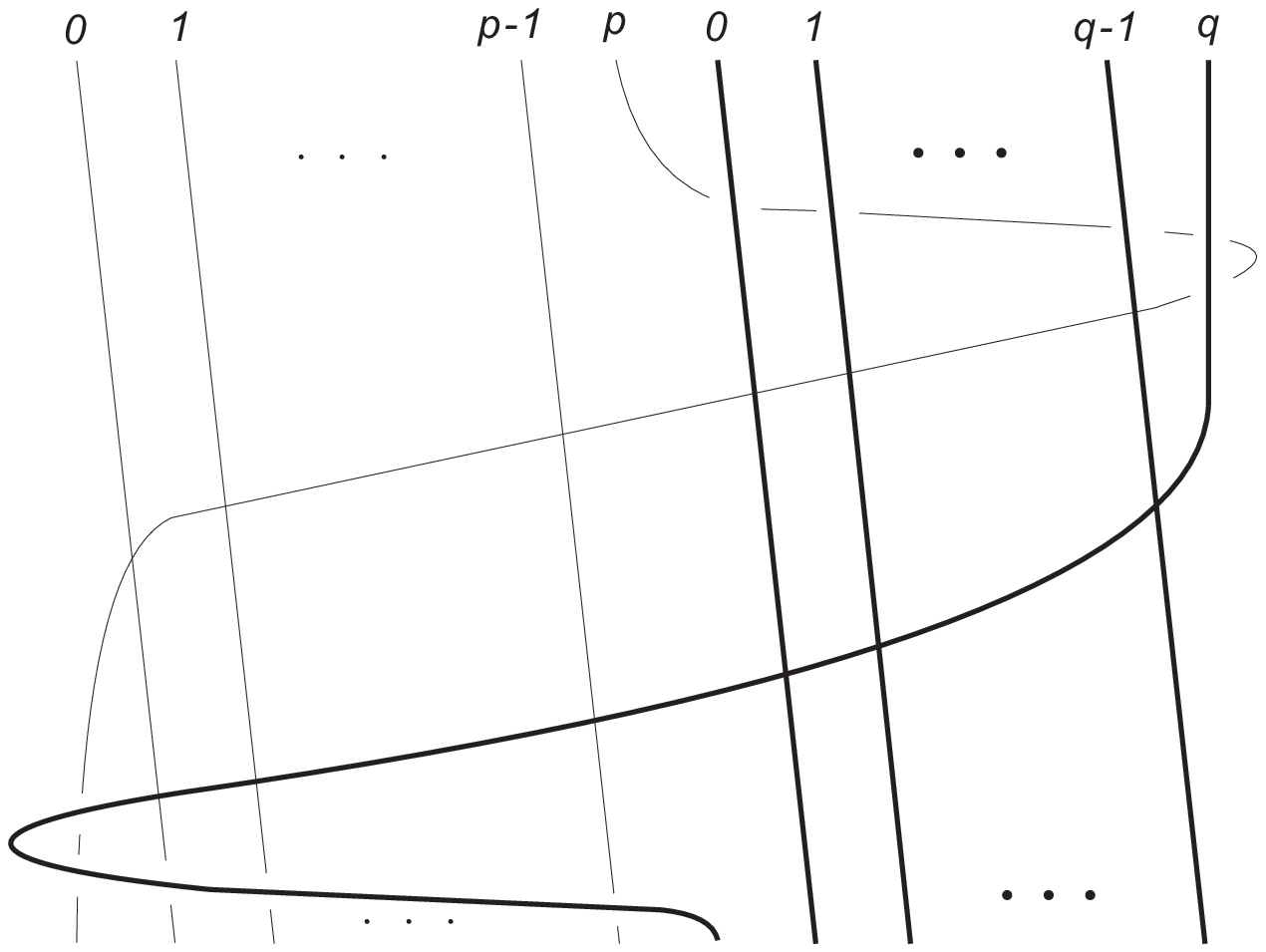}}}$$

$$ \xymatrix{&\ar[r]^-{\txt{(I)}}& \raisebox{-5pc}
{\includegraphics[scale=0.5]{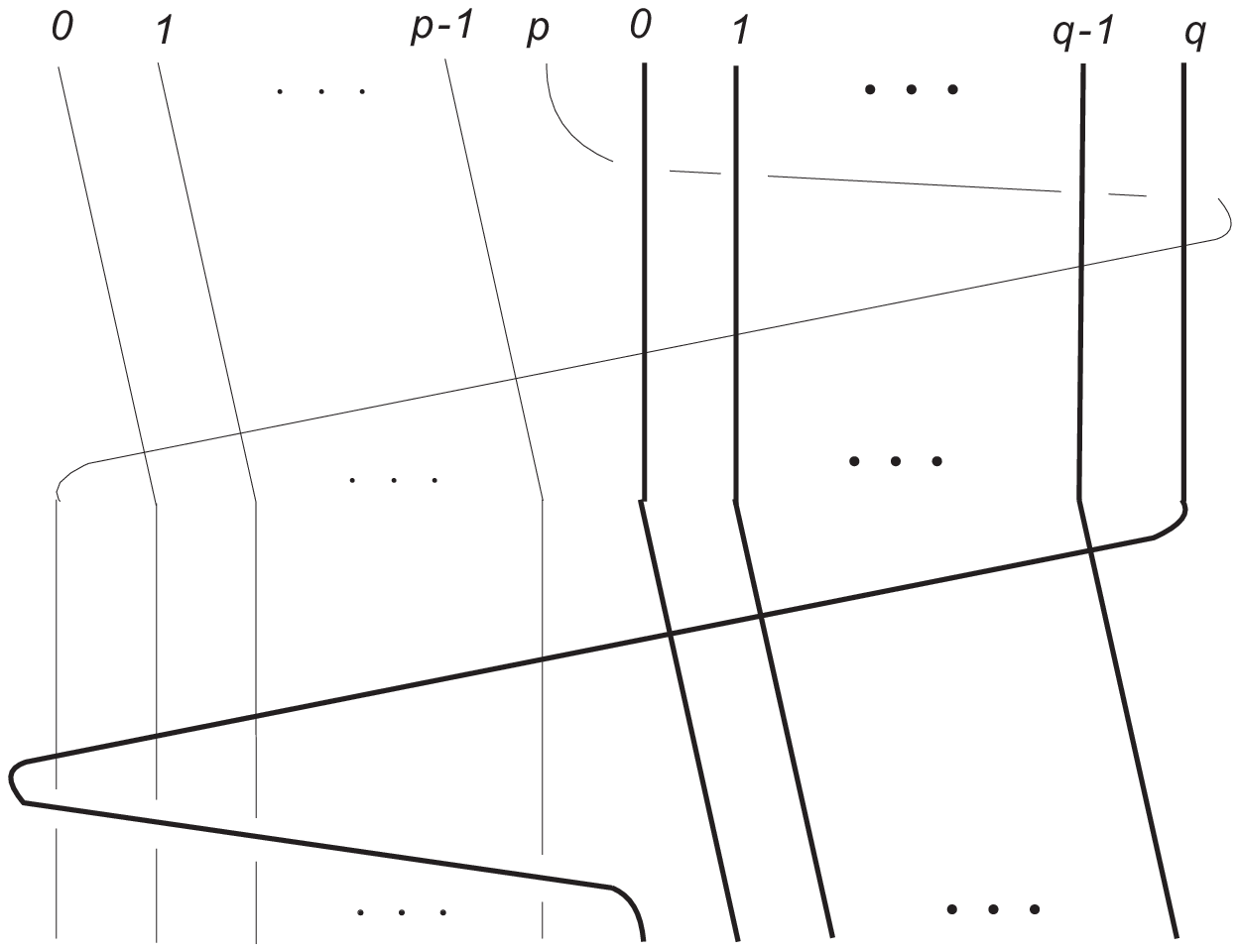}}=\bar{t}_{p,q}{t}_{p,q}}.$$

(ii) For $0\leq j<q$,
$$\xymatrix{~~\bar{d}_j^{p,q}t_{p,q}=\raisebox{-2.5pc}
{\includegraphics[scale=0.5]{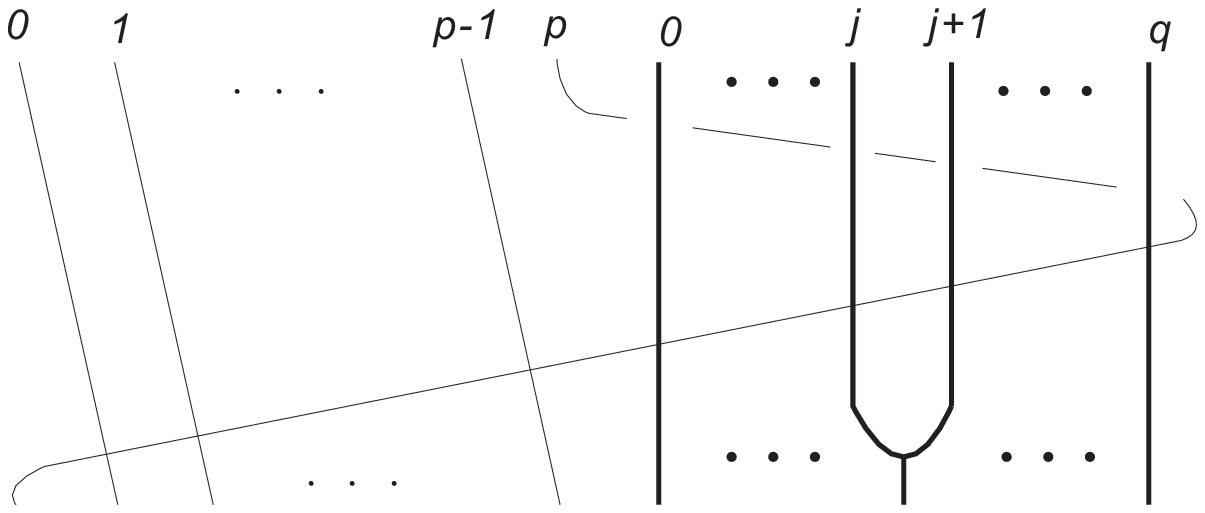}}&\ar[rr]^-{ \txt{flip map is\\
quasi-\\triangular}}&&}$$
$$\xymatrix{\includegraphics[scale=0.5]{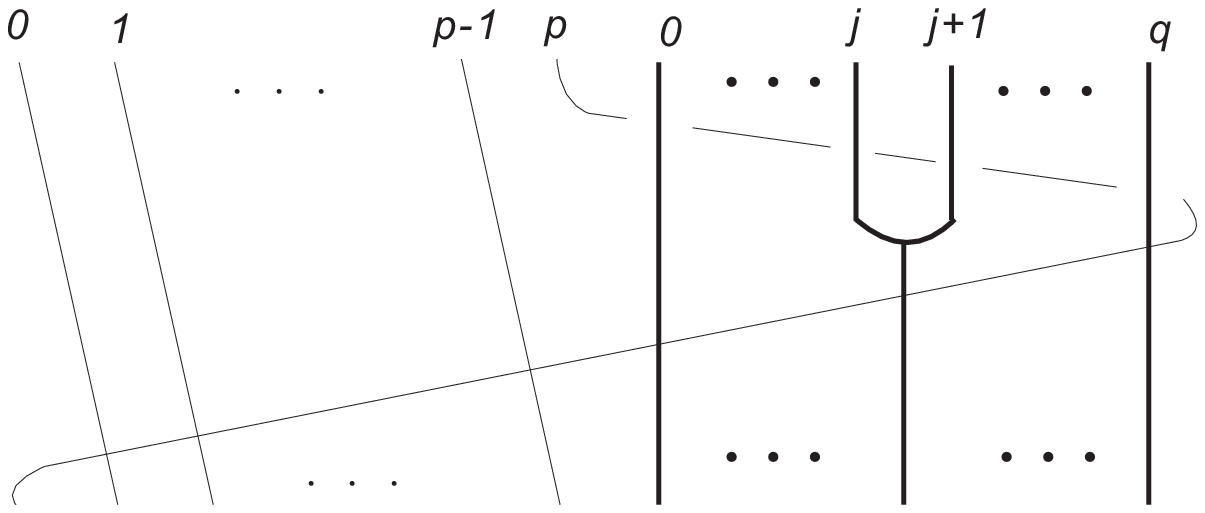}&
\ar[rr]^-{ \txt{$R$ is quasi-\\triangular}}&&}$$
$$ \xymatrix{ \raisebox{-2.5pc}{\includegraphics[scale=0.5]
{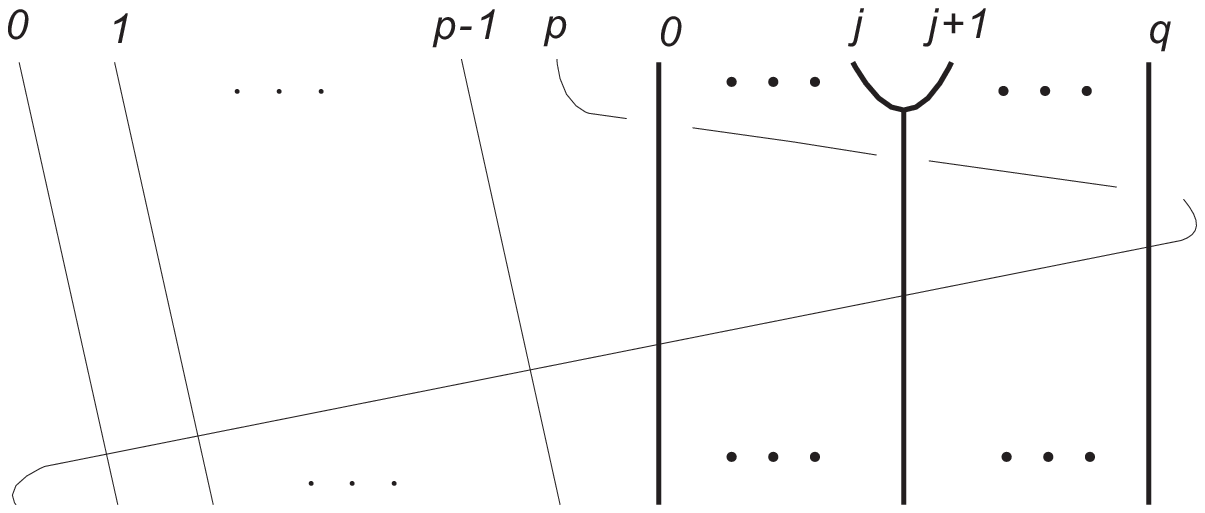}}={t}_{p,q}\bar{d}_j^{p,q}}$$

Similar proof holds for
${d}_i^{p,q}\bar{t}_{p,q}=\bar{t}_{p,q}{d}_i^{p,q}$,
 for $0\leq i<p$.

The flip map and  $R^{\pm 1}$ are quasitriangular and normal,
$d_{p}^{p,q}=d_{0}^{p,q}t_{p,q}$ and
$\bar{d}_{q}^{p,q}=\bar{d}_{0}^{p,q}\bar{t}_{p,q}$, so
 the other commutative equalities can be proved easily.

For the cylindrical condition, we use inductions on $p$ and $q$. For
$p=q=1$, using the fourth picture in the process of turning
$t_{p,q}\bar{t}_{p,q}$ to $\bar{t}_{p,q}t_{p,q}$, we get
$$\xymatrix{\ar@{}[r]|-{\displaystyle(t_{1,1}\bar{t}_{1,1})^2=\qquad}&
\includegraphics[scale=0.5,height=5cm]{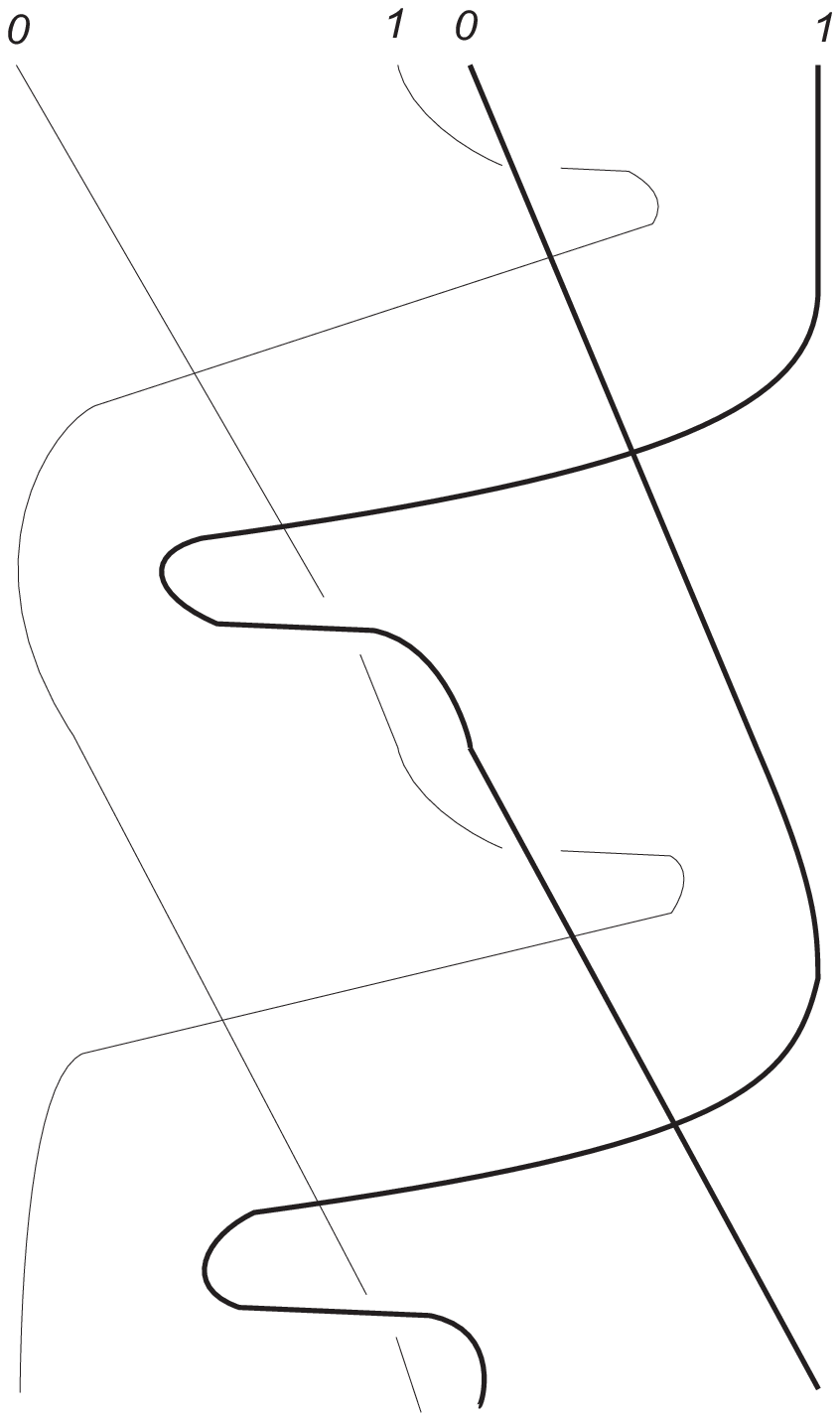}
\ar[r]^-{\txt{(I)}}&\includegraphics[scale=0.5,height=5cm]{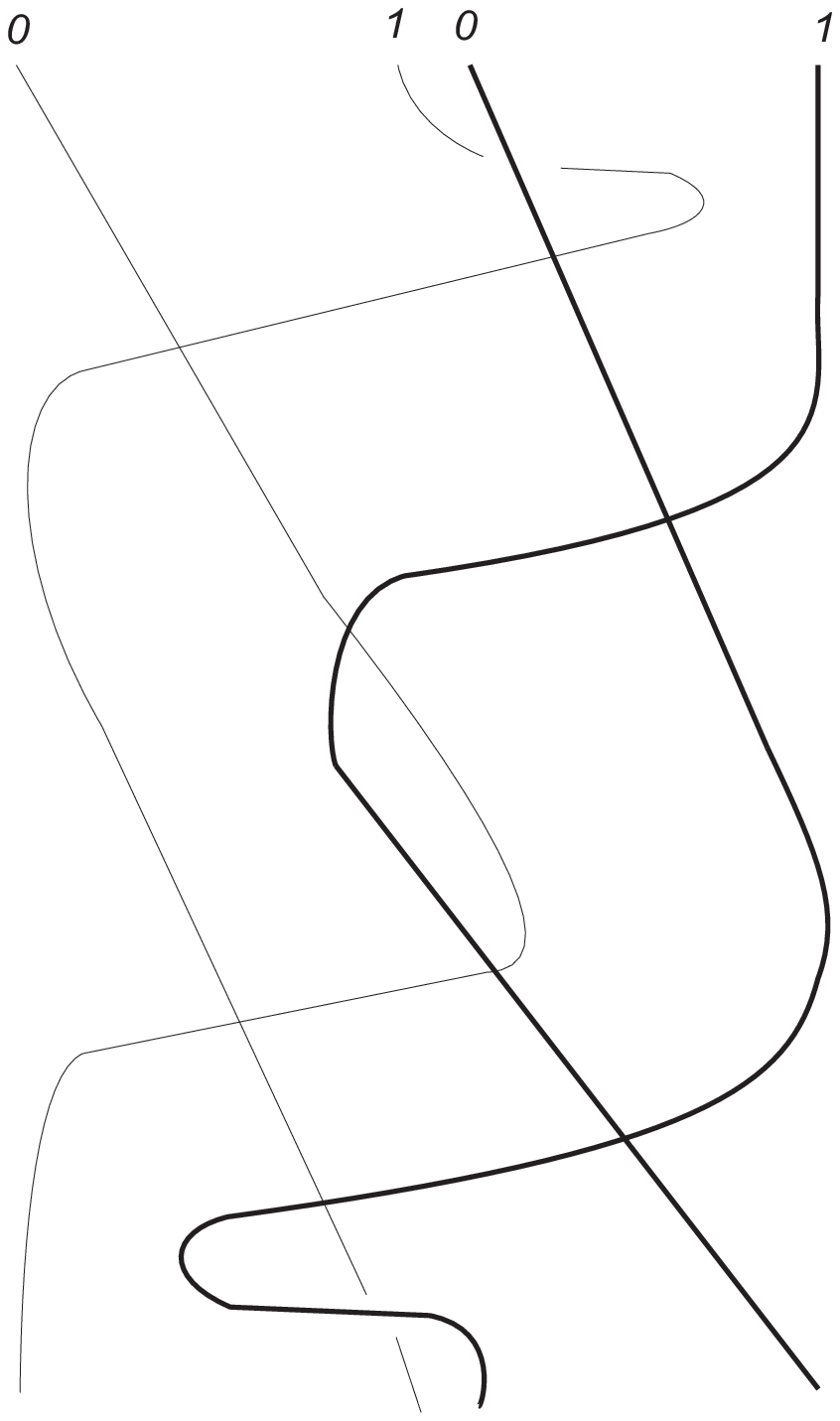}\ar[r]^-{\txt{(I)}}&}$$
$$\xymatrix{
\includegraphics[scale=0.5,height=5cm]{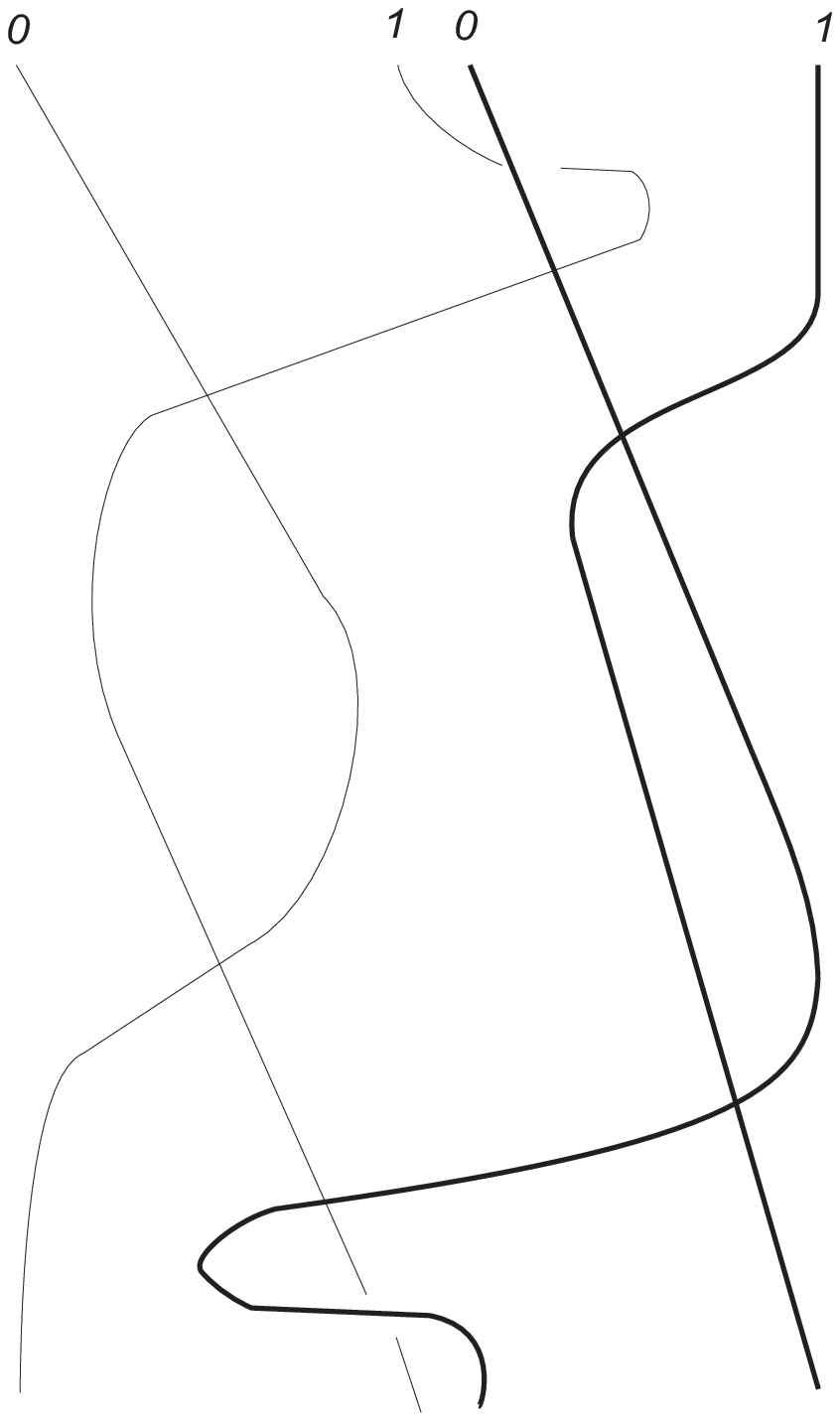}
\ar[r]^-{\txt{(I)}}&\includegraphics[scale=0.5,height=5cm]{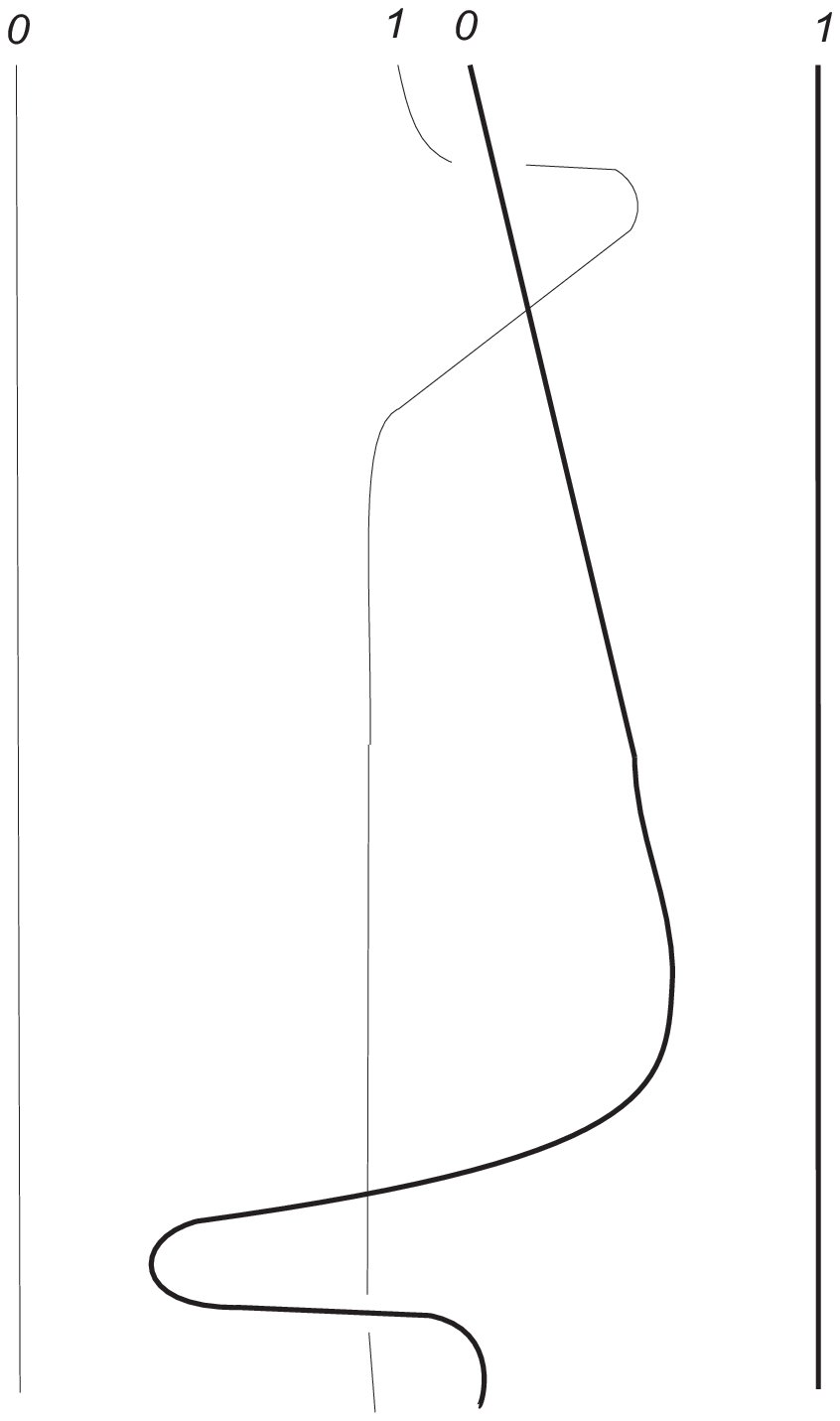}\ar[r]^-{\txt{(I)}}&
\includegraphics[scale=0.5,height=5cm]{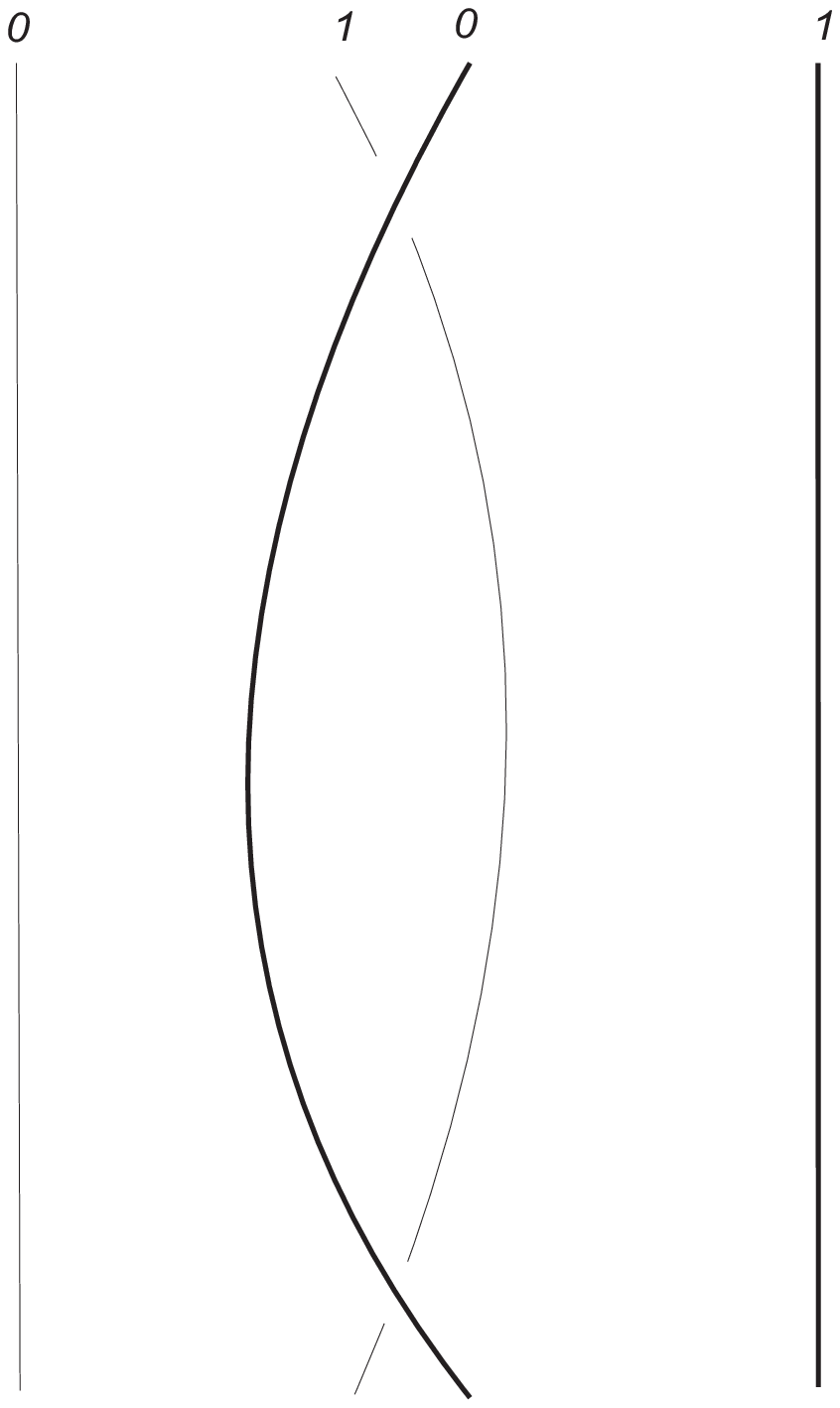}
\ar@{}[r]|-{{\displaystyle=~ id.}}&}$$ Suppose that
$t_{m,n}^{m+1}\bar{t}_{m,n}^{\,n+1}=id_{m,n}$ for $\forall\, m<p$
and $\forall\, n<q$, we need to prove
$t_{p,q}^{p+1}\bar{t}_{p,q}^{\,q+1}=id_{p,q}$. We have
$$
 t_{p,q}^{p+1}=
\raisebox{-5pc}{\includegraphics[scale=0.5,height=5cm]{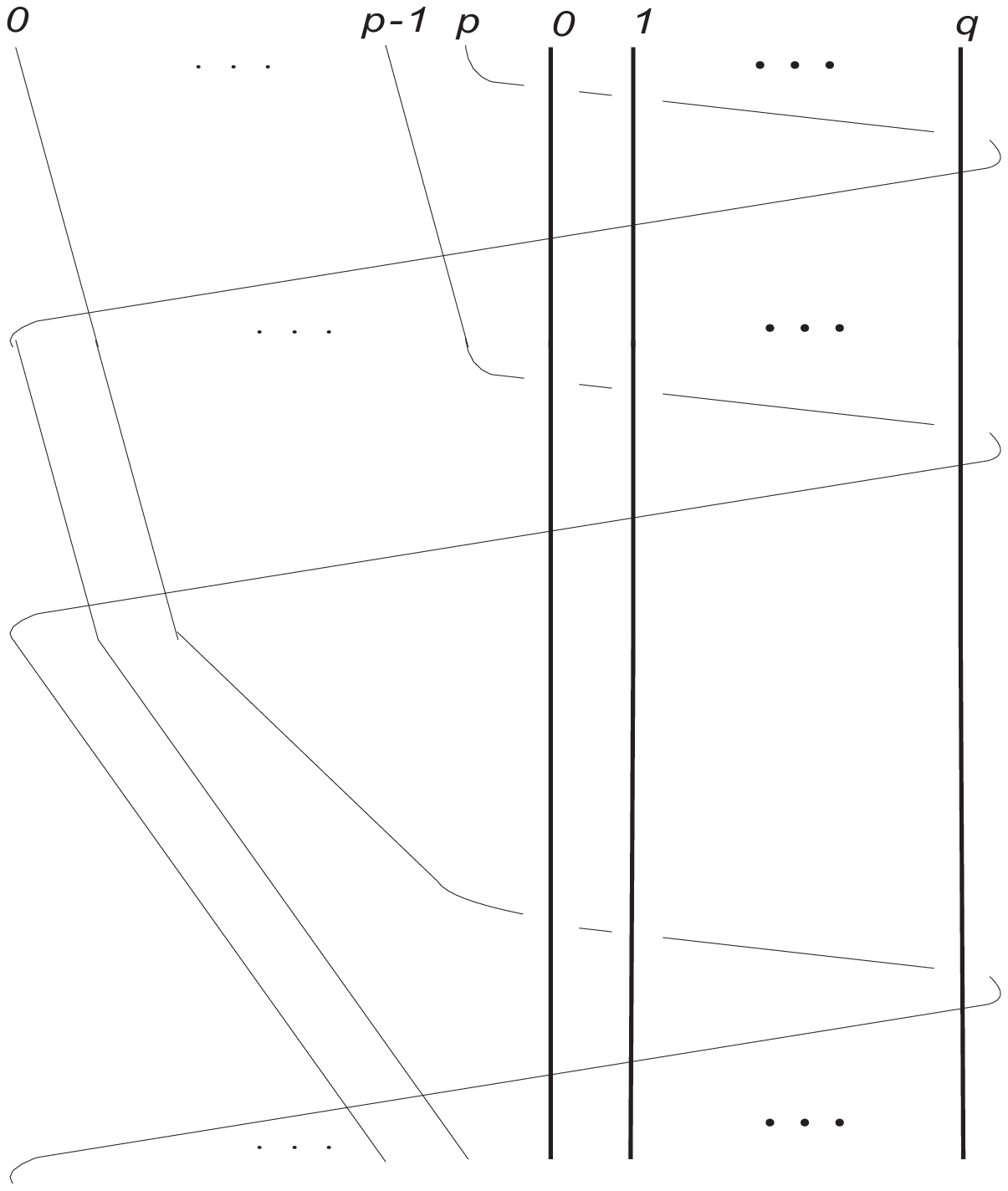}}\quad
\text{and}\quad
\bar{t}_{p,q}^{\,q+1}=\raisebox{-5pc}{\includegraphics[scale=0.5,height=5cm]{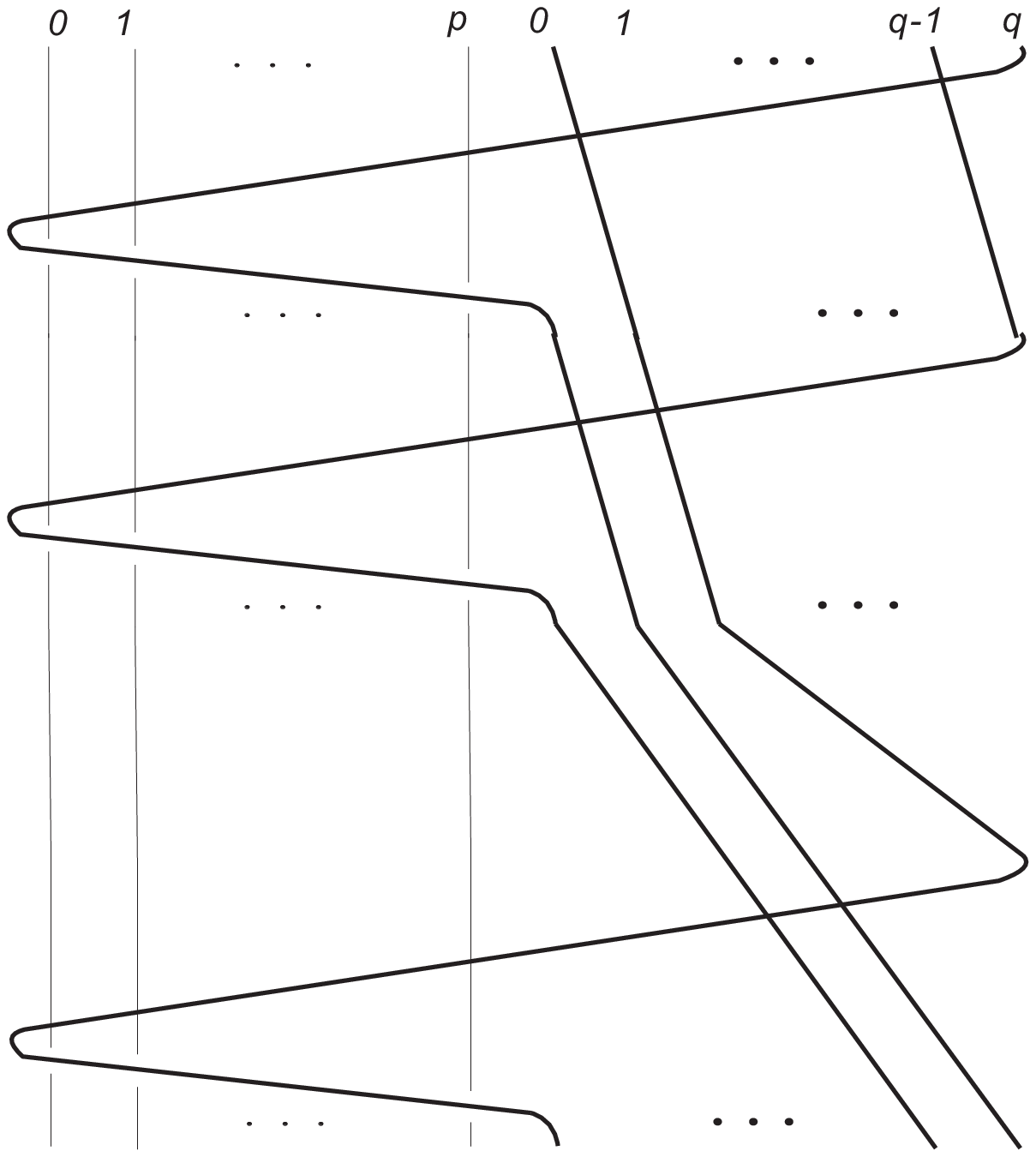}}.$$

We can use bands in  graphs to stand for parallel lines, that is,
lines without any intersections or crosses between themselves.

If we can draw the elements $b_0,\ldots,b_{p-1}$ of $B$ together by
a grey  band, and draw the elements $a_0,\ldots,a_{q-1}$ of $A$
together by a black band, then using the movements for the case
$m=n=1$, we will get the proposition. We just give the equivalent
moves  for turning the lines $b_0$ and $b_1$ in the graph of
$t^{p+1}_{p,q}$ to parallel lines, others can be done by similar
moves. The only intersections between $b_0$ and $b_1$ occur while
doing the $p$-th and $p+1$-th powers of $t_{p,q}$. So we concentrate
on that part of graph.

$$\xymatrix{\includegraphics[scale=0.5,height=4cm,width=5cm]{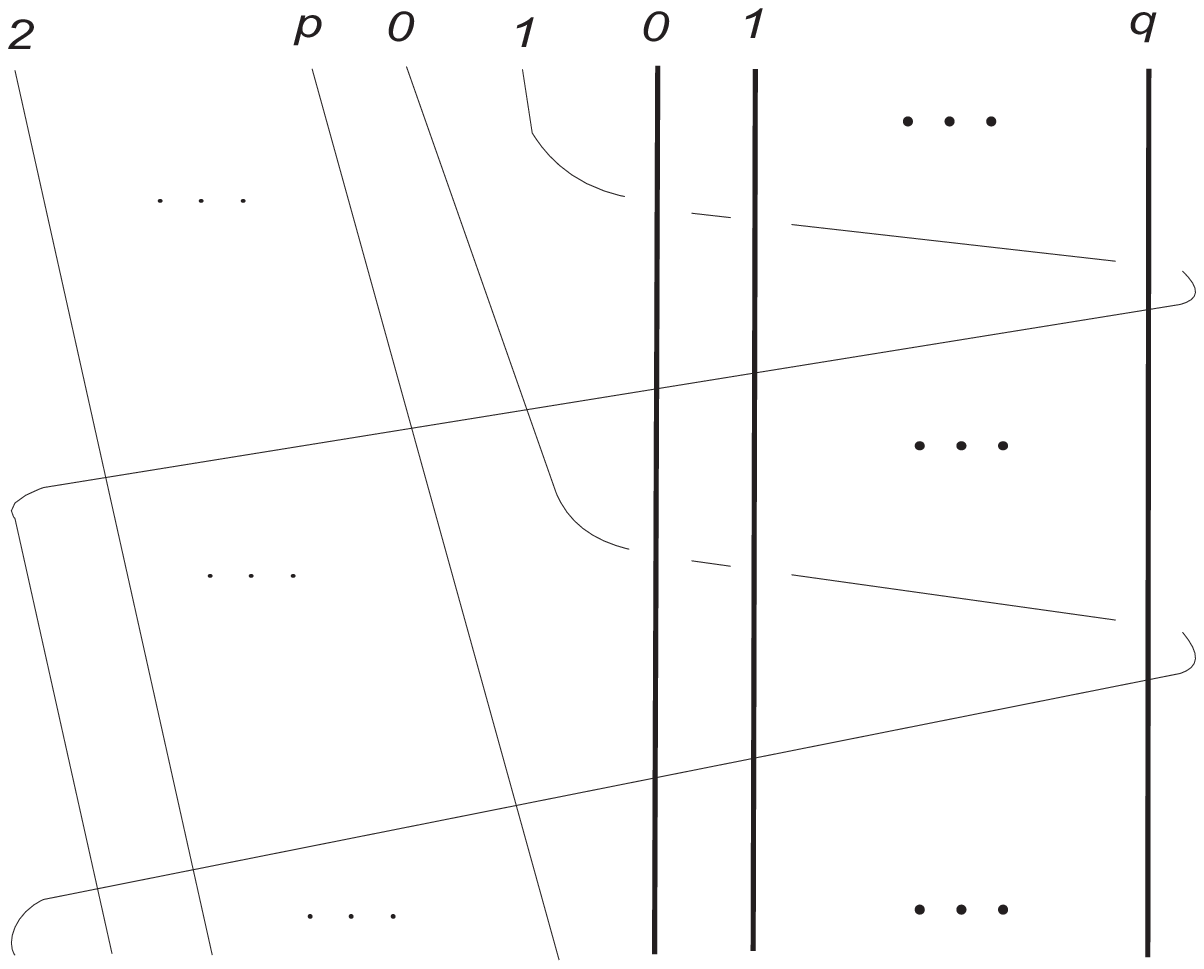}\ar[r]^-{\txt{(II)}}&
\includegraphics[scale=0.5,height=4cm,width=5cm]{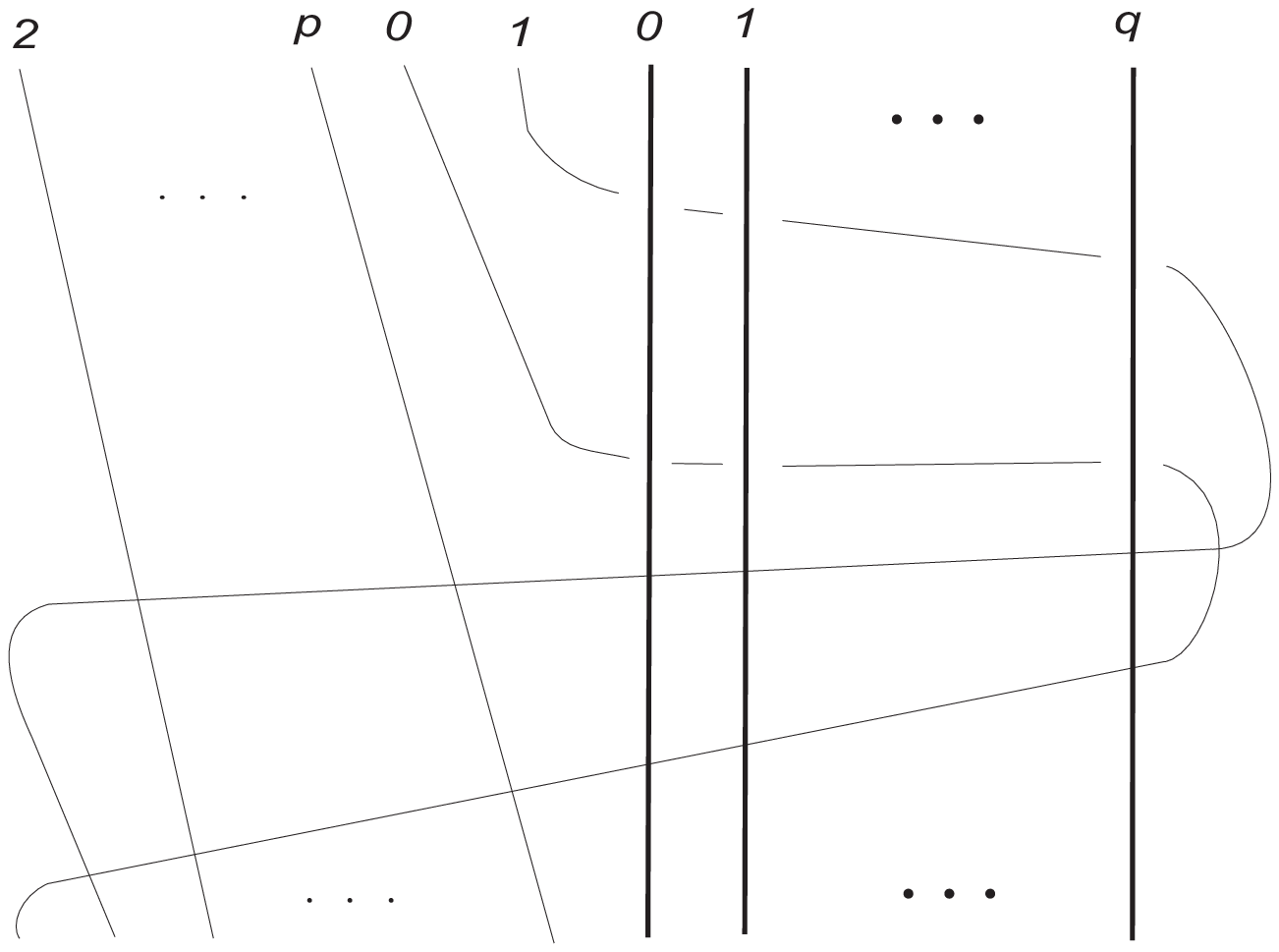}}$$
$$\xymatrix{
\ar[rr]^-{\txt{flip maps'\\ braid\\ relations}}
&&\includegraphics[scale=0.5,height=4cm,width=5cm]{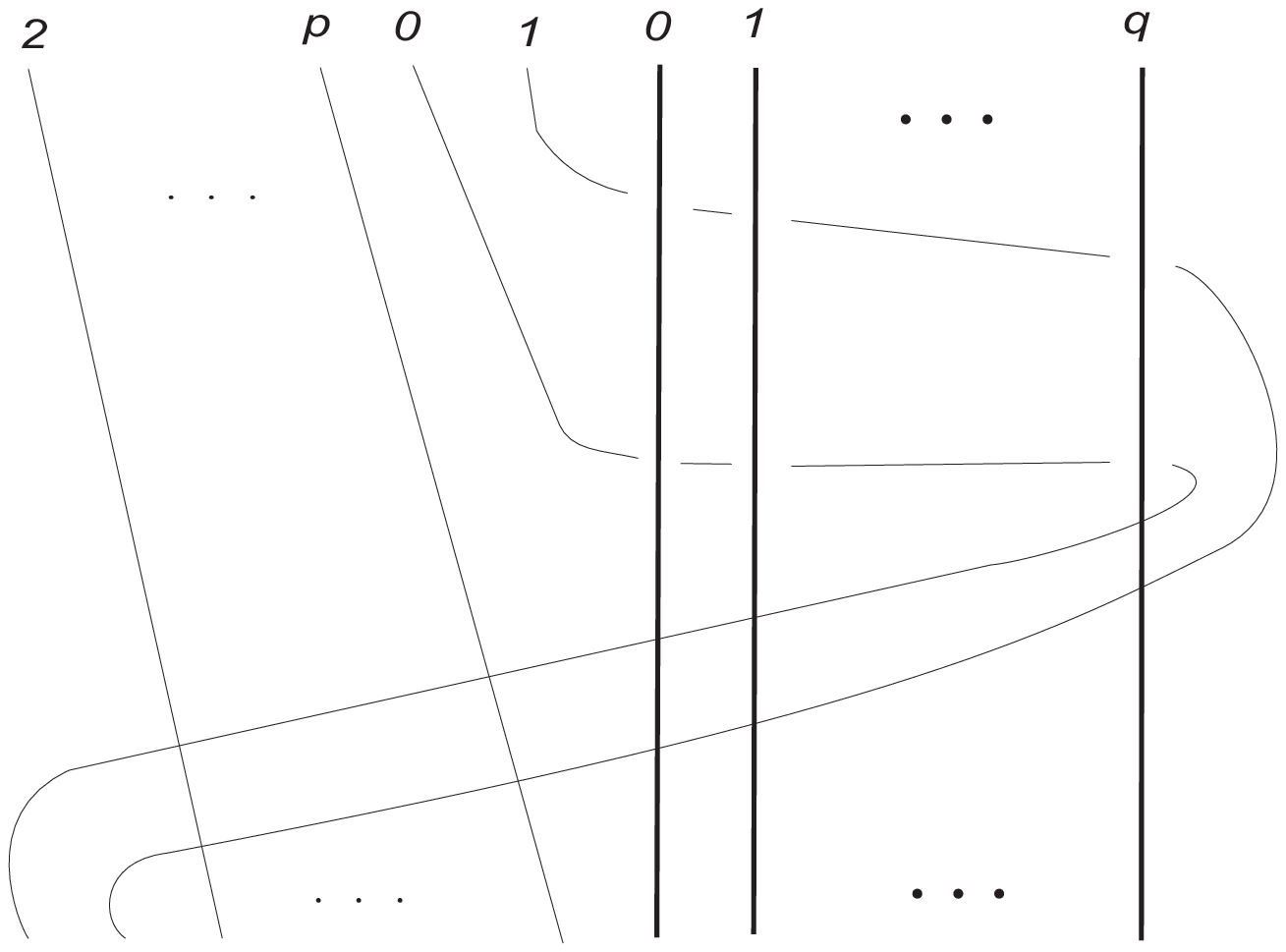}}.$$
\end{proof}

Let $\Delta_{\bullet}(A\natural B)$ be the \textit{diagonal} of the
cylindrical module $A\natural B$, i.e.,
$$\Delta_n(A\natural B)=A\natural B(n,n).$$ It is a cyclic module with face maps
$d_i=d_i^{n,n}\bar{d}_i^{n,n}$, degeneracy maps
$s_i=s_i^{n,n}\bar{s}_i^{n,n}$ and  the cyclic operator
$t_n=t_{n,n}\bar{t}_{n,n}$.

\begin{prop}\label{2.6}
$\Delta_{\bullet}(A\natural B)$ is isomorphic to
$C_{\bullet}(A\#_{_R}B)$ as cyclic modules.
\end{prop}
\begin{proof}Define  morphisms $\Phi_n: A\natural B(n,n)\ra
C_n(A\#_{_R}B)$ by
%$$\Phi=(\Theta_0 \ot id^{\ot 2n})\cdots (id^{\ot(n-1)}\ot \Theta_{n-1}\ot id^{\ot 2})(id^{\ot n}\ot \Theta_n)$$
$$\Phi=R_{2n+1,2n+2} (R_{2n-1,2n}R_{2n,2n+1})
%(R_{2n-3,2n-2}R_{2n-2,2n-1}R_{2n-1,2n})
\cdots(R_{12}R_{23}\cdots R_{n+1,n+2}),$$
 and $\Psi_n:C_n(A\#_{_R}B)\ra A\natural
B(n,n)$ by
$$\Psi=R^{-1}_{n+1,n+2}(R^{-1}_{n,n+1}R^{-1}_{n+2,n+3})
\cdots (R^{-1}_{12}R^{-1}_{34}\cdots R^{-1}_{2n+1,2n+2}).$$
$$\Phi=\raisebox{-2.5pc}{\includegraphics[scale=0.5]{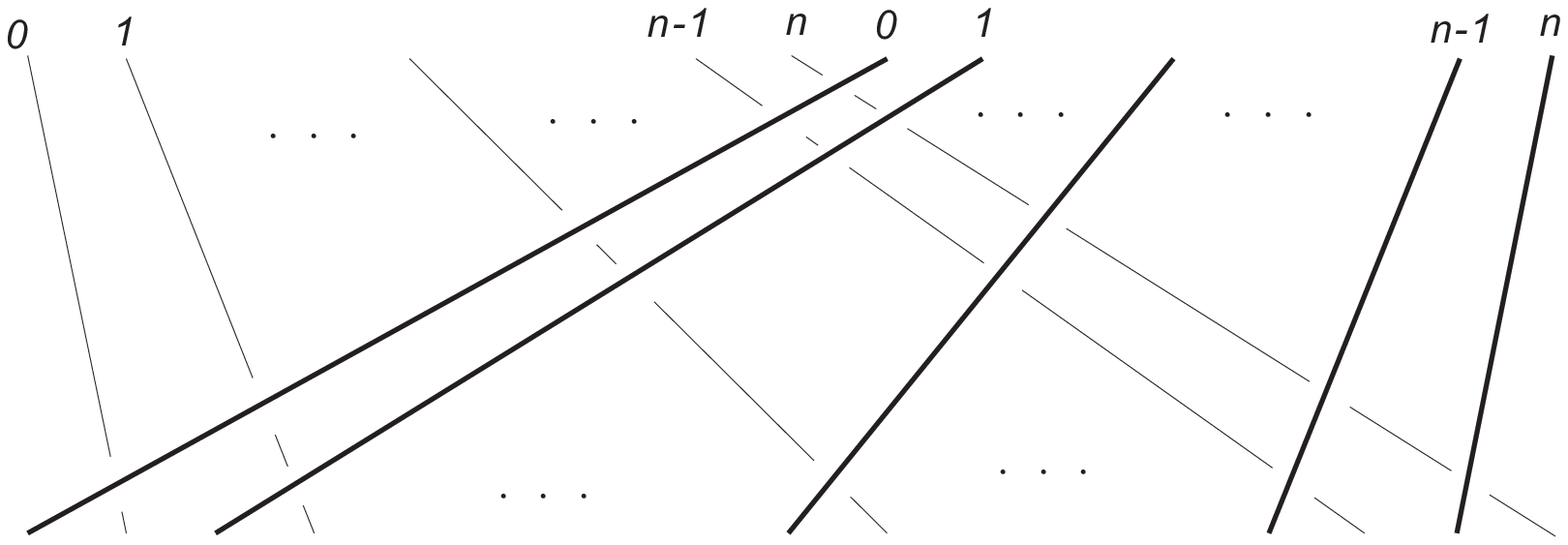}}$$

$$\Psi=\raisebox{-2.5pc}{\includegraphics[scale=0.5]{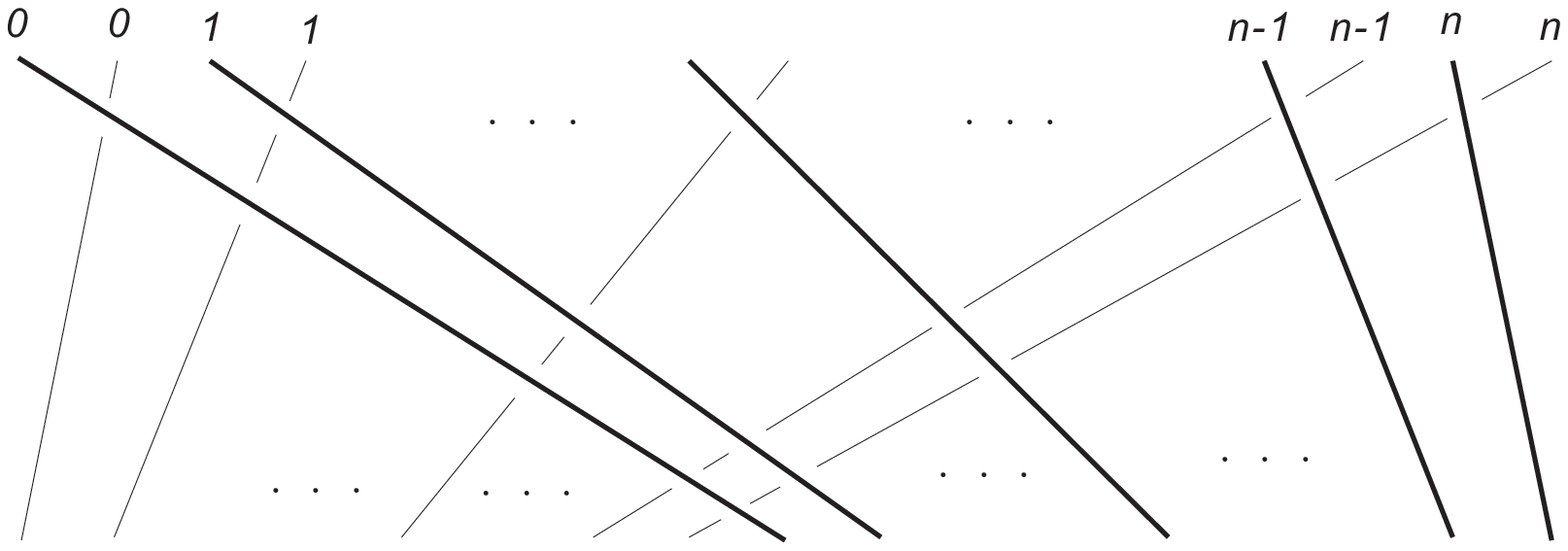}}$$

Note that
$R^{\pm1}_{i,i+1}R^{\pm1}_{j,j+1}=R^{\pm1}_{j,j+1}R^{\pm1}_{i,i+1}$
for $|i-j|>1$. So $\Phi$ and $\Psi$ are inverses to each other.

We need to prove that $\Phi$ and $\Psi$ are  morphisms of cyclic
modules. We only show that $\Phi$ commutates with the cyclic
operator and the face maps. It is similar for $\Psi$.

Again using the fourth picture in the process of turning
$t_{p,q}\bar{t}_{p,q}$ to $\bar{t}_{p,q}t_{p,q}$, we get
$$\xymatrix{\ar@{}[r]|-{\displaystyle \Phi t_{n,n}\bar{t}_{n,n}=}
&\quad\includegraphics[scale=0.5,height=5cm]{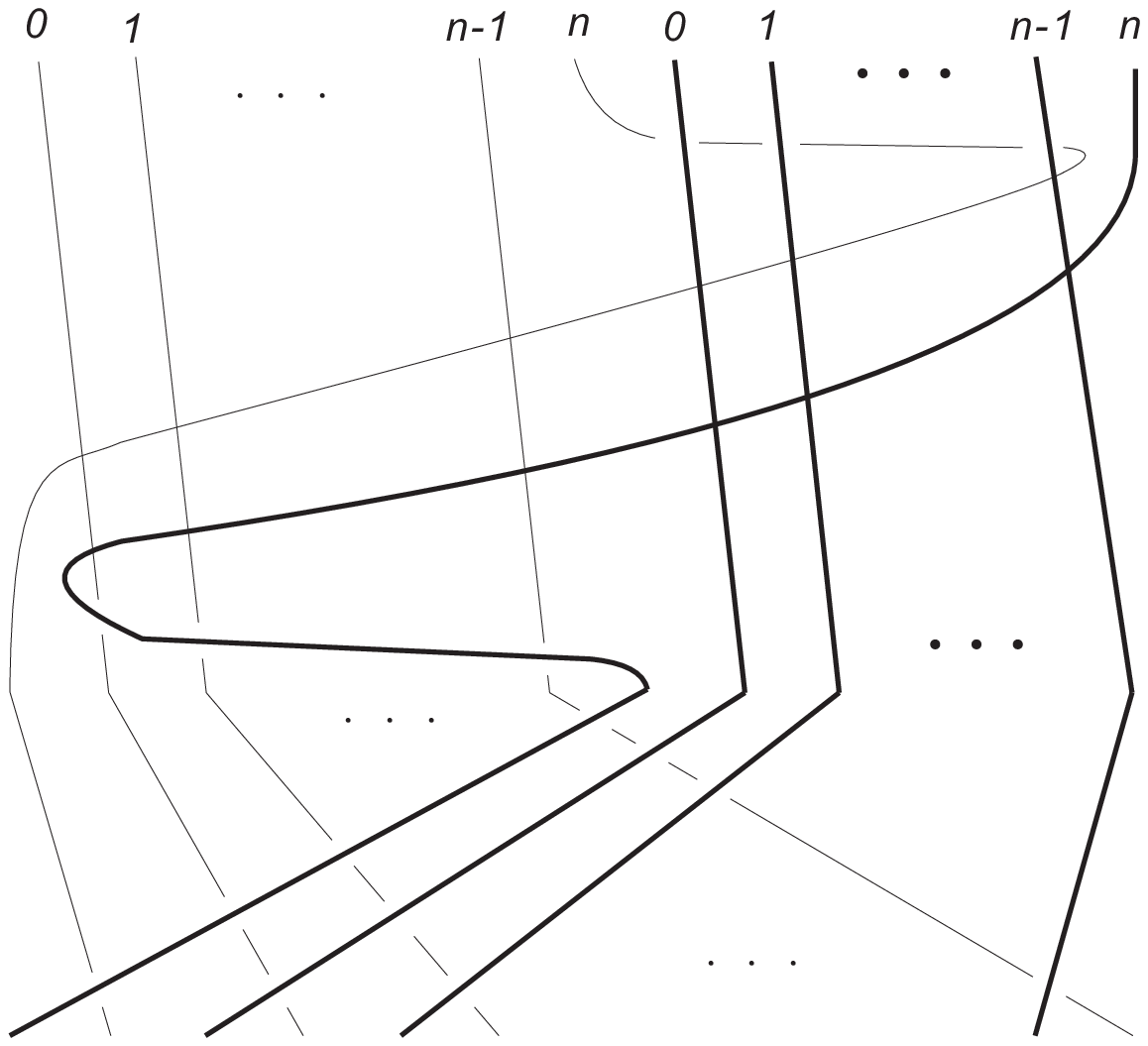}
\ar[r]^-{\txt{(I)}}&\includegraphics[scale=0.5,height=5cm]{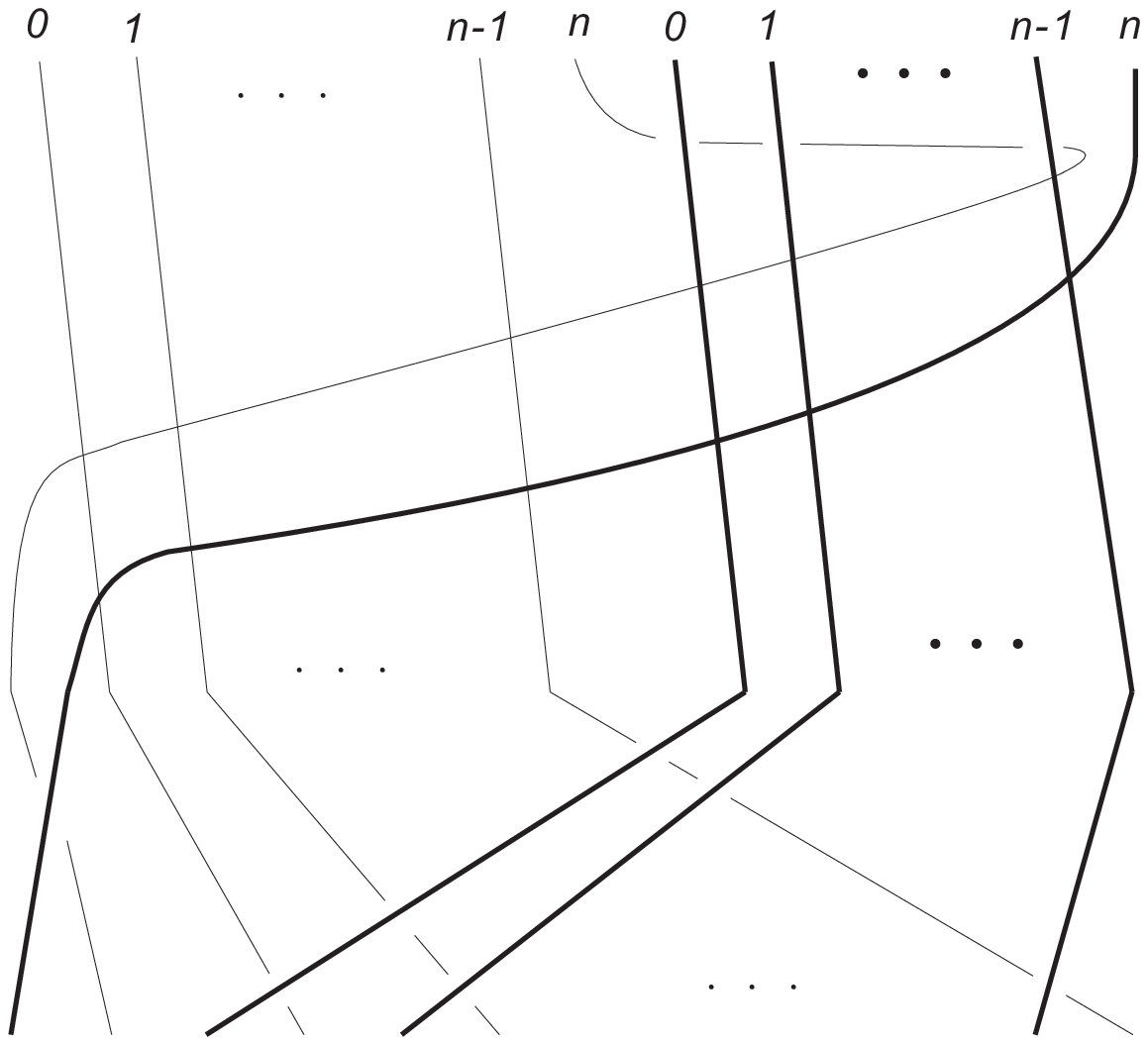}
}$$
$$\xymatrix{\ar[r]^-{\txt{(I)}}
&\quad\includegraphics[scale=0.5,height=5cm]{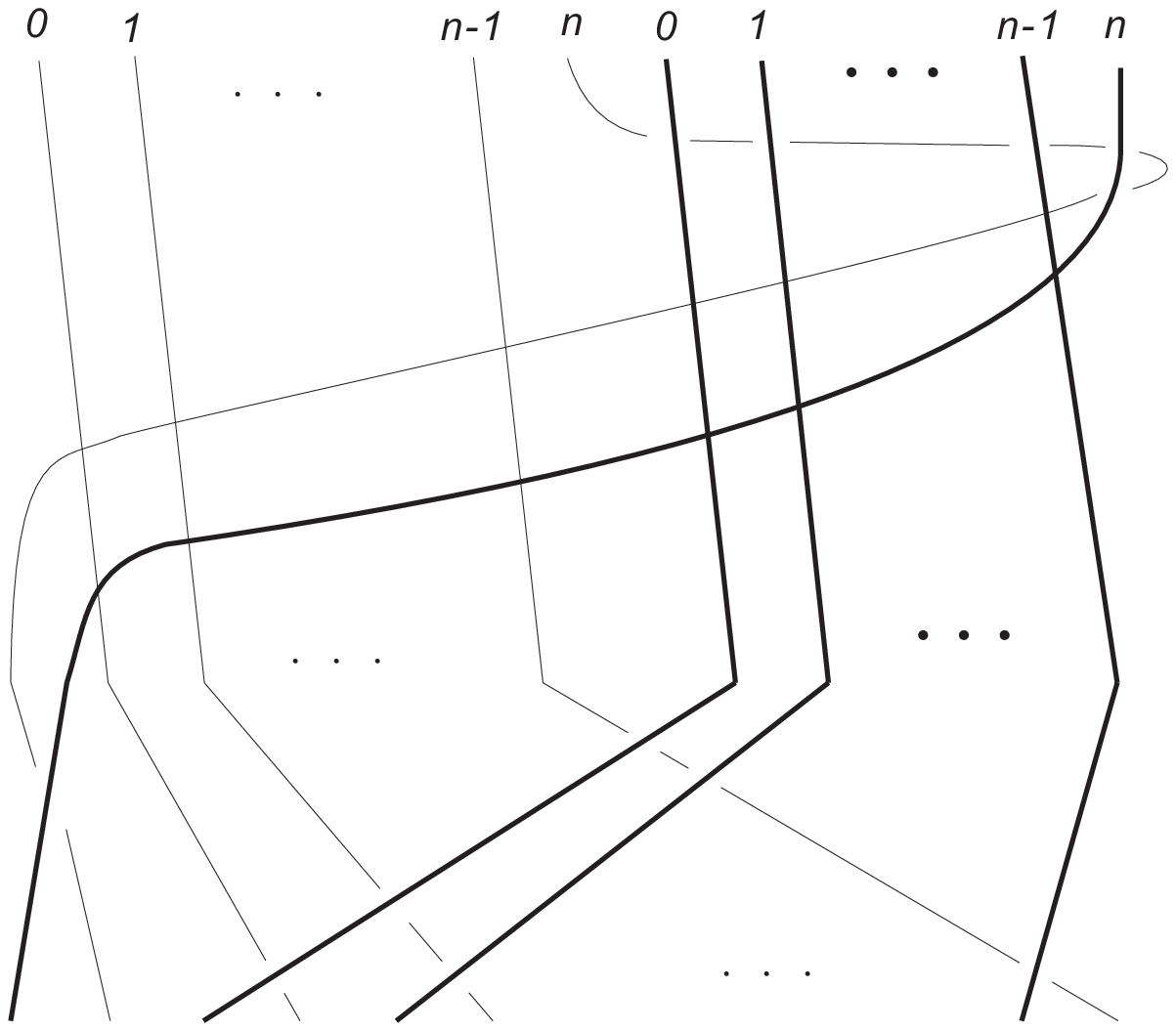}
\ar[r]^-{\txt{(I)}}_-{\txt{(II)}}&\includegraphics[scale=0.5,height=5cm]{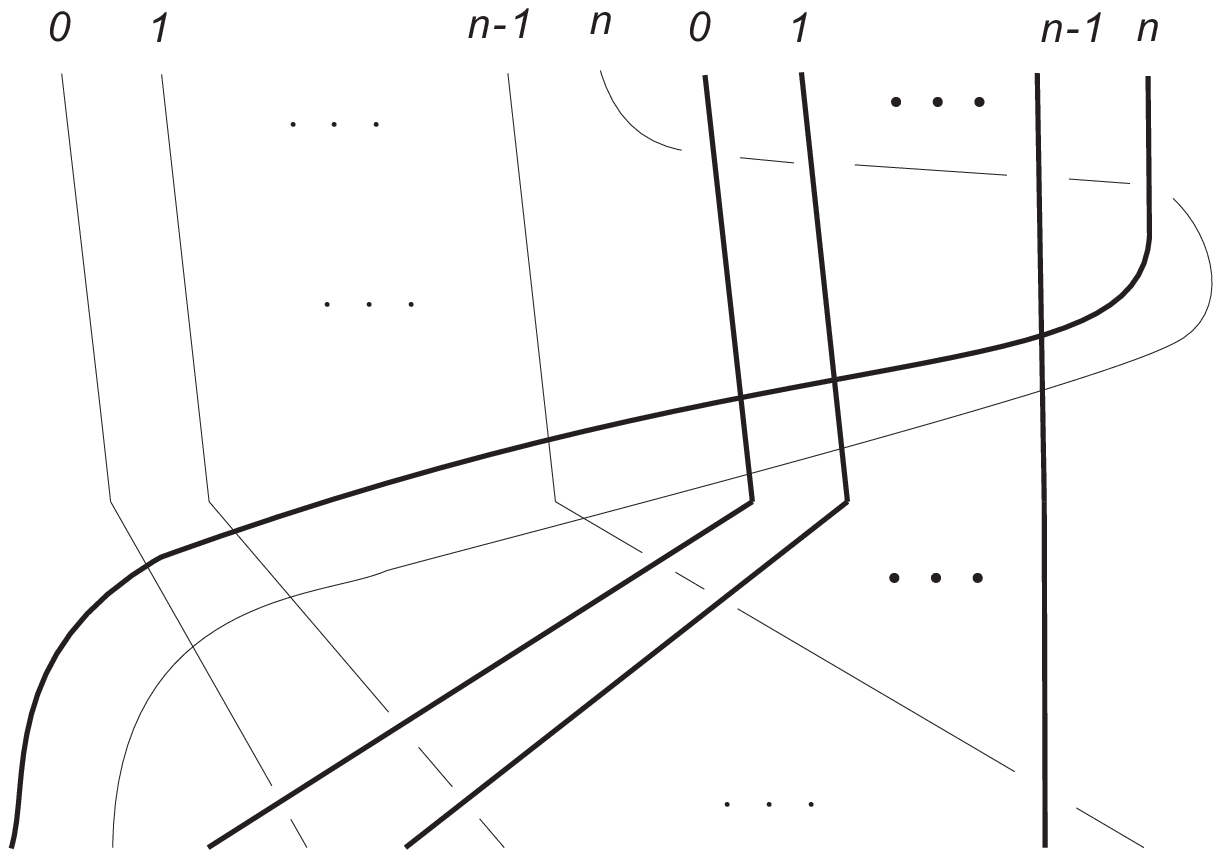}
}$$
$$\xymatrix{\ar[r]^-{\txt{(II)}}
&\quad\includegraphics[scale=0.5,height=5cm]{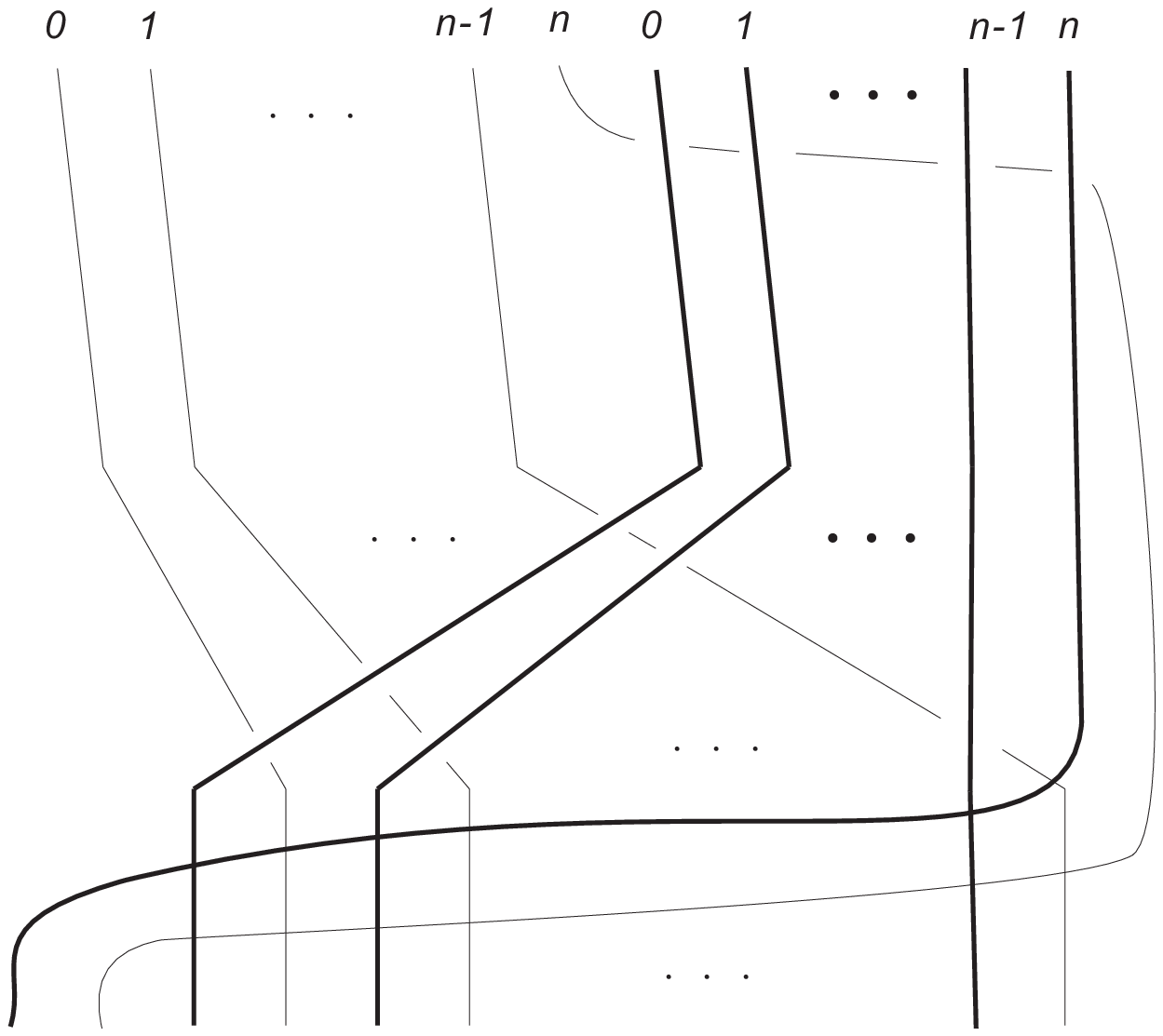}
\ar[r]^-{\cong}&\includegraphics[scale=0.5,height=5cm]{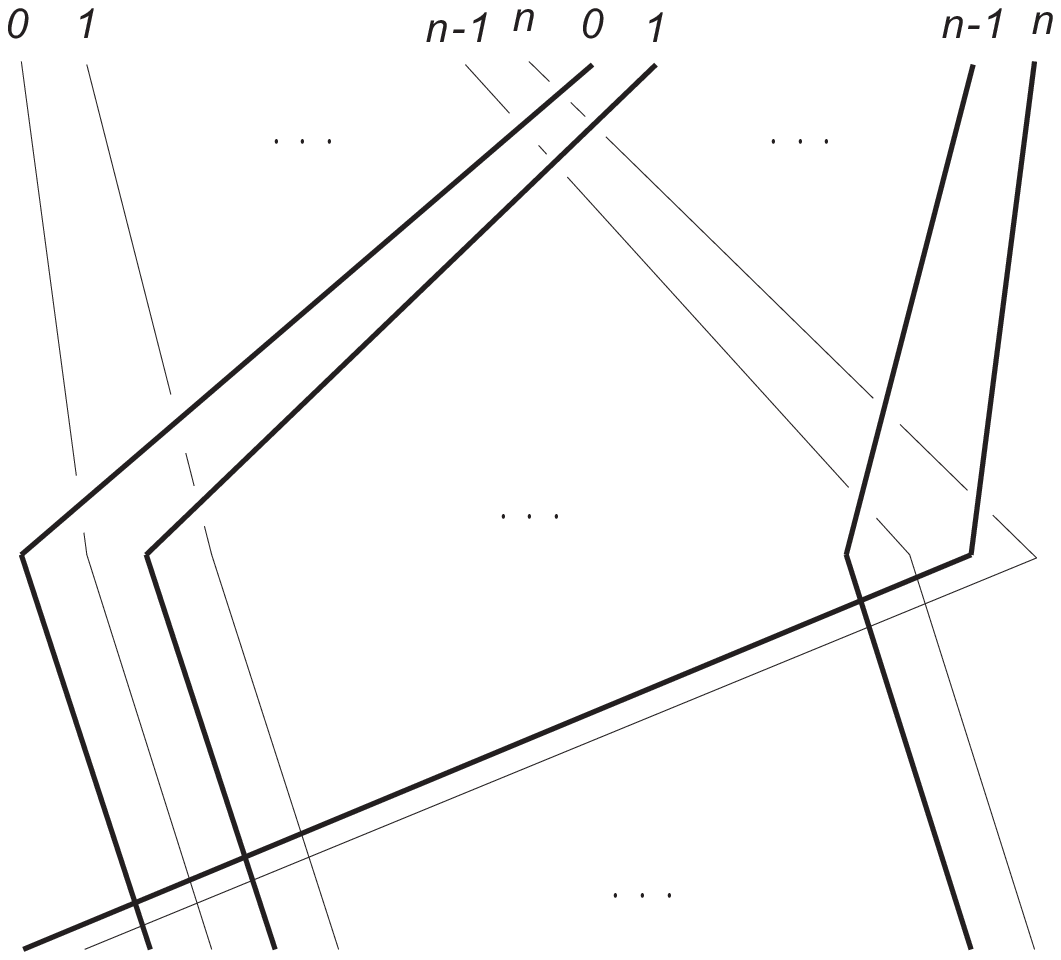}\ar@{}[r]|-{\displaystyle
=t\Phi. }& }$$
 Since $R$ and $R^{-1}$ are quasitriangular,  for
$0\leq i<n$,
$$\xymatrix{\ar@{}[r]|-{\displaystyle d_i\Phi =}
&\quad\includegraphics[scale=0.5,height=3.5cm,width=5.5cm]{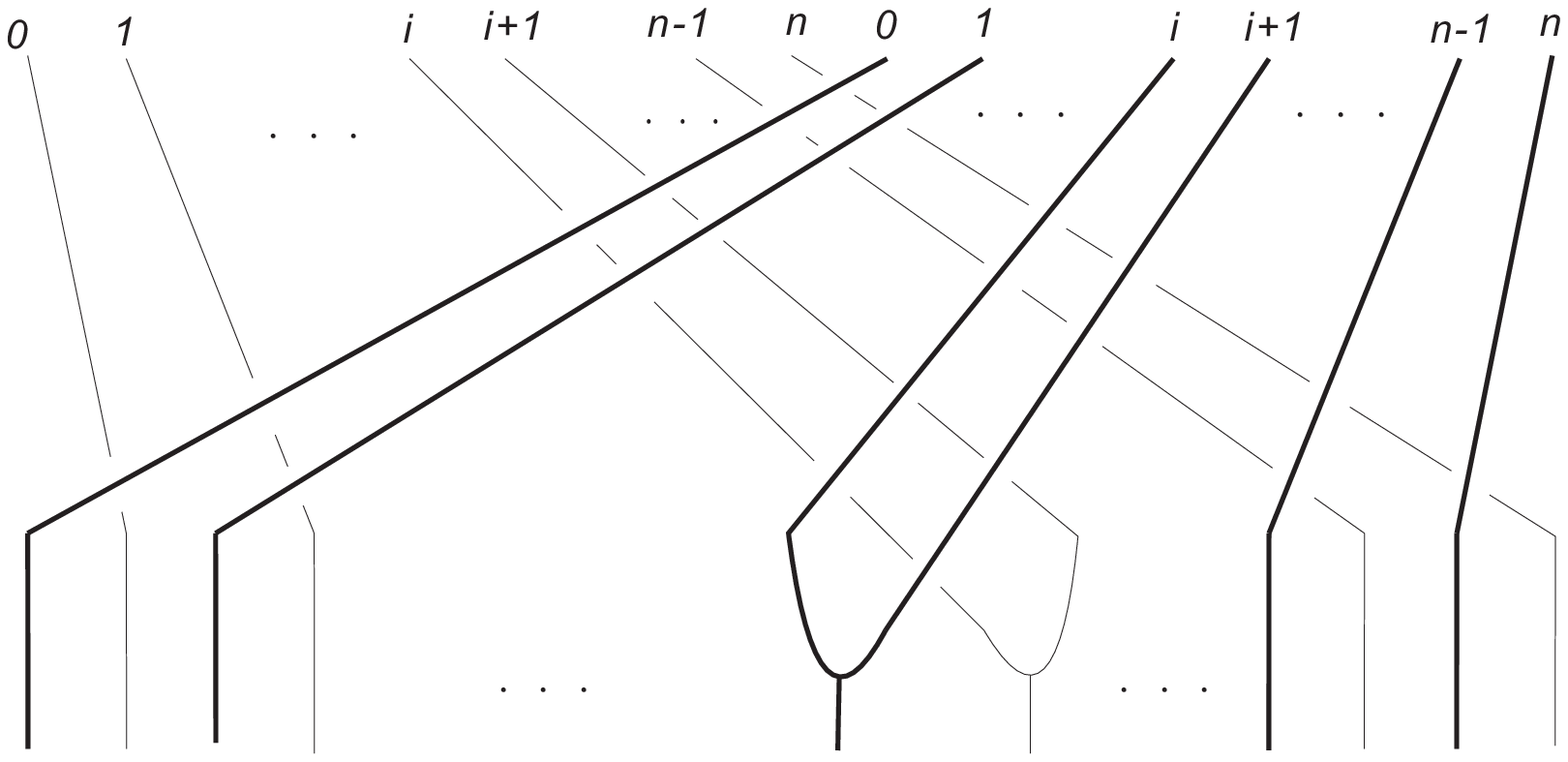}
\ar@{}[r]|-{\cong}&\includegraphics[scale=0.5,height=3.5cm,width=5.5cm]{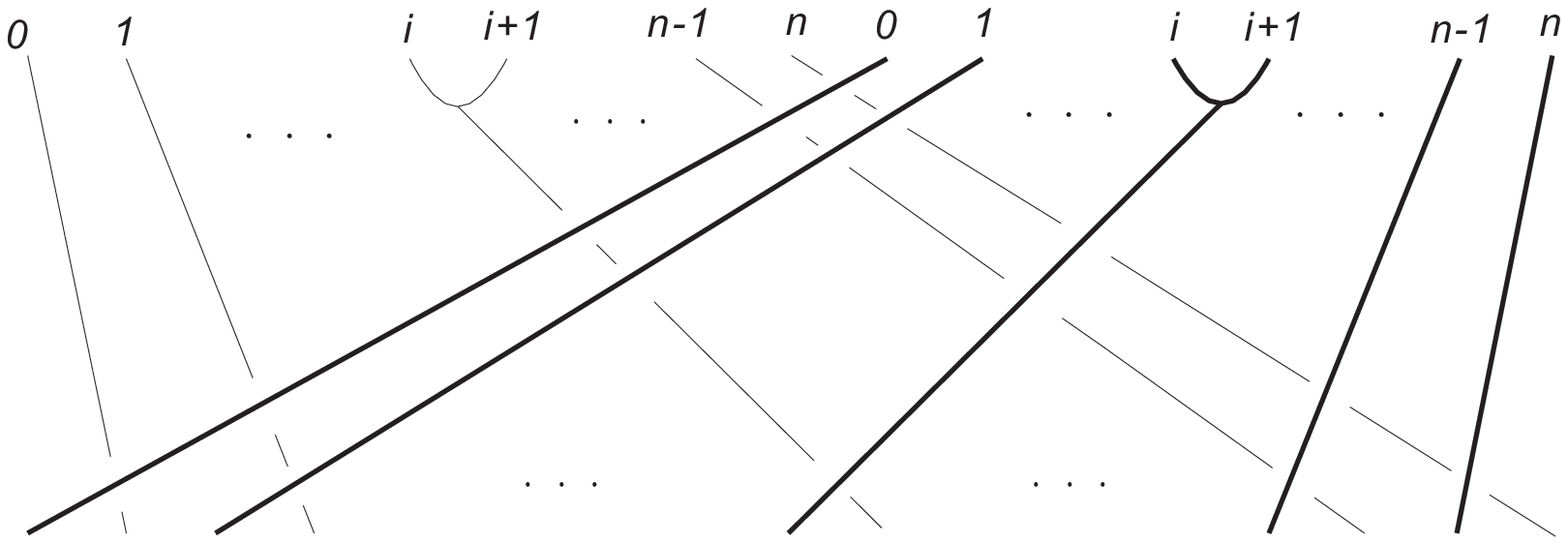}\quad\ar@{}[r]|-{\displaystyle
=\Phi\bar{d}^{n,n}_id^{n,n}_i. }& }$$
\end{proof}

\section{\label{33}Application of the generalized Eilenberg-Zilber theorem}
Let  $(\{C_{m,n}\}_{m,n\geq
0},d^{m,n}_i,s^{m,n}_i,t_{m,n},\bar{\,d}^{m,n}_j,\bar{s}^{m,n}_j,\bar{t}_{m,n})$
be a cylindrical module. We can set as in \eqref{6} the degree $-1$
endomorphism $\mathrm{b}$ (resp.$\bar{\,\mathrm{b}}$), the degree
$1$ endomorphism $\mathrm{B}$ (resp.$\bar{~\mathrm{B}}$) and the
degree $0$ endomorphism $\mathrm{T}$ (resp. $\bar{\mathrm{T}}$)
associated with $d_i,s_i,t$ (resp. $\bar{d}_j,\bar{s}_j,\bar{t}$).
The total parachain complex is
 a mixed complex. Explicitly, let $\mathbb{C}_n=\bigoplus_{i+j=n}C_{i,j}$,
$\mathbbm{b}=\mathrm{b}+\!\!\bar{~\mathrm{b}}$ and
$\mathbb{B}=\mathrm{B}+\mathrm{T}\bar{~\mathrm{B}}$. Since
$C_{\bullet,\bullet}$ is a cylindrical module,
$\mathrm{T\bar{T}}=1$,
  $\mathrm{\bar{\,b}B=-B\bar{\,b}}$ and
$\mathrm{\bar{\,B}b=-b\bar{\,B}}$. Then by Lemma \ref{2.2},
\begin{align*}
\mathbbm{b}\mathbb{B}+\mathbb{B}\mathbbm{b} &
=(\mathrm{b}+\!\!\bar{~\mathrm{b}})(\mathrm{B}+\mathrm{T}\bar{~\mathrm{B}})
 +(\mathrm{B}+\mathrm{T}\bar{~\mathrm{B}})(\mathrm{b}+\!\!\bar{~\mathrm{b}})\\
&=\mathrm{bB}+\mathrm{Bb}+\!\!\bar{~\mathrm{b}}\mathrm{B}
+\mathrm{B\bar{\,b}}+\mathrm{T(\!\bar{~B}b+b\!\bar{~B})}+\mathrm{T(\!\!\bar{~{b}}\!\bar{~{B}}+\!\bar{~B}\bar{\,b})}\\
&=1-\mathrm{T}+\mathrm{T}(1-\mathrm{\bar{T}})=0,
\end{align*}
$(\mathbb{C}_{\bullet},\mathbbm{b},\mathbb{B})$ is a mixed complex.

The generalized Eilenberg-Zilber theorem for paracyclic modules was
proved by Getzler and Jones \cite{GJ}  using topological method,
later it was reproved by Khalkhali  and Rangipour \cite{KR2} using
an algebraic method. The theorem tells us that, for a cylindrical
module there exists a quasi-isomorphism from its total  mixed
complex to its diagonal mixed complex. Due to Proposition \ref{2.6},
we have:

\begin{theorem} Let $A\#_{_R}B$ be a strong smash product algebra, $A\natural
B$ a cylindrical module defined in \eqref{7} and \eqref{8}. Then
there exists a quasi-isomorphism of mixed complexes
$\Tot_{\bullet}(A\natural B)$
  and $C_{\bullet}(A\#_{_R}B)$.
\end{theorem}

It was discovered by Getzler and Jones \cite{GJ} that the Hochschild
homology, the cyclic homology, the negative cyclic homology and the
periodic cyclic homology can be unified to be cyclic homologies of a
mixed complex with coefficients. Specifically, let $M_{\bullet}$ be
a mixed complex and $W$ be a graded $k[u]$-module, denote
$M_{\bullet}[[u]]\ot_{k[u]}W $ by $M_{\bullet}\boxtimes W$. Note
that this tensor product is a graded tensor product. Let
$(C_{\bullet}, \mathrm{b, B})$ be the mixed complex associated to
its cyclic module structure. $C_{\bullet}\boxtimes W$ is a complex
with  the differential $(\mathrm{b}+u\mathrm{B})\ot_{k[u]} id_W$.
Call the homology of the complex $C_{\bullet}\boxtimes W$
\textit{the cyclic homology of the mixed complex of $C_{\bullet}$
with coefficients in $W$} and denote it by $\Hc_{*}(C_{\bullet};
W)$. Then
 for $W=k[u]$ (resp. $k[u,u^{-1}],k[u,u^{-1}]/uk[u]$ and $k[u]/uk[u]$)
$\Hc_{*}(C_{\bullet}; W)=\Hc^{-}_{*}(C_{\bullet})$ (resp. $
\mathrm{HP}_{*}(C_{\bullet}),\Hc_{*}(C_{\bullet})$ and
$\Hh_{*}(C_{\bullet})$).
 If  $C$ is the usual
cyclic module associated with an algebra $A$, then we simply denote
$\Hc_*(C_{\bullet}(A);W)$ by $\Hc_*(A;W)$.

 The first author would like to thank Professor
Getzler for pointing out the flatness condition concealed here,
which turns out useful in the consequent arguments.

\begin{lemma}\label{3.02}
Let $k$ be a field, $V$ a $k$-vector space and $u$ a variable. Then
$V[[u]]$ is a flat $k[u]$-module.
\end{lemma}
\begin{proof}Since $k[u]$ is a principal ideal domain, a
$k[u]$-module is flat if and only if it is  torsion-free. Clearly,
$V[[u]]$ is a torsion-free $k[u]$-module.
\end{proof}

\begin{lemma}
Let $\mathcal {R}$ be a ring, $M$ a left $\mathcal {R}$-module and
$P_{\bullet},Q_{\bullet}$ bounded below complexes of flat right
$\mathcal{R}$-modules. If $P_{\bullet}$ and $Q_{\bullet}$ are
quasi-isomorphic, then
$$\Ho_n(P_{\bullet}\ot_{\mathcal
{R}}M)\cong\Ho_n(Q_{\bullet}\ot_{\mathcal {R}}M).$$
\end{lemma}
\begin{proof} We know from \cite{W} that, for bounded below complexes of flat
right $\mathcal{R}$-modules $P_{\bullet}$ and $Q_{\bullet}$,
$$\Ho_n(P_{\bullet}\ot_{\mathcal
{R}}M)\cong\mathbb{T}\mathrm{or}_n^{\mathcal
{R}}(P_{\bullet},M)\text{ and }\Ho_n(Q_{\bullet}\ot_{\mathcal
{R}}M)\cong\mathbb{T}\mathrm{or}_n^{\mathcal {R}}(Q_{\bullet},M)$$
for each $n$, where $\mathbb{T}\mathrm{or}$ is the hypertor. And we
have spectral sequences converging to them, that is,
\begin{gather*}
E^2_{p,q}(P)=\Tor^{\mathcal {R}}_p(\Ho_q(P_{\bullet}),M) \Rightarrow
\mathbb{T}\mathrm{or}^{\mathcal {R}}_{p+q}(P_{\bullet},M),\\
E^2_{p,q}(Q)=\Tor^{\mathcal {R}}_p(\Ho_q(Q_{\bullet}),M) \Rightarrow
\mathbb{T}\mathrm{or}^{\mathcal {R}}_{p+q}(Q_{\bullet},M).
\end{gather*}
$E^2_{p,q}(P)\cong E^2_{p,q}(Q)$ for all $p,q$, as
$\Ho_q(P_{\bullet})\cong \Ho_q(Q_{\bullet})$. It yields  that
$\mathbb{T}\mathrm{or}^{\mathcal {R}}_{n}(P_{\bullet},M)\cong
\mathbb{T}\mathrm{or}^{\mathcal {R}}_{n}(Q_{\bullet},M)$ by using
the mapping lemma for $E^{\infty}$ (see \cite{W}).
\end{proof}

The above two lemmas still hold in the graded module category.
\begin{coro}[\cite{GJ}]\label{3.4}
If there exists $f:C\ra C'$ is a quasi-isomorphic of mixed
complexes, then for any graded $k[u]$-module $W$, we have an
isomorphism of cyclic homology groups $$\Hc_{\bullet}(C;W)\cong
\Hc_{\bullet}(C';W).$$
\end{coro}

Using the generalized Eilenberg-Zilber theorem for paracyclic
modules, we have

\begin{coro}
\label{3.2}  Let $A\#_{_R}B$ be a strong smash product algebra,
$A\natural B$ be the cylindrical module defined in \eqref{7} and
\eqref{8}.
 Then $$\Hc_*(A\#_{_R}B;W)\cong \Hc_*(\Delta(A\natural B);W)\cong\Hc_*(\Tot(A\natural B);W).$$
\end{coro}

The following corollary will be used in the next section.
\begin{coro}\label{coro 3.5}
Let $(\mathfrak{C},\mathfrak{d})$ be a complex of $k$-modules and
$W$ a graded $k[u]$-module. Then for each $n$,
$$\Ho_n(\mathfrak{C}[[u]]\ot_{k[u]}W)=\Ho_n(\mathfrak{C})[[u]]\ot_{k[u]}
W.$$
\end{coro}
\begin{proof}Since $\mathfrak{C}[[u]]$ is a complex of flat
$k[u]$-modules,
$$\Ho_n(\mathfrak{C}[[u]]\ot_{k[u]}W)=\mathbb{T}\mathrm{or}^{k[u]}_n(\mathfrak{C}[[u]],W).$$
Note that the differential of the complex $\mathfrak{C}[[u]]$ does
not depend on $u$. We have a spectral sequence converging to the
hypertor whose $E^2$-term is
$$
\Tor_p^{k[u]}(\Ho_q(\mathfrak{C})[[u]],W).
$$
Because $\Ho_q(\mathfrak{C})[[u]]$ is also a flat $k[u]$-module, the
spectral sequence collapses. We get
$$
\Ho_n(\mathfrak{C}[[u]]\ot_{k[u]}W)=\mathbb{T}\mathrm{or}^{k[u]}_n(\mathfrak{C}[[u]],W)=\Ho_n(\mathfrak{C})[[u]]\ot_{k[u]}W.
$$

This completes the proof.
\end{proof}

\section{\label{44}Cyclic homology of a strong smash product algebra}
We can also construct a spectral sequence to calculate the cyclic
homology of a strong smash product algebra $A\#_{_{R}}B$. This is
the same as calculating the cyclic homology of $\Tot(A\natural B)$.
The first column of the cylindrical module $A\natural B$ plays an
important role. Denote by $C_{\bullet}({}_A^{~\natural}B)$ this
paracyclic module $A\natural B(\bullet,0)$.

\begin{lemma}\label{4.1}For each $n\in \mathbb{N}$, $C_{n}({}_A^{~\natural}B)$ is an $A$-bimodule via
the left $A$-module action
$$a.(b_0,\ldots,b_n\,|\,a_0)=(id^{\ot (n+1)}\ot m_A)\big(\Gamma_n(a,b_0,\ldots,b_n)\,|\,a_0\big)$$
and the right $A$-module action
$$(b_0,\ldots,b_n\,|\,a_0).a=(b_0,\ldots,b_n\,|\,a_0a)$$
where $\Gamma_n $ is defined in Section \ref{222}, and $a_0,a\in
A,b_j\in B$.
\end{lemma}
\begin{proof} The right action is trivial. By Proposition \ref{1.2}, $R^{-1}$ is also
quasitriangular and normal, then the left $A$-module action is
well-defined. And both actions are compatible.
\end{proof}

For each $p\in \mathbb{N}$, we can define a \textit{Hochschild
complex} $(C_{\bullet}(A,C_{p}({}_A^{~\natural}B)),d)$, whose
homology is \textit{the Hochschild homology of the algebra} $A$ with
coefficients in $C_{p}({}_A^{~\natural}B)$ (see \cite{L}). The
Hochschild complex is defined explicitly as follows: for any $q\in
\mathbb{N}$,
$$C_{q}(A,C_{p}({}_A^{~\natural}B))=C_{p}({}_A^{~\natural}B)\ot A^{\ot q}=B^{\ot (p+1)}\ot A\ot A^{\ot q},$$
the differential  $d: C_{q}(A,C_{p}({}_A^{~\natural}B))\ra
C_{q-1}(A,C_{p}({}_A^{~\natural}B))$ is
\begin{equation}\label{9}\begin{split}
d(b_0,\ldots,b_{p}\,|\,a_0\,|\, a_1,\ldots,a_q)
&=\big((b_0,\ldots,b_{p}\,|\,a_0).a_1\,|\,a_2,\ldots,a_q\big)\\
&\ +\sum_{i=1}^{q-1}(-1)^{i}(b_0,\ldots,b_{p}\,|\,a_0\,|\, a_1,\ldots,a_ia_{i+1},\ldots,a_q)\\
 &\ +(-1)^{q}\big(a_q.(b_0,\ldots,b_{p}\,|\,a_0)\,|\,
 a_1,\ldots,a_{q-1}\big).\end{split}
\end{equation}
Denote this Hochschild homology by
$\Ho_{\bullet}(A,C_{p}({}_A^{~\natural}B))$.

%Compare the operators, we get
\begin{coro}
$C_{\bullet}(A,C_{\bullet}({}_A^{~\natural}B))$ is a cylindrical
module with the same operators defined for $A\natural B$ in
\eqref{7} and \eqref{8}.
\end{coro}
Indeed, for each $p,q\in \mathbb{N}$,
$$C_{q}(A,C_{p}({}_A^{~\natural}B))=C_{p}({}_A^{~\natural}B)\ot A^{\ot q}=
B^{\ot (p+1)}\ot A\ot A^{\ot q}=A\natural B(p,q).$$ Note that
$\mathrm{\bar{\,b}}$ of $A\natural B$ is exactly the differential
$d$ defined in \eqref{9}. \vspace{1.5em}

 Define $C^A_{\bullet}({}_A^{~\natural}B)$ as  the
\textit{co-invariant space} of $C_{\bullet}({}_A^{~\natural}B)$
under the left and right actions of $A$ constructed in Lemma
\ref{4.1}, i.e.,
$$C^A_{\bullet}({}_A^{~\natural}B) =C_{\bullet}({}_A^{~\natural}B)/
\mathrm{span}\{a.x-x.a\mid a\in A,x\in
C_{\bullet}({}_A^{~\natural}B)\}.$$ And we define the following
operators on $C^A_{\bullet}({}_A^{~\natural}B)$:
 \begin{equation}\label{10}\begin{split}
\tau_n(b_0,\ldots,b_n\,|\,a)&=\mathrm{f}^{n+1,1}(b_0,\ldots,b_{n-1},R(b_n\ot a)),\\
\partial_i(b_0,\ldots,b_n\,|\,a)&=(b_0,\ldots,b_ib_{i+1},\ldots,b_n\,|\,a),\quad \text{for }0\leq i<n,\\
\partial_n(b_0,\ldots,b_n\,|\,a)&=\partial_0\tau_n(b_0,\ldots,b_n\,|\,a),\\
\sigma_j(b_0,\ldots,b_n\,|\,a)&=(b_0,\ldots,b_j,1,\ldots,b_n\,|\,a),\quad
\text{for }0\leq j\leq n.
 \end{split}\end{equation}

Use the following notations as in \cite{CBMZ}, for $R$,$$R(b\ot
a)=a^R\ot b^R ,\quad R_{23}R_{12}(b\ot a_1\ot a_2)=a_1^{R_1}\ot
a_2^{R_2 } \ot b^{{R_1}^{\scriptstyle R_2}},\text{ etc,}$$ and for
$R^{-1}$,
$$R^{-1}(a\ot b)=b^r\ot a^r ,\quad R^{-1}_{23}R^{-1}_{12}(a\ot b_1\ot b_2)=b_1^{r_1}\ot
b_2^{r_2 } \ot a^{{r_1}^{\scriptstyle r_2}},\text{ etc,}$$
 where
$a,a_1,a_2\in A$ and $b,b_1,b_2\in B$. Then one can check that these
operators in \eqref{10} are well defined on the co-invariant space.
For example,
 \begin{align*}
 \tau(a.(b_0,\ldots,b_n\,|\,a_0))
&=\tau(b_0^{r_1},b_1^{r_2},\ldots,b_n^{r_{n+1}}
 \,|\,a^{{r_1}^{{\adots}^{\scriptstyle r_{n+1}}
}}a_0)\\
&=\big(b_{n}^{{r_{n+1}}^{\scriptstyle
R}},b_0^{r_1},b_1^{r_2},\ldots,b_{n-1}^{r_{n}}
 \,|\,(a^{{r_1}^{{\adots}^{\scriptstyle r_{n+1}}
}}a_0)^R\big)\\
&=\big(b_{n}^{{r_{n+1}}^{\scriptstyle {R_1}^{\scriptstyle
R_2}}},b_0^{r_1},b_1^{r_2},\ldots,b_{n-1}^{r_{n}}
 \,|\,a^{{r_1}^{{\adots}^{\scriptstyle {r_{n+1}}^{\scriptstyle R_1}}
}}a_0^{R_2}\big)\\
&=\big(b_{n}^{{\scriptstyle
R}},b_0^{r_1},b_1^{r_2},\ldots,b_{n-1}^{r_{n}}
 \,|\,a^{{r_1}^{{\adots}^{\scriptstyle r_{n}}
}}a_0^{R}\big)\\
&=a^{R_2}.\,(b^{{R_1}^{\scriptstyle R_2}},b_0,\ldots,b_{n-1}\,|\,a_0^{R_1})\quad,\\\\
 \tau((b_0,\ldots,b_n\,|\,a_0).a)&=\tau(b_0,\ldots,b_n\,|\,a_0a)\\
&=(b_n^R,b_0,\ldots,b_{n-1}\,|\,(a_0a)^R)\\
&=(b_n^{{R_1}^{\scriptstyle
R_2}},b_0,\ldots,b_{n-1}\,|\,{a_0}^{R_1}a^{R_2})\\
&=(b_n^{{R_1}^{\scriptstyle
R_2}},b_0,\ldots,b_{n-1}\,|\,{a_0}^{R_1}).\,a^{R_2}\quad.\\
\end{align*}

\begin{prop}
$C_{\bullet}^A({}_A^{~\natural}B)$ is a cyclic module with operators
defined in \eqref{10}.
\end{prop}
\begin{proof}We only check that $\tau_n^{n+1}=id$. The other
identities are similar to check. In the coinvariant subspace, we
have
\begin{align*}
\tau^{n+1}(b_0,\ldots,b_n\,|\,a)
&=\tau^n(b_n^R,b_0,\ldots,b_{n-1}\,|\,a^R)\\
&=\tau^{n-1}(b_{n-1}^{R_{2}},b_n^{R_1},b_0,\ldots,b_{n-2}\,|\,{a^{R_1}}^{R_2})=\cdots\\
&=(b_0^{R_{n+1}},\ldots,b_{n-1}^{R_{2}},b_n^{R_1}\,|\,{{a^{R_1}}^{R_2^{\adots}}}^{R_{n+1}})\\
&=(b_0^{R_{n+1}},\ldots,b_{n-1}^{R_{2}},b_n^{R_1}\,|\,1).\,{{a^{R_1}}^{R_2^{\adots}}}^{R_{n+1}}\\
&={{a^{R_1}}^{R_2^{\adots}}}^{R_{n+1}}.\,(b_0^{R_{n+1}},\ldots,b_{n-1}^{R_{2}},b_n^{R_1}\,|\,1)\\
&=(b_0,\ldots,b_{n}\,|\,a).
\end{align*}
\end{proof}
In fact, the above proposition is a special case of the following
theorem.
\begin{theorem}\label{h01}
For any $q\in\mathbb{N}$, $\Ho_q(A, C_{\bullet}({}_A^{~\natural}B))$
is a cyclic module with $(d_i,s_j,t)$ induced from operators of
$A\natural B$ defined in \eqref{7}. Especially, we have
 $$\Ho_0(A,C_{\bullet}({}_A^{~\natural}B))=C^A_{\bullet}({}_A^{~\natural}B).$$
\end{theorem}
\begin{proof}
We need to check that,  $t_{n,q}^{n+1}$ inducing on $\Ho_q(A,
C_{n}({}_A^{~\natural}B))$ turns out to be identity. For any $x\in
\Ho_q(A, C_{n}({}_A^{~\natural}B))$, $d(x)=0$, or equivalently,
$\mathrm{\bar{\,b}}(x)=0$,
\begin{align*}
(t_{n,q}^{n+1}-id)(x)&=(t_{n,q}^{n+1}-t_{n,q}^{n+1}\bar{\,t}_{n,q}^{~q+1})(x)=t_{n,q}^{n+1}(1-\bar{t}_{n,q}^{~q+1})(x)\\
&=t_{n,q}^{n+1}(\mathrm{\bar{\,b}\bar{~B}+\bar{~B}\bar{\,b}})(x)=\mathrm{\bar{\,b}\bar{~B}}\,t_{n,q}^{n+1}(x)
=0\in\Ho_q(A, C_{n}({}_A^{~\natural}B)).
\end{align*}
Since the barred operators commutate with the unbarred operators,
all unbarred operators $(d_i,s_j,t)$ are well-defined on $\Ho_q(A,
C_{\bullet}({}_A^{~\natural}B))$ preserving the relations \eqref{3}
and \eqref{5}.
\end{proof}

\begin{lemma}\label{4.3}
The homology group of the complex
$$\cdots\ra
C_{q}(A,C_{\bullet}({}_A^{~\natural}B))[[u]]\stackrel{d}\longrightarrow
C_{q-1}(A,C_{\bullet}({}_A^{~\natural}B))[[u]]\ra\cdots $$ is
$\Ho_q(A,C_{\bullet}({}_A^{~\natural}B))[[u]]$, for each $q$.
\end{lemma}

By Corollary \ref{3.2}, in order to calculate the cyclic homology of
the strong smash algebra $A\#_{_R}B$ with coefficients in $W$, we
can compute the cyclic homology of $\Tot(A\natural B)$ with
coefficients in $W$, that is, the homology of the complex
$$\big(\Tot(A\natural B)\boxtimes W,  (\mathrm{b+\!\bar{\,b}}+u\mathrm{B}+u\mathrm{T\!\bar{~B}})\ot
id\big).$$
We define a filtration on $\Tot(A\natural B)\boxtimes W$
by rows. Set
$$F^p_{n}\big(\Tot(A\natural B)\boxtimes W\big)
=\sum_{\begin{matrix}\scriptstyle i+j=n+2l,\\
\scriptstyle i\leq p+2l,\\\scriptstyle l\geq 0
\end{matrix}}(B^{\ot(i+1)}\ot A^{\ot (j+1)})u^l\ot_{k[u]}W, \text{ for }p\geq
0; $$ and $F^p_{n}\big(\Tot(A\natural B)\boxtimes W\big) =0, \text{
for }p< 0.$

The spectral sequence  $E^{r}_{p,q}$ of this  filtration  with
$d^r:E^r_{p,q}\ra E^r_{p-r,q+r-1}$ starts from $$E_{p,q}^0
=\sum_{l\geq 0}(B^{\ot(p+2l+1)}\ot A^{\ot (q+1)})u^l\ot_{k[u]}W,
$$
equipped with $d^0=\mathrm{\bar{\,b}}\ot id:E^0_{p,q}\ra
E^0_{p,q-1}$.

Recall that
$C_{q}(A,C_{\bullet}({}_A^{~\natural}B))[[u]]=\sum_{p,l\geq
0}(B^{\ot (p+2l+1)}\ot A^{\ot (q+1)})u^l$,
 $$E_{\bullet,q}^0 =C_{q}(A,C_{\bullet}({}_A^{~\natural}B))\boxtimes W.$$

So from Lemma \ref{4.3}  and Corollary \ref{coro 3.5}, we get:
\begin{lemma} The $E^1$-term of the spectral sequence is
$$E^1_{\bullet,q}= \Ho_q(A,C_{\bullet}({}_A^{~\natural}B))\boxtimes W,$$
equipped with $d^1:E^1_{p,q}\ra E^1_{p-1,q}$ that is induced by
$(\mathrm{b}+u\mathrm{B})\ot id$.
\end{lemma}

\begin{theorem}\label{lim1}
The $E^2$-term of the spectral sequence is identified with the
cyclic homology
 of the cyclic module $\Ho_{\bullet}\big(A,C_{\bullet}({}_A^{~\natural}B)\big)$
  with coefficients in $W$. It converges to the cyclic homology of
  the strong smash product algebra $A\#_{_R}B$ with coefficients in $W$. That is,
$$
E^2_{p,q}=\Hc_p\Big(\Ho_q\big(A,C_{\bullet}({}_A^{~\natural}B)\big); W\Big)\Rightarrow \Hc_{p+q}(A\#_{_R}B;W).
$$
\end{theorem}

In parallel, one can also consider the bottom row of the cylindrical
module $A\natural B$. We just state the process and indicate the
differences here. We skip proofs which are similar as in previous
discussions.

Denote the paracyclic module $A\natural B(0,\bullet)$ by
$C_{\bullet}(A_B^{\natural})$.
 \begin{lemma} \label{4.7}For each $n\in \mathbb{N}$,
$C_{n}(A_B^{\natural})$ is a $B$-bimodule via the left $B$-module
action
$$b.( b_0\,|\,a_0,\ldots,a_n)=( bb_0\,|\,a_0,\ldots,a_n)$$
and the right $B$-module action
$$( b_0\,|\,a_0,\ldots,a_n).b=(m_B\ot id^{\ot (n+1)})\big(b_0\,|\,
\Theta_n^{-1}(a_0,\ldots,a_n,b)\big),$$ where $\Theta_n$ is defined
in Section \ref{222}, and $a_i\in A, b_0, b\in B$.
\end{lemma}

For each $q\in \mathbb{N}$,  a \textit{Hochschild complex}
$(C_{\bullet}(B,C_{q}(A^{\natural}_B)),\de)$ can be defined, its
homology is \textit{the Hochschild homology of the algebra} $B$ with
coefficients in $C_{q}(A^{\natural}_B)$. The Hochschild complex is
defined explicitly as follows: For any $p\in \mathbb{N}$,
$$
C_{p}(B,C_{q}(A^{\natural}_B))=C_{q}(A^{\natural}_B)\ot B^{\ot p}
=B\ot A^{\ot (q+1)}\ot B^{\ot p},
$$
the differential  $\de:
C_{p}(B,C_{q}(A^{\natural}_B))\ra C_{p-1}(B,C_{q}(A^{\natural}_B))$
is
\begin{equation}\label{11}
\begin{split}
\de(b_0\,|\,a_0,\ldots,a_q\,|\,b_1,\ldots,b_{p})
&=\big((b_0\,|\,a_0,\ldots,a_q).b_1\,|\,b_2,\ldots,b_{p}\big)\\
&\ +\sum_{i=1}^{p-1}(-1)^{i}(b_0\,|\,a_0,\ldots,a_q\,|\,b_1,\ldots,b_ib_{i+1},\ldots,b_{p})\\
&\ +(-1)^{p}\big(b_p.(b_0\,|\,a_0,\ldots,a_q)\,|\,
 b_1,\ldots,b_{p-1}\big).
 \end{split}
\end{equation}
Denote this Hochschild homology by
$\Ho_{\bullet}(B,C_{q}(A^{\natural}_B))$.

A difference occurs here, as the positions of $A$'s and $B$'s are
changed.
\begin{coro}
$C_{\bullet}(B,C_{\bullet}(A_B^{\natural}))$ is a cylindrical
module, which is isomorphic to  $A\natural B$.
\end{coro}
\begin{proof}
We give the isomorphisms between
$C_{\bullet}(B,C_{\bullet}(A_B^{\natural}))$ and $A\natural B$, then
the bi-paracyclic operators on
$C_{\bullet}(B,C_{\bullet}(A_B^{\natural}))$ are constructed from
the operators on $A\natural B$ through the isomorphisms. In this
way, we get the corollary. We give the isomorphisms and the
operators on $C_{\bullet}(B,C_{\bullet}(A_B^{\natural}))$
explicitly. For each $p, q\in\mathbb{N}$,
$$
\phi_{p,q}:C_{p}(B,C_{q}(A_B^{\natural}))\longrightarrow A\natural B(p,q)
$$
$$
\phi_{p,q}=(id^{\ot
p}\ot \Theta_q^{-1})(id^{\ot (p-1)}\ot \Theta_q^{-1}\ot id)\cdots
(id\ot \Theta_q^{-1}\ot id ^{\ot (p-1)}),
$$ and
$$
\psi_{p,q}: A\natural B(p,q) \longrightarrow  C_{p}(B,C_{q}(A_B^{\natural}))
$$
$$
\psi_{p,q}=\phi_{p,q}^{-1}=(id\ot
\Theta_q\ot id ^{\ot (p-1)})(id^{\ot 2}\ot\Theta_q\ot id ^{\ot
(p-2)})\cdots(id^{\ot p}\ot \Theta_q).
$$
We can see that
$$\phi\de=\mathrm{b}\phi,\quad\psi\mathrm{b}=\de\psi.$$
Hence,  the operators
$(\mathtt{d}_i,\mathtt{s}_j,\mathtt{t},\mathtt{\bar{d}}_i,
\mathtt{\bar{s}}_j,\mathtt{\bar{t}})$ on
$C_{p}(B,C_{q}(A_B^{\natural}))$ are defined as follows:
\begin{equation}\begin{split}\label{12}
\mathtt{d}_i^{p,q}=\psi_{p-1,q}d_i^{p,q}\phi_{p,q},&\quad
 \mathtt{\bar{d}}_i^{p,q}=\psi_{p,q-1}\bar{d}_i^{p,q}\phi_{p,q},\\
\mathtt{s}_j^{p,q}=\psi_{p+1,q}s_j^{p,q}\phi_{p,q},&\quad
\mathtt{\bar{s}}_j^{p,q}=\psi_{p,q+1}\bar{s}_j^{p,q}\phi_{p,q},\\
\mathtt{t}_{p,q}=\psi_{p,q}{t}_{p,q}\phi_{p,q},&\quad
\mathtt{\bar{t}}_{p,q}=\psi_{p,q}{\bar{t}}_{p,q}\phi_{p,q}.
\end{split}
\end{equation}
\end{proof}

Define $C^B_{\bullet}(A_B^{\natural})$ as  the \textit{co-invariant
space} of $C_{\bullet}(A_B^{\natural})$ under the left and right
actions given in Lemma \ref{4.7}, i.e.,
$$
C^B_{\bullet}(A_B^{\natural}) =C_{\bullet}(A_B^{\natural})/
\mathrm{span}\{b.x-x.b\mid b\in B,x\in
C_{\bullet}(A_B^{\natural})\}.
$$
And we define the following
operators on $C^B_{\bullet}(A_B^{\natural})$:
 \begin{equation}\label{13}
 \begin{split}
\tau'_n(b\,|\,a_0,\ldots,a_n)&=(R^{-1}(a_n\ot b),a_0,\ldots,a_{n-1})\\
\partial'_i(b\,|\,a_0,\ldots,a_n)&=(b\,|\,a_0,\ldots,a_ia_{i+1},\ldots,a_n),
\quad \text{for }0\leq i<n,\\
\partial'_n(b\,|\,a_0,\ldots,a_n)&=\partial'_0\tau'_n(b\,|\,a_0,\ldots,a_n),\\
\sigma'_j(b\,|\,a_0,\ldots,a_n)&=(b\,|\,a_0,\ldots,a_j,1,\ldots,a_n),\quad
\text{for }0\leq j\leq n.
 \end{split}
 \end{equation}
Indeed, $\tau'_n$ is induced by ${\bar{t}}_{0,n}$, as
$\psi_{p,q}=id$ and $\phi_{p,q}=id$ when $p$ is $0$. One can check
that these operators are well defined on the co-invariant space.

\begin{theorem}\label{h02}
For any $p\in\mathbb{N}$, $\Ho_p(B, C_{\bullet}(A_B^{\natural}))$ is
a cyclic module with
$(\mathtt{\bar{d}}_i,\mathtt{\bar{s}}_j,\mathtt{\bar{t}})$ induced
from operators of $C_{\bullet}(B,C_{\bullet}(A_B^{\natural}))$
defined in \eqref{12}. Especially, we have
 $$\Ho_0(B,C_{\bullet}(A_B^{\natural}))=C^B_{\bullet}(A_B^{\natural})$$is a cyclic module with operators
defined in \eqref{13}.
\end{theorem}

We define a filtration on $\Tot(A\natural B)\boxtimes W$ by columns.
Set
$$\widetilde{F}^q_n\big(\Tot(A\natural B)\boxtimes W\big)
=\sum_{\begin{matrix}\scriptstyle i+j=n+2l,\\
\scriptstyle j\leq q+2l,\\\scriptstyle l\geq 0\end{matrix}} (B^{\ot
(i+1)}\ot A^{\ot (j+1) })u^l\ot_{k[u]}W,\text{ for }q\geq 0;$$ and
$\widetilde{F}^q_n\big(\Tot(A\natural B)\boxtimes W\big) =0$ for
$q<0$.

 The
spectral sequence $\widetilde{E}^{r}_{q,p}$ of this filtration
 with $\tilde{d}^r:\widetilde{E}^r_{q,p}\ra
\widetilde{E}^r_{q-r,p+r-1}$ starts from
\begin{align*}
\widetilde{E}_{q,p}^0 &=\sum_{l\geq 0} (B^{\ot (p+1)}\ot A^{\ot
(q+2l+1) })u^l\ot_{k[u]}W,
\end{align*}
equipped with $\tilde{d}^0=\mathrm{b}\ot
id:\widetilde{E}^0_{q,p}\longrightarrow \widetilde{E}^0_{q,p-1}$.

\begin{lemma}The $E^1$-term of the spectral sequence is
$$\widetilde{E}^1_{\bullet,p}=\Ho_p(B,C_{\bullet}(A_B^{\natural}))\boxtimes W,$$
equipped with $\tilde{d}^1:\widetilde{E}^1_{q,p}\longrightarrow
\widetilde{E}^1_{q-1,p}$ that is induced by
$(\mathrm{\bar{\,b}}+u\mathrm{\bar{\,B}})\ot id$.
\end{lemma}

\begin{theorem}\label{lim2}
 The $E^2$-term of the spectral sequence is identified with the cyclic homology
 of the cyclic module $\Ho_{\bullet}\big(B,C_{\bullet}(A_B^{\natural})\big)$
  with coefficients in $W$. It converges to the cyclic homology of
  the strong smash product algebra $A\#_{_R}B$ with coefficients in $W$.
  That is,
$$\widetilde{E}^2_{q,p}
\cong\Hc_q\Big(\Ho_p\big(B,C_{\bullet}(A_B^{\natural})\big);
W\Big)\Rightarrow \Hc_{p+q}(A\#_{_R}B;W).$$
\end{theorem}

 By  Proposition 1.1.13 of \cite{L}, we
can use the derived functor $\Tor$ to express the Hochschild
homology of an algebra $A$ with coefficients in $M$ which is an
$A$-bimodule, that is,
$$\Ho_n(A,M)\cong \Tor_n^{A^e}(M,A),$$ where $A^{e}=A\ot A^{op}$.
For a separable algebra that is projective over its enveloping
algebra, its homology with coefficients in any module is zero.
Hence, the spectral sequence collapses at $E^2$, that is,
$E^2_{p,q}=0$ for all $p,q$ unless $q=0$. So we have
\begin{coro}
If the algebra $A$ $($resp. $B$$)$ is separable, then there is a
natural isomorphism of cyclic homology groups
\begin{gather*}
\Hc_n(A\#_{_R}B;W)\cong \Hc_n\big(C^A_{\bullet}({}_A^{~\natural}B);
W\big),\\
 \text{$($resp.,} ~~\Hc_n(A\#_{_R}B;W)\cong
\Hc_{n}\big(C^B_{\bullet}(A_B^{\natural});W\big)~\text{$)$.}
\end{gather*}
\end{coro}

From the above results, one can observe that our theorems take
advantage of good homological property of either of two subalgebras.
Even in the case of the crossed product algebra $A\rtimes H$, where
$H$ is a Hopf algebra with invertible antipode and $A$ is an
$H$-module algebra, the ``nice'' homological property of $A$
sometimes will play a key role in computing the cyclic homology of
the crossed product, by comparison with the homological property of
$H$ being weak. We will illustrate this point by examples in the
next section.

\section{\label{55}Examples}
 \noindent \textbf{5.1} \ In this subsection, we apply our theorems to Majid's double
crossproduct of Hopf algebras which is inspired by  bismash product
of groups defined by Takeuchi \cite{T}. Bismash product of groups
 is a generalization of semiproduct of groups. In order to
define this product, he provided the notion of a matched pair of
groups. Given a matched pair of groups $(G,K)$, the bismash product
of $G$ and $K$ denoted by $G\bowtie K$ is still a group.

 The theory is developed by Majid \cite{Ma}. He defined a matched pair
 of Hopf algebras and constructed a product Hopf algebra which he called a double
crossproduct of Hopf algebras. Using this new definition he provided
another way to construct Drinfeld's quantum double. We start by
recalling the definition due to Majid \cite{Ma}.

\begin{defi}[\cite{Ma}]  A pair $(B,H)$ of Hopf algebras is said to
be \textit{matched} if $B$ is a left $H$-module coalgebra via
$\alpha$, and $H$ is a right $B$-module coalgebra via $\beta$,
$$\alpha:H\ot B\ra B,\quad \alpha(h\ot b)=h\rhd b,
\quad \beta:H\ot B\ra H,\quad\beta(h\ot b)=h\lhd b,$$ such that
 the following equalities hold for $\forall ~b,c\in B, h,g\in H$.
\begin{align}
h\rhd 1_{_B}=\ep_{_H}(h)1_{_B},&\quad h\rhd(bc)=\sum
\Big(h_{(1)}\rhd
b_{(1)}\Big)\Big((h_{(2)}\lhd b_{(2)})\rhd c\Big)\label{m1}\\
1_{_H}\lhd b=1_{_H}\ep_{_B}(b),&\quad (hg)\lhd b=\sum \Big(h\lhd
(g_{(1)}\rhd b_{(1)})\Big)\Big(g_{(2)}\lhd b_{(2)}\Big)\label{m2}\\
\sum h_{(1)}\lhd b_{(1)}\ot & h_{(2)}\rhd b_{(2)}= h_{(2)}\lhd
b_{(2)}\ot h_{(1)}\rhd b_{(1)}\label{m3}.
\end{align}
The \textit{double crossproduct}  $B\bowtie H$ is a Hopf algebra
equipped with
\begin{align*}
(b\ot h)(c\ot g)&=b(h_{(1)}\rhd c_{(1)})\ot (h_{(2)}\lhd c_{(2)})g\\
\Delta(b\ot h)&=b_{(1)}\ot h_{(1)}\ot b_{(2)}\ot h_{(2)}\\
\ep(b\ot h)&=\ep_{_{B}}(b)\ep_{_{H}}(h)\\
S(b\ot h)&=(1\ot S_{_H}h)(S_{_B}b\ot 1).
\end{align*}
\end{defi}

The double crossproduct of Hopf algebras relates closely to the
smash product algebra.

\begin{prop}\label{5.1}
Let $(B,H)$ be  a matched pair of Hopf algebras. If  $H$ and $B$
have invertible antipodes, then the double crossproduct of $B$ and
$H$ denoted by $B\bowtie H$ is a strong smash product algebra. In
particular, the group algebra of the bismash product of a matched
pair of groups is a strong smash product algebra.
\end{prop}

We  need the following lemma in the proof of Proposition \ref{5.1}.
\begin{lemma} Let $(B,H)$ be a matched pair of Hopf algebras. If
both $H$ and $B$ have invertible antipodes, then we have the
following identities:
\begin{align}
\label{m4}S_B^{-1}(h\rhd b)=(h\lhd b_{(2)})\rhd S_B^{-1}(b_{(1)}),\\
\label{m5}S_H^{-1}(h\lhd b)=S_H^{-1}(h_{(2)})\lhd (h_{(1)}\rhd b).
\end{align}
\end{lemma}
\begin{proof}
Since $(B,\eta,m,\Delta,\ep_{_B},S_{_B})$ is a Hopf algebra, then
$(B^{op},\eta,m^{op},\Delta,\ep_{_B},S_B^{-1})$ is also a Hopf
algebra. Denote the convolution map on $\Hom(B^{op},B^{op})$ by
$\ast'$. Define the operator $T\in \textrm{End}_{k}(H\rhd B)$ by
$$
T(h\rhd b):=(h\lhd b_{(2)})\rhd S_B^{-1}(b_{(1)}).
$$
We should
check that
\begin{align*}
(id\ast' T)(h\rhd b)=\ep_B(h\rhd b)1_B\quad \text{and} \quad (T\ast'
id)(h\rhd b)=\ep_B(h\rhd b)1_B.
\end{align*}
Actually, we only need to check the first equality. Indeed, if it
holds, then
\begin{align*}
T(h\rhd b)&=(\ep_B\ast'T)(h\rhd b)=((S_B^{-1}\ast' id)\ast' T)(h\rhd
b)\\&=(S_B^{-1}\ast' (id\ast' T))(h\rhd b)
=(S_B^{-1}\ast'\ep_B)(h\rhd b)=S_B^{-1}(h\rhd b).
\end{align*}
From \eqref{m1} and \eqref{m3}, we have
\begin{align*}
(id\ast' T)(h\rhd b)&=\big(h_{(2)}\rhd b_{(2)}\big)T(h_{(1)}\rhd
b_{(1)})\\
&=\big(h_{(2)}\rhd b_{(3)}\big)\big((h_{(1)}\lhd b_{(2)})\rhd
S_B^{-1}(b_{(1)})\big)\\
&=\big(h_{(1)}\rhd b_{(2)}\big)\big((h_{(2)}\lhd b_{(3)})\rhd
S_B^{-1}(b_{(1)})\big)\\
&=h\rhd (b_{(2)}S_B^{-1}(b_{(1)}))=\ep_B(b)h\rhd 1_B
=\ep_H(h)\ep_{_B}(b)1_B\\
&=\ep_B(h\rhd b)1_B.
\end{align*}
Using the same method, we can prove \eqref{m5}.\end{proof}
\begin{proof}[Proof of Proposition \ref{5.1}] We need to construct
an isomorphism $R$ from $H\ot B$ to $ B\ot H$, which is
quasitriangular and normal. For $b\in B, h\in H$, set
$$
R(h\ot
b)=\sum h_{(1)}\rhd b_{(1)}\ot  h_{(2)}\lhd b_{(2)}.
$$

1) $R$ is quasitriangular: $\forall~h,g\in H,b\in B$,
\begin{align*}
(id\ot m)R_{12}R_{23}(h\ot g\ot b)&=\sum (id\ot m)R_{12}\big(h\ot
g_{(1)}\rhd b_{(1)}\ot g_{(2)}\lhd b_{(2)}\big)\\
&=h_{(1)}\rhd (g_{(1)}\rhd b_{(1)})_{(1)}\ot \Big( h_{(2)}\lhd
(g_{(1)}\rhd b_{(1)})_{(2)}\Big)\Big(g_{(2)}\lhd b_{(2)}\Big)\\
&=h_{(1)}\rhd (g_{(1)}\rhd b_{(1)})\ot \Big( h_{(2)}\lhd
(g_{(2)}\rhd b_{(2)})\Big)\Big(g_{(3)}\lhd b_{(3)}\Big)\\
&=\sum (h_{(1)}g_{(1)})\rhd b_{(1)}\ot (h_{(2)}g_{(2)})\lhd
b_{(2)}=R(hg\ot b).
\end{align*}
The third equality holds due to the $H$-module coalgebra structure
of $B$, the forth equality holds because of \eqref{m2}. Similarly,
one can prove that $R\circ (id\ot m)=(m\ot id)R_{23}R_{12}$.

2) $R$ is normal: $\forall~ h\in H$,
\begin{align*}
R(h\ot 1_{_B})=\sum h_{(1)}\rhd 1_{_B}\ot  h_{(2)}\lhd
1_{_B}=\ep(h_{(1)})1_{_B}\ot h_{(2)}=1_{_B}\ot h.
\end{align*}
Similarly, one can prove that $R(1_{_H}\ot b)=b\ot 1_{_H}$.

3) $R$ is invertible: For $\forall ~b\in B, h\in H$,  set $r: B\ot
H\ra H\ot B$
$$
r(b\ot h):=h_{(3)}\lhd \big(S_H^{-1}(h_{(2)})\rhd
S_B^{-1}(b_{(3)})\big) \ot \big(S_H^{-1}(h_{(1)})\lhd
S_B^{-1}(b_{(2)})\big)\rhd b_{(1)}.
$$
We claim that $r$ is the inverse of $R$.
\begin{align*}
R\circ r(b\ot h)&=R\big(h_{(3)}\lhd \big(S_H^{-1}(h_{(2)})\rhd
S_B^{-1}(b_{(3)})\big)
\ot \big(S_H^{-1}(h_{(1)})\lhd S_B^{-1}(b_{(2)})\big)\rhd b_{(1)}\big)\\
&=\Big(h_{(3)}\lhd \big(S_H^{-1}(h_{(2)})\rhd
S_B^{-1}(b_{(3)})\big)\Big)_{(1)} \rhd
\Big(\big(S_H^{-1}(h_{(1)})\lhd S_B^{-1}(b_{(2)})\big)\rhd
b_{(1)}\Big)_{(1)}\\
&~\ot\Big(h_{(3)}\lhd \big(S_H^{-1}(h_{(2)})\rhd
S_B^{-1}(b_{(3)})\big)\Big)_{(2)}
\lhd \Big(\big(S_H^{-1}(h_{(1)})\lhd S_B^{-1}(b_{(2)})\big)\rhd b_{(1)}\Big)_{(2)}\\
&=\Big(h_{(5)}\lhd \big(S_H^{-1}(h_{(4)})\rhd
S_B^{-1}(b_{(6)})\big)\Big) \rhd \Big(\big(S_H^{-1}(h_{(2)})\lhd
S_B^{-1}(b_{(4)})\big)\rhd
b_{(1)}\Big)\\
&~\ot\Big(h_{(6)}\lhd \big(S_H^{-1}(h_{(3)})\rhd
S_B^{-1}(b_{(5)})\big)\Big)
\lhd \Big(\big(S_H^{-1}(h_{(1)})\lhd S_B^{-1}(b_{(3)})\big)\rhd b_{(2)}\Big)\\
&=\Big(\big(h_{(5)}\lhd (S_H^{-1}(h_{(4)})\rhd
S_B^{-1}(b_{(6)}))\big) (S_H^{-1}(h_{(2)})\lhd
S_B^{-1}(b_{(4)}))\Big)\rhd
b_{(1)}\\
&~\ot h_{(6)}\lhd \Big( (S_H^{-1}(h_{(3)})\rhd S_B^{-1}(b_{(5)}))
\big((S_H^{-1}(h_{(1)})\lhd S_B^{-1}(b_{(3)}))\rhd b_{(2)}\big)\Big)\\
&=\Big(\big(h_{(5)}\lhd (S_H^{-1}(h_{(4)})\rhd
S_B^{-1}(b_{(6)}))\big)
  (S_H^{-1}(h_{(3)})\lhd S_B^{-1}(b_{(5)}))\Big)\rhd
b_{(1)}\\
&~\ot h_{(6)}\lhd \Big((S_H^{-1}(h_{(2)})\rhd S_B^{-1}(b_{(4)}))
 \big((S_H^{-1}(h_{(1)})\lhd S_B^{-1}(b_{(3)}))\rhd b_{(2)}\big)\Big)\\
 &=\Big((h_{(3)}S_H^{-1}(h_{(2)}))\lhd S_B^{-1}(b_{(4)})\Big)\rhd b_{(1)}\ot
 h_{(4)}\lhd\Big(S_H^{-1}(h_{(1)})\rhd (S_B^{-1}(b_{(3)})b_{(2)})\Big)\\
 &=(1_H\lhd S_B^{-1}(b_{(2)}))\rhd b_{(1)}\ot h_{(2)}\lhd(S_H^{-1}(h_{(1)})\rhd
 1_B)\\
 &=b\ot h.
 \end{align*}
The third and the forth equalities above are due to the $B$-module
coalgebra structure of $H$ and the $H$-module coalgebra structure of
$B$. The fifth equality is due to \eqref{m3}. The sixth and the last
equalities hold because of \eqref{m1} and \eqref{m2}.
\begin{align*}
r\circ R(h\ot b)&=r(\sum h_{(1)}\rhd b_{(1)}\ot  h_{(2)}\lhd
b_{(2)})\\
&=(h_{(2)}\lhd b_{(2)})_{(3)}\lhd \Big(S_H^{-1}\big((h_{(2)}\lhd
b_{(2)})_{(2)}\big)\rhd S_B^{-1}\big((h_{(1)}\rhd b_{(1)})_{(3)}\big)\Big) \\
&~\ot \Big(S_H^{-1}\big((h_{(2)}\lhd b_{(2)})_{(1)}\big)\lhd
S_B^{-1}\big((h_{(1)}\rhd
b_{(1)})_{(2)}\big)\Big)\rhd (h_{(1)}\rhd b_{(1)})_{(1)}\\
&=h_{(6)}\lhd  b_{(6)}\Big(S_H^{-1}(h_{(5)}\lhd
b_{(5)})\rhd S_B^{-1}(h_{(3)}\rhd b_{(3)})\Big) \\
&~\ot \Big(S_H^{-1}(h_{(4)}\lhd b_{(4)})\lhd S_B^{-1}(h_{(2)}\rhd
b_{(2)})\Big)h_{(1)} \rhd b_{(1)}\\
&=h_{(6)}\lhd  b_{(6)}\Big(S_H^{-1}(h_{(5)}\lhd
b_{(5)})\rhd S_B^{-1}( h_{(4)}\rhd b_{(4)})\Big) \\
&~\ot \Big(S_H^{-1}( h_{(3)}\lhd b_{(3)}  )\lhd S_B^{-1}(h_{(2)}\rhd
b_{(2)})\Big)h_{(1)} \rhd b_{(1)}\\
&=h_{(7)}\lhd b_{(7)}\Big(S_H^{-1}(h_{(6)}\lhd b_{(6)})(h_{(5)}\lhd
b_{(5)})\rhd S_B^{-1}(b_{(4)})\Big)\\
&~\ot\Big( S_H^{-1}(h_{(4)})\lhd (h_{(3)}\rhd
b_{(3)})S_B^{-1}(h_{(2)}\rhd b_{(2)})\Big)h_{(1)} \rhd
b_{(1)}\\
&=h_{(3)}\lhd b_{(3)}S_B^{-1}(b_{(2)})\ot
S_H^{-1}(h_{(2)})h_{(1)} \rhd b_{(1)}\\
&=h\ot b.
\end{align*}
The fifth equality above holds because of \eqref{m4} and \eqref{m5}.

For $(G,K)$ a matched pair of groups, since $k[G\bowtie
K]=k[G]\bowtie k[K]$, it is a strong smash product algebra.
\end{proof}

Therefore, thanks to Theorems \ref{lim1} and  \ref{lim2}, we have
\begin{coro}\label{bh}
Given a matched pair of Hopf algebra $(B,H)$, if each of $B$ and $H$
has an invertible antipode, then
\begin{gather*}
\Hc_p\Big(\Ho_q\big(B,C_{\bullet}({}_B^{~\natural}H)\big);W\Big)\Rightarrow
\Hc_{p+q}(B\bowtie H;W),\\
\Hc_q\Big(\Ho_p\big(H,C_{\bullet}(B_H^{\natural})\big);W\Big)\Rightarrow
\Hc_{p+q}(B\bowtie H;W).
\end{gather*}
\end{coro}

For a finite group $G$ and an arbitrary $G$-bimodule $M$, since
$k[G]$ is semisimple, $\Ho_n(k[G],M)$ is $0$  for all $n$ except for
$n=0$. Then by the above corollary, Theorems \ref{h01} and
\ref{h02}, we have
\begin{coro}\label{gk}
Given a matched pair of finite groups $(G,K)$, then
\begin{gather*}
\Hc_{n}(k[G\bowtie K];W)\cong
\Hc_n(C^{G}_{\bullet}({}_{G}^{~\natural}K);W),\\
\Hc_{n}(k[G\bowtie K];W)\cong
\Hc_n(C^{K}_{\bullet}({G}_{K}^{\natural});W).
\end{gather*}
\end{coro}

If $H$ is a finite dimensional Hopf algebra, then the antipode of
$H$ is always invertible (see, Corollary 5.6.1 of \cite{S}). Using
the adjoint action of $H$ on itself, Majid in Example 4.6 of
\cite{Ma} constructed  a matched pair $(H, H^{*cop})$ and deduced
the Drinfeld's quantum double $D(H)=H^{*cop}\bowtie H$. By
Corollaries \ref{bh} and \ref{gk}, we have
\begin{coro}
If $H$ is a finite dimensional Hopf algebra, then
$$
\Hc_q\Big(\Ho_p\big(H,C_{\bullet}
({{H^{*cop}}_{_H}}^{\!\!\natural})\big);W\Big)\Rightarrow
\Hc_{p+q}(D(H);W).
$$
If moreover, $H$ is semisimple $($equivalently,
there is an integral $t\in H$ with $\ep(t)=1$$)$, then
$$
\Hc_{n}(D(H);W)\cong \Hc_n\big(C^{H}_{\bullet}({{H^{*cop}}_{_H}}^{\!\!\natural});W\big).
$$
\end{coro}

\begin{remark} Actually, any Drinfeld's quantum double turns out to be of
Majid's double crossproduct structure (see \cite{M}), while the
recently appeared attractive objects, such as the two-parameter or
the multiparameter (restricted) quantum (affine) groups, the pointed
Hopf algebras arising from Nichols algebras of diagonal type (cf.
\cite{BW, BGH, HP, HRZ, HW, PHR,  ARS, AS, H} and references
therein), are of Drinfeld's double structures (under certain
conditions for the root of unity cases). Thereby, our machinery
established for the strong smash product algebras is indeed suitable
to a large class of many interesting Hopf algebras.
\end{remark}

\noindent \textbf{5.2} \ The following first example comes from the
rank $1$ case (modified) of the smash product algebra
$\mathscr{A}_q\#\mathscr{D}_q$ introduced in \cite{Hu} (p.525,
subsection {\bf 3.5}), which was used to define intrinsically and
construct a quantum Weyl algebra $\mathcal W_q(2n)$. Although our
example here is still a crossed product algebra, Proposition 5.3 in
\cite{AK}, under the assumption of the Hopf algebra $H$ being
semisimple (so automatically finite dimensional), does not work for
our example.

\begin{example}
Let $q\in k$ be an $N$-th primitive root of unity. Define
$\mathscr{A}$ to be $k[x]/(x^N{-}1)$ which is isomorphic to the
group algebra $k[\mathbb{Z}/N\mathbb{Z}] $. Define $\mathscr{D}$ to
be the associative $k$-algebra generated by $\pa$, $\sg^{\pm1}$,
subject to relation $\sg^{-1} \pa \sg=q\pa$. $\mathscr{D}$ is a Hopf
algebra with the coproduct, counit, and antipode defined as follows:
$$
\De(\partial)=\partial\ot 1+\sg\ot \partial, \quad \De(\sg)=\sg\ot
\sg,\quad \ep(\partial)=0,\quad \ep(\sg)=1,
$$
$$
S(\partial)=-\sg^{-1}\partial, \quad S(\sg)=\sg^{-1}.
$$
The antipode of $\mathscr{D}$ is invertible, as $S^{2N}=id$. One can
calculate that $S^{-1}(\partial)=-\partial \sg^{-1}$ and $
S^{-1}(\sg)=\sg^{-1}$. Let $(n)_q=1+q+\cdots +q^{n-1}$ for
$0<n\in\mathbb{N}$. Then $(N)_q=0$.
\end{example}
\begin{lemma}$(\cite{Hu})$
$\mathscr{A}$ is a $\mathscr{D}$-module algebra via $\partial.1=0$,
$\sg.1=1$, and for $n>0$,
$$
\partial.\,x^n=(n)_q\,x^{n-1},\quad\sg.\,x^n=q^nx^n.
$$
\end{lemma}
\begin{proof}
We should first check that $\mathscr{A}$ is a $\mathscr{D}$-module.
Indeed, if $n>0$,
$$(\sg^{-1}\partial\sg).\,x^n=q(n)_q\,x^{n-1}=q\partial.\,x^{n},$$
$$\partial.\,x^N=(N)_q\,x^{N-1}=0=\partial.\,1, \text{ and }
\sg.\,x^N=q^Nx^N=1.$$

From direct calculation, we get
$$(\partial.\,x^i)x^j+(\sg.\,x^i)(\partial.\,x^j)=(i{+}j)_q\,x^{i+j-1}=\partial.\,(x^ix^j).$$
This completes the proof.\end{proof}

 As we stated in Example
\ref{eg ch1.5}, the crossed product $\mathscr{A}\rtimes \mathscr{D}$
is a strong smash product algebra, since $\mathscr{D}$ is a Hopf
algebra with invertible antipode. For example, $R(\partial\ot
x^n)=(n)_qx^{n-1}\ot 1+q^nx^n\ot \partial$, $R(\sg\ot x^n)=q^nx^n\ot
\sg$ and $R^{-1}(x^n\ot \partial)=q^{-n}\partial \ot
x^n-q^{-n}(n)_q\ot x^{n-1}$ for $n>0$.

Since $\mathscr{A}$ is a group algebra of  a finite group, then it
is a semisimple  Hopf algebra, so the spectral sequence collapses
and we have
\begin{coro}
$$
\Hc_n(\mathscr{A}\rtimes \mathscr{D};W) \cong
\Hc_n\big(C^{\mathscr{A}}_{\bullet}({}_{\mathscr{A}}^{~\natural}{\mathscr{D}});
W\big).
$$
\end{coro}

\begin{example}
Let $\mathscr D_N$ be the quotient algebra of $\mathscr D$ by the
ideal $\langle \partial^N\rangle$. As $\langle \partial^N\rangle$ is
a Hopf ideal (owing to $\Delta(\partial^s)=\sum_{i=0}^s
\binom{s}{i}_q\,\sg^i\partial^{s-i}\otimes\partial^i$ and
$\Delta(\partial^N)=\partial^N\otimes 1+\sg^N\otimes \partial^N$),
$\mathscr D_N$ is a Hopf algebra. In particular, when $N=2$,
$\mathscr D_2$ is nothing but the Pareigis' Hopf algebra $\mathcal
P$ (see \cite{P} or the next subsection for definition).

 Furthermore, consider the quotient Hopf algebra $\bar{\mathscr D}_N$ of
 $\mathscr D_N$ by the Hopf ideal $\langle \sg^N{-}1\rangle$. $\bar
 {\mathscr D}_N$ is just the Taft algebra. Cyclic homology of the
 Taft algebra as a special truncated quiver algebra is computed by Taillefer \cite{Ta}.
\end{example}

\noindent \textbf{5.3} \ This subsection is devoted to effectively
computing the cyclic homology of the Pareigis' Hopf algebra
$\mathcal P$ using our theory.

In \cite{P}, Pareigis defined a noncommutative and noncocommutative
Hopf algebra $\P$, which links closely the category of complexes and
the category of comodules over $\P$. That is, the category of
complexes is equivalent as a tensor category to the category of
comodules over $\P$.  Explicitly,  $\P$ is defined to be the
quotient algebra of the free algebra $k\langle s,t,t^{-1}\rangle$ by
the two sided ideal that is generated by
$$tt^{-1}-1, \ t^{-1}t-1, \ s^2, \ st+ts.$$
Then $\P$ turns out to be a Hopf algebra with the following
coproduct, counit and  antipode,
\begin{gather*}
\Delta(t)=t\ot t,\quad  \ep(t)=1, \quad S(t)=t^{-1};\\
\Delta(s)=s\ot1+t^{-1}\ot s,\quad  \ep(s)=0, \quad S(s)=st.
\end{gather*}
$\P$ can be regarded as the crossed product algebra of $k[s]/s^2$
and $k[t,t^{-1}]$,  where $k[s]/s^2$ is a module algebra over
$k[t,t^{-1}]$ with the conjugate action $t.s=tst^{-1}=-s$. Denote by
$D$ the algebra of dual number $k[s]/s^2$, and by $T$ the Laurent
polynomial ring $k[t,t^{-1}]$.
$$\P\cong
D\rtimes T.$$ $\P$ is a strong smash product algebra $D\#_{_R}T$
with the invertible
 $R: T \ot D\ra D\ot T$ defined to be
$$R(t^r\ot s)=(-1)^r(s\ot t^r).$$

 Let $W=k[u]/uk[u]$. We would like to calculate the
Hochschild homology of $\P$ first. Consider the cyclic module
$E^1_{\bullet,q}=\Ho_q\big(D,C_{\bullet}({}_D^{~~\natural}T)\big)\cong
\Tor^{D^e}_q(D,C_{\bullet}({}_D^{~~\natural}T))$. Its face maps,
degeneracy maps, and cyclic operators are induced by the
corresponding operators defined in \eqref{7} for the cylindrical
module $D\natural T(\bullet,q)$.

Using the following resolution of $D$ by projective $D^e$-modules
 (see e.g., \cite{W})
$$\xymatrix{R_{\bullet}:~\cdots\ar[r]^-{\nu}&D^{e}\ar[r]^{\mu}&D^{e}\ar[r]^{\nu}&D^{e}
 \ar[r]^{\mu}&D^{ e}\ar[r]^{\mathrm{m}}&D\ar[r]& 0},$$
where $\mu=1\ot s-s\ot 1$, $\nu=1\ot s+s\ot 1$, $\mathrm{m}$ is the
product of $D$, we get
$$\Ho_q\big(D,C_{p}({}_D^{~~\natural}T)\big)\cong
\begin{cases}T^{\ot (p+1)}\ot 1\oplus T^{\ot(p+1)}_{ev}\ot s& \text{ for }q=0\\
T^{\ot (p+1)}_{ev}\ot 1\oplus T^{\ot(p+1)}_{od}\ot s&\text{ for
}q=2n-1>0\\T^{\ot (p+1)}_{od}\ot 1\oplus T^{\ot(p+1)}_{ev}\ot
s&\text{ for }q=2n>0\end{cases},$$ where $$T_{ev}^{\ot
(p+1)}:=k\{(t^{r_0},\ldots,t^{r_p})\mid r_0+\cdots +r_p\text{ is
even}\},$$ $$T_{od}^{\ot (p+1)}:=k\{(t^{r_0},\ldots,t^{r_p})\mid
r_0+\cdots +r_p\text{ is odd}\}.$$

In order to specify the operators of the cyclic module
$\Ho_q\big(D,C_{\bullet}({}_D^{~~\natural}T)\big)$, we should
represent the elements of
$\Ho_q\big(D,C_{\bullet}({}_D^{~~\natural}T)\big)$ by elements of
$D\natural T(\bullet,q)$. According to the Comparison Theorem, there
is a unique chain map lifting $id_D$ from the resolution
$R_{\bullet}$ to the bar resolution of $D$ up to chain homotopy
equivalence. This required chain map $\zeta_{\bullet}$ is defined as
follows,
$$\xymatrix{\cdots\ar[r]^{\nu}&D^{e}\ar[r]^{\mu}\ar[d]^{\zeta_3}
&D^{e}\ar[r]^{\nu}\ar[d]^{\zeta_2}&D^{e}
\ar[r]^{\mu}\ar[d]^{\zeta_1}&D^{
e}\ar[r]^{\mathrm{m}}\ar[d]^{\zeta_0}
&D\ar[r]\ar[d]^{id}& 0\\
\cdots\ar[r]^{\mathrm{b}'}&D^{\ot 5}\ar[r]^{\mathrm{b}'}&D^{\ot
4}\ar[r]^{\mathrm{b}'}&D^{\ot 3}\ar[r]^{\mathrm{b}'}&D^{\ot
2}\ar[r]^{\mathrm{b}'}&D\ar[r]& 0}
$$
where $\zeta_n: D^e\ra D^{\ot (n+2)}$ and
\begin{align*}
&\zeta_0=id,\\
&\zeta_{n}(s\ot s)=s^{\ot (n+2)},~\zeta_n(1\ot 1)=-1\ot s^{\ot n}\ot
1,\\
&\zeta_{n}(1\ot s)=(-1)^{n}\ot s^{\ot (n+1)},~\zeta_n(s\ot 1)=(-1)^n
s^{\ot (n+1)}\ot 1.
\end{align*}
Hence,
$$E^1_{p,q}=\Ho_q\big(D,C_{p}({}_D^{~~\natural}T)\big)\cong
\begin{cases}T^{\ot (p+1)}\ot 1\oplus T^{\ot(p+1)}_{ev}\ot s &\text{ for }q=0\\
T^{\ot (p+1)}_{ev}\ot 1\ot s^{\ot (2n-1)}\oplus T^{\ot(p+1)}_{od}\ot
s^{\ot 2n}&\text{ for }q=2n-1>0\\T^{\ot (p+1)}_{od}\ot 1\ot s^{\ot
2n}\oplus T^{\ot(p+1)}_{ev}\ot s^{\ot (2n+1)}&\text{ for
}q=2n>0\end{cases}.$$ The cyclic operator $\tau$ on
$\Ho_q\big(D,C_{p}({}_D^{~~\natural}T)\big)$ is defined via
$$\tau(t^{r_0},\ldots,t^{r_p}\mid s^{l}\mid \underbrace{s,\ldots,s}_{n \text{ times}})
=(-1)^{(l+n)r_p}(t^{r_p},t^{r_0},\ldots,t^{r_{p-1}}\mid s^{l}\mid
\underbrace{s,\ldots,s}_{n \text{ times}}), \text{ where }l=0,1.$$

We can describe the cyclic modules $E^1_{\bullet,q}$ simply. Let
$C_{\bullet}(T)$ be the cyclic module of the algebra $T$. Since the
face maps, degeneracy maps, and the cyclic operators of
$C_{\bullet}(T)$ do not change the total degree of $t$,
$C_{\bullet}(T)$ can be decomposed into the direct sum of two
sub-cyclic modules $C_{\bullet}(T)_{ev}$ and $C_{\bullet}(T)_{od}$
with $C_{p}(T)_{ev}=T^{\ot(p+1)}_{ev}$ and
$C_{p}(T)_{od}=T^{\ot(p+1)}_{od}$. Let $C_{\bullet}'(T)_{ev}$ be the
cyclic module with $C_{n}'(T)_{ev}=T^{\ot(n+1)}_{ev}$ and the
operators
\begin{align*}
\tau(t^{r_0},t^{r_1},\ldots,t^{r_n})
&=(-1)^{r_n}(t^{r_n},t^{r_0},\ldots,t^{r_{n-1}}),\\
\partial_i(t^{r_0},t^{r_1},\ldots,t^{r_n})
&=(t^{r_0},\ldots,t^{r_i+r_{i+1}},\ldots,t^{r_n}),\quad \text{for }0\leq i<n,\\
\partial_n(t^{r_0},t^{r_1},\ldots,t^{r_n})
&=(-1)^{r_n}(t^{r_0+r_n},t^{r_1},\ldots,t^{r_{n-1}}),\\
\sigma_j(t^{r_0},t^{r_1},\ldots,t^{r_n})&=(t^{r_0},\ldots,t^{r_j},1,t^{r_{j+1}},
\ldots,t^{r_n}),\quad
\text{for }0\leq j\leq n,
\end{align*}
where $r_0+\cdots+r_n$ is an even integer.

\begin{coro} $E^1_{\bullet,0}$ is identified with $C_{\bullet}(T)\oplus
C_{\bullet}'(T)_{ev}$ as cyclic modules; for $n>0$,
$E^1_{\bullet,n}$ is identified with $C_{\bullet}(T)_{od}\oplus
C_{\bullet}'(T)_{ev}$ as cyclic modules.
\end{coro}

\begin{lemma}The Hochschild homology of $(C_{\bullet}'(T)_{ev},\partial)$
is $0$.
\end{lemma}
\begin{proof}
Indeed, we can construct a chain contraction
$\{h_n:C_{n}'(T)_{ev}\ra C_{n+1}'(T)_{ev}\}$ of the identity chain
map. That is,
 $$h_n(t^{r_0},t^{r_1},\ldots,t^{r_n})
 =\frac{1}{2}\sum_{i=0}^{n}(-1)^i
 (t^{r_0-1},t^{r_1},\ldots,t^{r_i},t,
 t^{r_{i+1}},\ldots,t^{r_n}).$$
 We can check directly that
  $\partial h_n+h_{n-1}\partial=id$. Hence
  $\Ho_n(C_{\bullet}'(T)_{ev},\partial)=0$, for all $n\geq 0$.
\end{proof}

Since $\Hh_n(T)\cong
\begin{cases}T&\text{ for }n=0,1\\0&\text{ for }n\geq 2\end{cases}$
and $\Hh_n(T)_{od}\cong
\begin{cases}T_{od}&\text{ for }n=0,1\\0&\text{ for }n\geq
2\end{cases}$, we obtain that
$$E^2_{p,q}=0,\forall~ p\neq 0,1;$$
$$E^2_{0,0}=E^2_{1,0}=T\quad \text{and}\quad E^2_{0,q}=E^2_{1,q}=T_{od},\text{  for }q>0.$$
So the spectral sequence collapses at $E^2$, and
$\Hh_n(\P)=\bigoplus_{p+q=n}E^{2}_{p,q}$.
\begin{coro}
The Hochschild homology of $\P$ is $$\Hh_n(\P)\cong
\begin{cases}T&\text{ for
}n=0\\T\oplus T_{od}&\text{ for }n=1\\T_{od}\oplus T_{od}&\text{ for
}n>1\end{cases}.$$
\end{coro}

The cyclic homology of $T$ is well-known (see e.g., p.337 in
\cite{W})
$$\Hc_n(T)\cong\begin{cases}T &\text{ for }n=0\\k&\text{ for }n>0 \end{cases}.$$
Thanks to the short exact sequences $0\ra
\overline{\Hc}_{n-1}(\P)\ra \overline{\Hh}_n(\P)\ra
\overline{\Hc}_{n}(\P)\ra 0$ (see e.g., \cite{L} Theorem 4.1.13),
where $\overline{\Hh}_n(\P):=\Hh_n(\P)/\Hh_n(T)$ and
$\overline{\Hc}_n(\P):=\Hc_n(\P)/\Hc_n(T)$, we get the cyclic
homology of $\P$.
\begin{prop}
$$\Hc_n(\P)\cong\begin{cases}T&\text{ for }n=0\\  T_{od}\oplus k&
\text{ for }n>0\end{cases}.$$
\end{prop}

\begin{remark} The Laurent polynomial ring $T$ is isomorphic to  the group
algebra $k[\mathbb{Z}]$. If making use of the results of \cite{GJ},
one can construct another spectral sequence
$\widetilde{E}^{r}_{p,q}$ with $\widetilde{E}^{2}_{p,q}=0$ for
$\forall q\neq 0,1$, converging to the cyclic homology of $\P$. In
this way, it remains to determine
$d^2:\widetilde{E}^{2}_{p+2,0}\longrightarrow
\widetilde{E}^{2}_{p,1}$ to achieve $\widetilde{E}^3$. Since this
spectral sequence collapses at $\widetilde{E}^3$, one then does
more.
\end{remark}

\vspace{1em}

\textbf{Acknowledgements.} The authors are supported in part by the
NNSF (Grants: 10971065, 10728102), the PCSIRT and the RFDP from the
MOE, the National and Shanghai Leading Academic Discipline Projects
(Project Number: B407). The first author would like to express her
gratitude to Professor Ezra Getzler for pointing out the flatness
condition. She is  indebted to her advisor Professor Marc Rosso for
his kind help. The authors also would like to thank Professor
Joachim Cuntz for his useful comments and encouragement.

\footnotesize{{\sc Jiao ZHANG \\Department of Mathematics, East
China Normal University,\\ 500 Dongchuan Road, Min Hang, 200241,
Shanghai,   P. R. China.\\\&\\Institut de Math\'{e}matiques de
Jussieu, Universit\'{e} Paris
Diderot--Paris VII,\\175 Rue du Chevaleret, 75013, Paris, France.\\
 E-mail:}
zhangjiao@math.jussieu.fr}

\footnotesize{{\sc Naihong HU \\Department of Mathematics, East
China Normal University,
\\ 500 Dongchuan Road, Min Hang, 200241, Shanghai,   P. R. China.\\
 E-mail:}
 nhhu@math.ecnu.edu.cn}
\vspace{2em}

\end{document}